\documentclass[12pt, letterpaper]{amsart}

\newcommand{\la}{\langle}
\newcommand{\ra}{\rangle}

\usepackage{amssymb}
\usepackage{amsthm}
\usepackage{amsxtra}
\usepackage{graphicx}
\usepackage{color}
\usepackage{xcolor}

\newtheorem{theorem}{Theorem}

\newtheorem{proposition}[theorem]{Proposition}
\newtheorem{lemma}[theorem]{Lemma}
\newtheorem{corollary}[theorem]{Corollary}

\theoremstyle{remark}

\numberwithin{equation}{section}

\numberwithin{theorem}{section}

\numberwithin{table}{section}

\numberwithin{figure}{section}

\ifx\pdfoutput\undefined
  \DeclareGraphicsExtensions{.pstex, .eps}
\else
  \ifx\pdfoutput\relax
    \DeclareGraphicsExtensions{.pstex, .eps}
  \else
    \ifnum\pdfoutput>0
      \DeclareGraphicsExtensions{.pdf}
    \else
      \DeclareGraphicsExtensions{.pstex, .eps}
    \fi
  \fi
\fi

\title[Linear Stability Analysis of the RVM in an Axisymmetric Domain]{Linear Stability Analysis of the Relativistic Vlasov-Maxwell System in an Axisymmetric Domain}

\author{Katherine Zhiyuan Zhang}
\address{Brown University}

\begin{document}

\maketitle

\begin{abstract}
We consider the plasma confined in a general axisymmetric spatial domain with perfect conducting boundary which reflects particles specularly, and look at a certain class of equilibria, assuming axisymmetry in the problem. We prove a sharp criterion of spectral stability under these settings. Moreover, we provide several explicit families of stable/unstable equilibria using this criterion. 
\end{abstract}

\section{Introduction}

In plasma theory, an important goal is to study the stability properties of plasmas. The study of the stability properties of macroscopic systems like MHD and other fluid-like models has been carried out a lot (for example \cite{Friedberg1}, \cite{Nicholson1}). However, many plasma phenomena are microscopic so one must consider kinetic models, including the (relativistic and nonrelativistic) Vlasov-Maxwell system, Vlasov-Poisson system, Boltzmann equation, etc. (see \cite{Friedberg1}, \cite{Nicholson1}).

When the temperature is high or the density is low, the effect of collisions becomes minor compared to the effect of the electromagnetic forces. Such plasmas are modeled by the relativistic Vlasov-Maxwell system (RVM). The stability of the RVM has been studied a lot in the physics literature. The simplest case is a spatially homogeneous equilibrium with vanishing magnetic fields (for example \cite{G1}, \cite{GS3}). One of the most important results is Penrose's sharp criterion on linear stability for a spatially homogeneous equilibrium of the Vlasov-Poisson system (\cite{Penrose1}). In \cite{LS1}, \cite{LS3} and \cite{LS2}, the analysis of a spatially inhomogeneous equilibrium was carried out in domains without any spatial boundaries (i.e. whole space or periodic setting). A sharp criterion for spectral stability was given in \cite{LS2}, with some families of stable and unstable examples provided. The question of nonlinear stability is much more difficult, see, for example, \cite{LS3}.

On the other hand, in many real world applications, the plasma is confined to a bounded region. A typical example is the tokamak, which is one of the main foci of research in fusion energy. Therefore, an important topic is to understand the stability properties of a confined plasma. In \cite{HK1}, the confinement of a tokamak plasma is discussed using some fluid models and the role of different parts of the boundary are explored. For the microscopic model RVM, there are very few rigorous studies in bounded domains. A key paper in this direction is \cite{NS1}, in which the authors considered the case when the spatial domain is a solid torus (like a tokamak), and toroidal symmetry is assumed. A sharp criterion of spectral stability is obtained, thus reducing the problem of determining the linear stability to the positivity of a simpler self-adjoint operator $\mathcal{L}^0$.

However, there are other domains worthwhile studying. We want to investigate how the geometric structure of the domain influences the stability of the plasma. In this paper, we consider the RVM on a general axisymmetric spatial domain. We consider a certain class of equilibria, assuming axisymmetry in the problem. It is surprising that a sharp criterion of spectral stability can be proven, not just for a torus, but for any axisymmetric domain. In \cite{NS1} the domain $\Omega$ is exactly a torus and the authors used toroidal coordinates. Here in this paper we show that $\Omega$ can be any axisymmetric domain using cylindrical coordinates $(r, \varphi, z)$. An example here is the case when $\Omega$ is a solid ball.

A major difficulty in our setting is that $\Omega$ might include part of the $z$-axis, which creates a singularity in the problem. Namely, an operator $-\Delta + \frac{1}{r^2}$ appears in the analysis, which gives a singularity at $r=0$. This difficulty is surmounted by a trick applied in \cite{LS1}. The main point is to use the identity $-\Delta (g e^{i\varphi}) = (- \Delta + \frac{1}{r^2} ) g e^{i \varphi} $, which is valid for any $\varphi$-independent function $g$. We consider perturbations that are independent of $\varphi$, for which the operator $-\Delta + \frac{1}{r^2}$ acts nicely. This permits the use of a convenient Hilbert space for the axisymmetric functions, see \eqref{Hkdaggerdef}. It also allows us to use the standard elliptic theory in Lemma \ref{regularity}.

We prove that the linear stability of the equilibrium is equivalent to the positivity of a certain self-adjoint operator $\mathcal{L}^0$ (see \eqref{L0definition}), which acts only on scalar functions (see Theorem \ref{mainresult}). Moreover, we use $\mathcal{L}^0$ to provide several explicit examples of stable and unstable equilibria, as summarized in Theorem \ref{mainresultexample}. We give explicit inequalities that determine the stability. They contain information on the domain and hence enable a partial analysis of the effect of the geometry, see Theorem \ref{stableexample} and Corollary \ref{stableexampleshape}. For example, a thin torus with large major radius tends to favor instability for the equilibria rather than one with smaller radius. This is the first result on this question for the RVM model, and shows plenty of possiblities to carry out deeper investigations. On the other hand, we obtain instability for a family of equilibria that depend strongly on the angular momentum. In contrast to Section 5.2 in \cite{NS1}, we allow the equilibria to have electric as well as magnetic potentials, and therefore we are able to prove instability for a larger family of equilibria, see Proposition \ref{unstableexample2}. In addition, we show that under some constraint on the shape of the domain and some smallness assumption on the steady density distribution, such unstable equilibria can be constructed explicitly (see Theorem \ref{unstableexampleexistence}). 

The system RVM is
\begin{equation}
\partial_t f^{\pm} + \hat{v} \cdot \nabla_x f^{\pm} \pm (\textbf{E} + \hat{v} \times \textbf{B} ) \cdot \nabla_v f^{\pm} =0 \ ,
\end{equation}
\begin{equation}
\nabla_x \cdot \textbf{E} = \rho = \int_{\mathbb{R}^3} (f^+ - f^-)dv, \  \nabla_x \cdot \textbf{B} =0 \ ,
\end{equation}
\begin{equation}
\partial_t \textbf{E} -\nabla_x \times \textbf{B} = -\textbf{j} = -  \int_{\mathbb{R}^3} \hat{v} (f^+ - f^-)dv, \  \partial_t \textbf{B} + \nabla_x \times \textbf{E} =0
\end{equation}

In this system. $f^\pm (t, x, v) \geq 0$ is the density distribution of ions ($+$) and electrons ($-$). We confine the plasma in a region $ \Omega \subset \mathbb{R}^3$, so that $x \in \Omega$ is the particle position. $v \in \mathbb{R}^3$ is the particle momentum, $\la v\ra = \sqrt{1+ v^2}$ is the particle energy, and $\hat{v} = v/ \la v \ra$ is the particle velocity. Also, $\textbf{E}$ is the electric field, $\textbf{B}$ is the magnetic field, and therefore $ \pm (\textbf{E} + \hat{v} \times \textbf{B} ) $ is the electromagnetic force. Moreover, the charge density $\rho$ and the current density $\textbf{j}$ are defined as
\begin{equation}
\rho = \int_{\mathbb{R}^3} (f^+ - f^-) dv , \qquad \textbf{j} =  \int_{\mathbb{R}^3} \hat{v} (f^+ - f^-) dv  \ .
\end{equation}
At the boundary we impose the specular condition (which means that $f^\pm$ is even with respect to  $v_n= v \cdot e_n $ on $\partial \Omega$, with $e_n$ being the outward normal vector of $\partial \Omega$ at $x$): 
\begin{equation}
f^{\pm} (t,x,v) = f^{\pm} (t,x,v-2(v \cdot e_n (x))e_n(x) ), e_n(x) \cdot v <0,  \  \forall x \in \partial \Omega \ ,
\end{equation}
as well as the perfect conductor boundary condition
\begin{equation} 
\textbf{E}(t,x) \times n(x) =0,  \ \textbf{B} (t,x) \cdot n(x) =0, \ \forall x \in \partial\Omega \ .
\end{equation}
The system RVM with these boundary conditions enjoys the conservation of the total energy
\begin{equation}
\mathcal{E} (t) = \int_\Omega \int_{\mathbb{R}^3} \la v \ra (f^+ + f^-) dv dx + \frac{1}{2} \int_\Omega (|\textbf{E}|^2  + |\textbf{B}|^2 )  dx  \ .
\end{equation}

The contents in the paper are arranged as follows. It is natural to use cylindrical coordinates here. In Section 2, we set up the problem in cylindrical coordinates, including the coordinates and the symmetry assumptions. Section 3 gives the description of the particle trajectories and the family of equilibria we consider in this paper. We defined functional spaces for the functios to lie in. The difficulty caused by the singularity at $r=0$ is avoided by using the space $H^{k \dagger}$ (see \eqref{Hkdaggerdef}). Also, we linearize around the equilibria, defined the key operators $\mathcal{P}^\pm$, $\mathcal{A}^0_1$, $\mathcal{A}^0_2$, $\mathcal{B}^0$ (see \eqref{A01definition}, \eqref{A02definition}, \eqref{B0definition}), and give a precise statement of the main results. Section 4 is devoted to the description of the boundary conditions on the linearized problem. These conditions are written in terms of the electric and magnetic potentials, see \eqref{fullboundarycondition}. For the first main result (Theorem \ref{mainresult}), the proof of the stability part is given in Section 5 using linearized energy invariants and Casimir type invariants of the linearized RVM system (see \eqref{invarianceI} and \eqref{invarianceKg}). A rather delicate minimization (see Lemma \ref{minimization1} and Lemma \ref{minimization2}) leads to stability provided $\mathcal{L}^0 \geq 0$. For the proof of the instability part, which is given in Section 6, we express $f^\pm$ in terms of the electric and magnetic potentials by integrating the Vlasov equation along the particle trajectories, then plugging it into the linearized Maxwell system to obtain a matrix equation on the potentials (see \eqref{MaxwellMatrixEqn}). This equation involve several linear operators $\mathcal{Q}^\pm_\lambda$, $\mathcal{A}^\lambda_1$, $\mathcal{A}^\lambda_2$, $\mathcal{B}^\lambda$. This enables us to obtain growing modes by a continuation argument on the corresponding self-adjoint operators. A key step here is to investigate the limit of these operators when $\lambda \rightarrow 0$ and $\lambda \rightarrow + \infty$. Section 7 is devoted to some analysis on the operator $\mathcal{L}^0$ that determines the stability of some equilibria (Theorem \ref{stableexample}), as well as an example analyzing the effect of the geometry (Corollary \ref{stableexampleshape}). At last, in Section 8, we derive an explicit sufficient condition for instability by a more detailed study on $\mathcal{L}^0$ (Proposition \ref{unstableexample2}). We use some scaling techniques to enhance the dependence of the equilibrium on the angular momentum, which is the key to the instability. We also construct a family of unstable equilibria using this condition by solving the coupled elliptic system satisfied by the electric and magnetic potentials via a fixed point argument (Theorem \ref{unstableexampleexistence}).

\section{Coordinates and Symmetry}

We deal with the equation in the language of cylindrical coordinates $(r, \varphi , z )$, and consider the plasma constrained inside a region $\Omega  $, which is a $C^1$ axisymmetric (with respect to the $z$-axis) domain in $\mathbb{R}^3$, i.e. rotational invariant around the $z$-axis. $\Omega$ can be viewed as a solid of revolution (see the picture below for an example), determined as follows: Consider a counterclockwisely parametrized closed $C^1$ curve $\mathcal{C}$ in the plane $\{  \varphi = 0\}$, where $\beta$ is the arclength parameter: (For simplicity, we normalize the total arclength of $\mathcal{C}$ to be $1$. ) 

\vskip 0.75cm 

\begin{center}
\includegraphics[width=0.45\textwidth]{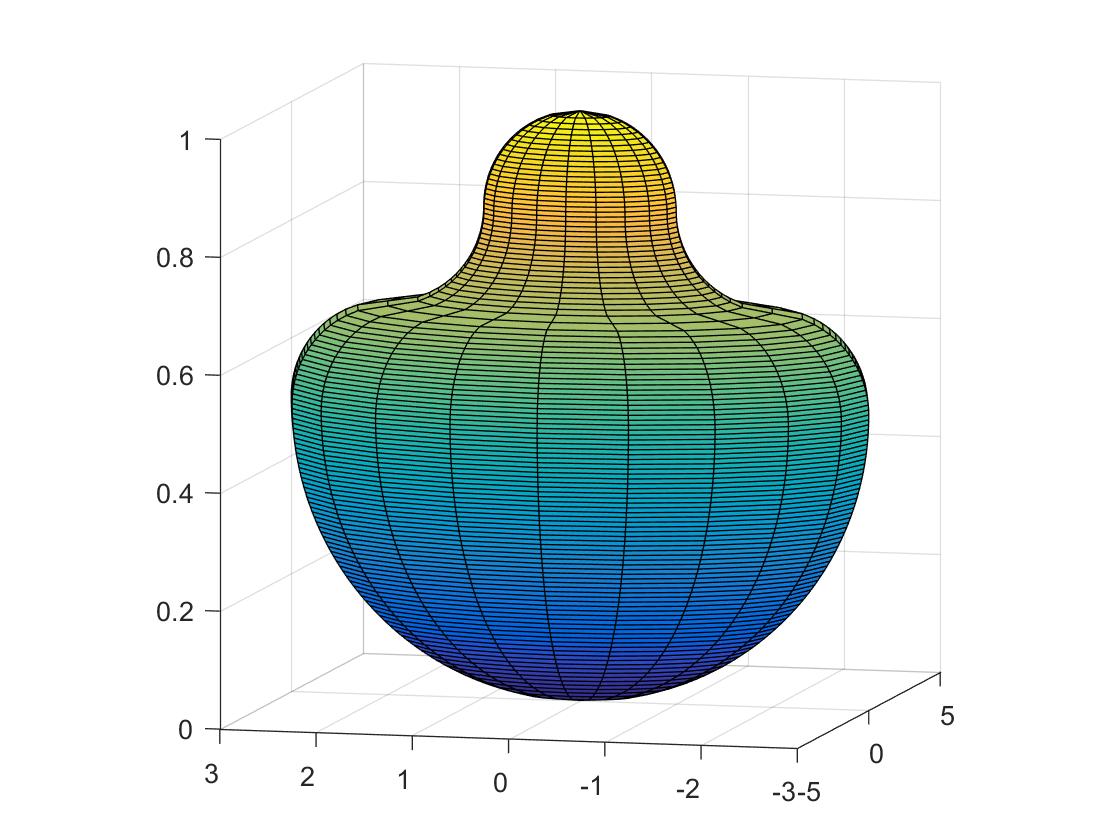}
\end{center}

\begin{equation} 
r = \tilde{r}(\beta), \quad z = \tilde{z} (\beta ), \quad \beta \in [0,1], \quad \tilde{r} \geq 0  \ for \ all \ \beta \in (0, 1) , \quad \tilde{r}, \tilde{\beta} \in C^1 (0, 1)  \ ,
\end{equation}

with either
\begin{equation}
\tilde{r} (0) = \tilde{r} (1) > 0 ,  \quad \tilde{z} (0) = \tilde{z} (1) \ ,
\end{equation}
or, 
\begin{equation} \label{curveconstraint}
\tilde{r}(0) =\tilde{r} (1) =0  ,  \quad   \tilde{z}' / \tilde{r}' =0 \ at \ \beta = 0, 1  \ .
\end{equation}
Let $\partial \Omega$ be the surface obtained by rotating $\mathcal{C}$ around the $z$-axis. Then the boundary $\partial \Omega$ is $C^1$ smooth. The last condition in \eqref{curveconstraint} implies that $\partial \Omega$ is $C^1$ smooth even if it touches the $z$-axis.

Let $e_r$, $e_\varphi$ and $e_z$ be the unit vectors in the cylindrical coordinate system (see Appendix A), and $e_n$, $e_{tg}$ to be the unit vectors in outward normal direction and tangential direction orthogonal to $e_\varphi$ on $\partial \Omega$, respectively. Then on $\partial \Omega$, the outward unit normal vector $e_n (x)  = ( \tilde{z}' e_r - \tilde{r}' e_z)/\sqrt{\tilde{z}'^2+\tilde{r}'^2} $, and $e_{tg} (x) =  ( - \tilde{r}'  e_r - \tilde{z}' e_z)/\sqrt{\tilde{z}'^2+\tilde{r}'^2} $.

On the boundary, we assume the specular condition on the density function $f^\pm$. This means that $f^\pm$ is even with respect to  $v_n= v \cdot e_n/|e_n|$ on $\partial \Omega$, i.e.
\begin{equation}
f^{\pm} (t,x,v) = f^{\pm} (t,x,v-2(v \cdot e_n (x)) e_n (x) ), \  e_n(x) \cdot v <0,  \forall x \in \partial \Omega, \forall v \in \mathbb{R}^3  
\end{equation}
as well as the perfect conductor boundary condition on the electric and magnetic fields:
\begin{equation} \label{perfectconductorboundary}
\textbf{E}(t,x) \times e_n(x) =0, \  \textbf{B} (t,x) \cdot e_n(x) =0, \forall x \in \partial\Omega \ .
\end{equation}
We also introduce the electric potential $\phi$ and the magnetic potential $\textbf{A}$:
\begin{equation} \label{potentialdefinition}
\textbf{E} = -\nabla \phi -\partial_t \textbf{A}, \  \textbf{B} = \nabla \times \textbf{A}
\end{equation} 
and impose the Coulomb gauge
\begin{equation}
\nabla \cdot \textbf{A} =0 \ . 
\end{equation}

\textit{Remark} The choices of $\phi$ and $\textbf{A}$ are not unique. Actually the choice of $\textbf{A}$ can differ by the gradient of a harmonic function.

For any vector-valued function $\textbf{g}$, we denote $\tilde{\textbf{g}} = g_r e_r + g_z e_z$. The Maxwell system then becomes

\begin{equation} \label{system01}
-\Delta \phi = \rho  \ , 
\end{equation}
\begin{equation} \label{system02}
(\partial_t^2 -\Delta + \frac{1}{r^2}) A_\varphi  = j_\varphi  \ , 
\end{equation}
\begin{equation} \label{system03}
(\partial_t^2 -\Delta ) \tilde{\textbf{A}} + \partial_t \nabla \phi  = \tilde{\textbf{j}} \ .
\end{equation}

Using cylindrical coordinates $(r, \varphi , z)$, we make the axisymmetry assumption:
\begin{equation}  \label{varphiindependenceassumption}
\phi, \ A_r, \  A_z,  \ A_\varphi \ are \ independent  \ of \ \varphi.
\end{equation}
Therefore, $f^{\pm}$ does not depend \textit{explicitly} on $\varphi$, although it might depend on it \textit{implicitly} through the components of $v$.

\section{Equilibrium, Linearization and Main Result}

We consider an equilibrium such that $B^0_\varphi =0$, $\textbf{A}^0 = A^0_\varphi e_\varphi $. Then the equilibrium field is
\begin{equation}
\textbf{E}^0 = -\nabla \phi^0 = -\frac{\partial \phi^0}{\partial r} e_r    - \frac{\partial \phi^0}{ \partial z} e_z   \ , 
\end{equation}
\begin{equation}
\begin{split}
\textbf{B}^0 
& =  - \frac{\partial A^0_\varphi}{\partial z} e_r +  \frac{1}{r}  \frac{\partial ( r A^0_\varphi )}{\partial r} e_z   \ .  \\
\end{split}
\end{equation}
We define the particle trajectories as
\begin{equation}  \label{particletrajectoryODE}
\dot{X}^\pm = \hat{V}^\pm,  \  \dot{V}^\pm =  \pm \textbf{E}^0 (X^\pm) \pm \hat{V}^\pm \times \textbf{B}^0 (X^\pm)
\end{equation}
with initial values $(X^\pm (0;x,v), V^\pm (0;x,v) ) = (x,v) $. Each particle trajectory exists and preserves the axisymmetry up to the first time it meets the boundary. Let $s_0$ be a time when the trajectory $X^\pm (s_0 -; x,v)$ hits the boundary $\partial \Omega$. Recall that $v_n  = v \cdot e_n = \frac{\tilde{z}' v_r - \tilde{r}' v_z}{\sqrt{ \tilde{z}'^2 + \tilde{r}'^2 }}$, $v_{tg} = v \cdot e_{tg}  = \frac{- \tilde{z}' v_z - \tilde{r}' v_r}{\sqrt{\tilde{z}'^2 + \tilde{r}'^2}}$. For any given $(x, v)$ and $(X^\pm, V^\pm)$ with $x$ and $X^\pm$ on $\partial \Omega$, we re-decompose $v$ and $V^\pm$ into their $n$-component, $tg$-component and $\varphi$- component: $v =  v_n e_n + v_{tg} e_{tg} +  v_\varphi e_\varphi   $, $V^\pm =  V^\pm_n e_n + V^\pm_{tg} e_{tg} +  V^\pm_\varphi e_\varphi   $, and define 
\begin{equation}  \label{reflectedvelocity}
v_* = - v_n e_n + v_{tg} e_{tg} +v_\varphi e_\varphi \ , \ V^\pm_* = - V^\pm_n e_n + V^\pm_{tg} e_{tg} +  V^\pm_\varphi e_\varphi \ .
\end{equation}
Thus from the specular boundary condition, the trajectory can be continued by the rule
\begin{equation}   \label{particlereflection}
( X^\pm (s_0 + ; x,v), V^\pm (s_0 + ; x,v)  ) = ( X^\pm (s_0 - ; x,v), V_*^\pm (s_0 - ; x,v)  )  \ .
\end{equation}

Furthermore we assume the equilibrium has a particle density of the form $f^{0,\pm}  (x,v) = \mu^\pm (e^\pm (x,v), p^\pm (x,v)  )  $, where 
\begin{equation}  \label{epdefinition}
e^\pm (x,v) = \la v \ra \pm \phi^0 (r, z)  ,   \  p^\pm (x,v) = r  ( v_\varphi  \pm A^0_\varphi (r, z))  \ .
\end{equation}
$e^\pm$ and $p^\pm$ are invariant along the particle trajectories. (The proof of the invariance of $e^\pm$ and $p^\pm$ can be found in Appendix B.) We assume that $\mu^\pm (e^\pm,p^\pm)$ are non-negative $C^1$ functions which satisfy
\begin{equation} \label{decayassumption}
\begin{split}
\mu^\pm_e (e, p) <  0, \   |\mu^\pm_p (e , p )| +|\mu^\pm_e (e, p)| \leq \frac{C_\mu }{1+ |e|^\gamma}, \ \gamma > 3 .
\end{split}
\end{equation}

Denote the transport operator by
\begin{equation}
D^\pm = \hat{v} \cdot \nabla_x \pm (\textbf{E}^0  + \hat{v} \times \textbf{B}^0 ) \cdot \nabla_v  \ .
\end{equation}
Linearizing around the equilibrium $\mu^\pm$, we obtain the linearized Vlasov equation
\begin{equation} \label{linearizedvlasov}
\partial_t f^\pm + D^\pm f^\pm = \mp (\textbf{E} + \hat{v} \times\textbf{B} ) \cdot \nabla_v f^{0,\pm}  \ , 
\end{equation}
which can be written in cylindrical coordinates as
\begin{equation}
\partial_t f^\pm + D^\pm f^\pm =  \pm (  \mu^\pm_e D^\pm \phi +  \mu^\pm_e \hat{v} \cdot \partial_t \textbf{A} + r \mu^\pm_p \partial_t A_\varphi + \mu^\pm_p D^\pm (r A_\varphi)  )  \ .
\end{equation}
As we will mention in Section 4, in the linearized problem we can assume the homogeneous Dirichlet condition for $\phi$ and $A_\varphi$.

We use the letter $\tau$ to denote the axisymmetric constraint for functions. Denote $Y =L^2_{1/r^2 } (\Omega)$, i.e. the weighted-$L^2$ space with weight $1/r^2$. This weight gives some singularity at $r=0$ if $\Omega$ touches the $z$-axis. For this, we define  
\begin{equation} \label{Hkdaggerdef}
H^{k \dagger}  (\Omega) := \{  g \in L^{2, \tau} (\Omega) |  e^{i\varphi} g \in H^k (\Omega)   \}  (k=1,2)  \ . 
\end{equation}
Then $H^{2 \dagger} (\Omega) \subset Y$ because the identity 
\begin{equation} \label{Deltagexp}
-\Delta (g e^{i\varphi}) = (- \Delta + \frac{1}{r^2} ) g e^{i \varphi} 
\end{equation}  
holds for any $\varphi$-independent function $g$. Indeed, 
\begin{equation}
\begin{split}
-\Delta (g e^{i\varphi})
& = - \frac{1}{r} \frac{\partial}{\partial r} (r \frac{\partial g}{\partial r}) e^{i\varphi} - \frac{\partial^2 g}{\partial z^2} e^{i\varphi} - \frac{1}{r^2} \frac{\partial^2 (e^{i\varphi})}{\partial \varphi^2} g \\ 
& = - \frac{1}{r} \frac{\partial}{\partial r} (r \frac{\partial g}{\partial r}) e^{i\varphi} - \frac{\partial^2 g}{\partial z^2} e^{i\varphi} + \frac{1}{r^2} g e^{i\varphi} \\ 
& = (-\Delta + \frac{1}{r^2}) g e^{i\varphi} \ . \\
\end{split}
\end{equation}

We will show in Lemma \ref{regularity} that $A_\varphi$ is automatically in $H^{2 \dagger } (\Omega) $ once we assume $\textbf{E}$, $\textbf{B} \in L^2 (\Omega)$. Let $\mathcal{X} $ be the space consisting of all the scalar functions in $ \in H^{2, \tau} (\Omega) \cap H^{2 \dagger }  (\Omega)$ that satisfy the Dirichlet boundary condition. We define the weighted-$L^2$ spaces $\mathcal{H}^\pm$ as
$$\mathcal{H}^\pm : = L^2_{|\mu^\pm_e |} (\Omega \times \mathbb{R}^3) \ . $$
 
By saying that an equilibrium is spectrally unstable, we mean that the linearized system with the boundary conditions admits a \textit{growing mode}, which is defined to be a solution of the form $( e^{\lambda t} f^\pm  ,  e^{\lambda t}  \textbf{E} ,  e^{\lambda t} \textbf{B} )$ with $Re \lambda >0$, $f^\pm \in \mathcal{H}^\pm$, and $\textbf{E}$, $\textbf{B} \in L^{2, \tau} (\Omega ; \mathbb{R}^3) $.

Let $\mathcal{P}^\pm$ be the orthogonal projection on the kernel of $D^\pm$ in the space $\mathcal{H}^\pm $. Formally, we define
\begin{equation}  \label{A01definition}
\mathcal{A}^0_1 h = \Delta h + \sum_\pm \int_{\mathbb{R}^3} \mu^\pm_e (1- \mathcal{P}^\pm) h dv \ ,
\end{equation}
\begin{equation}  \label{A02definition}
\mathcal{A}^0_2 h = (-\Delta + \frac{1}{r^2  }) h - \sum_\pm \int_{\mathbb{R}^3}  \hat{v}_\varphi \big( \mu^\pm_p r   h   + \mu^\pm_e \mathcal{P}^\pm(\hat{v}_\varphi h ) \big) dv  \ , 
\end{equation}
\begin{equation}  \label{B0definition}
\mathcal{B}^0 h = - \sum_\pm \int_{\mathbb{R}^3} \hat{v}_\varphi \mu^\pm_e (1 - \mathcal{P}^\pm) h dv   \ ,
\end{equation}
These operators are naturally derived from integration of the Vlasov equation along the particle trajectories. Now we can state our first main result as follows.

\begin{theorem} \label{mainresult}
Let $(f^{0, \pm},\textbf{E}^0, \textbf{B}^0)$ be an equilibrium of the relativistic Vlasov-Maxwell system satisfying $ f^{0, \pm} (x, v) = \mu^\pm (e^\pm, p^\pm) \geq 0$ and $\mu^\pm_e (e, p) < 0$, $   |\mu^\pm_p| +|\mu^\pm_e| \leq \frac{C_{\mu} }{1+ |e|^\gamma}$ with $ \gamma> 3$, $\mu^\pm \in C^1$, $\phi^0 \in C (\bar{\Omega})$, $A^0_\varphi \in C (\bar{\Omega})$. Then the operator
\begin{equation} \label{L0definition}
\mathcal{L}^0 = \mathcal{A}^0_2 - \mathcal{B}^0 (\mathcal{A}^0_1)^{-1} (\mathcal{B}^0)^* 
\end{equation}
on $\mathcal{X}$ is self-adjoint. Also, we have

(i) If $\mathcal{L}^0 \geq 0$, there exists no growing mode of the linearized equation (\ref{linearizedvlasov}).

(ii) Any growing mode, if it exists, must be purely growing, i.e. the exponent $\lambda$ of instability must be a real number.
 
(iii) If $\mathcal{L}^0 \ngeq 0 $, there exists a growing mode of the linearized Vlasov equation (\ref{linearizedvlasov}) and the linearized Maxwell system with the boundary conditions.
\end{theorem}

In other words, the equilibrium is spectrally stable if and only if $\mathcal{L}^0 \geq 0$.

Our second main result provides some explicit examples of spectrally stable/unstable equilibria using the criterion provided in Theorem \ref{mainresult}. We give a summarized version of the theorem here, and the precise statements will come in later.

\begin{theorem} \label{mainresultexample}
Let $ (\mu^\pm, \textbf{E}^0, \textbf{B}^0)$ be an equilibrium, with $\mu^\pm $ satisfying the decaying assumption \eqref{decayassumptionexamplepart} (which is slightly stronger than (\ref{decayassumption})). 

(i) If $p \mu^\pm_p (e, p) \leq 0$, and both $|A^0_\varphi|$ and $|\phi^0|$ satisfy some smallness condition \eqref{stableexamplecondition01} or \eqref{stableexamplecondition02}, then $\mathcal{L}^0 \geq 0$ and hence the equilibrium is spectrally stable. The smallness condition on $|A^0_\varphi|$ and $|\phi^0|$ depends on the shape of the domain. (See Corollary \ref{stableexampleshape}.)
 
(ii) If $\sup_{x \in \Omega} r(x) >1$, and $\mu^\pm$ satisfies $ p \mu^\pm_p \geq C'_\mu |p| \la p  \ra^{-\epsilon} \nu (e) $ together with some function $\nu (e)$ which satisfies $\nu (e) \geq C_\nu \exp (-e)$ for some positive constants $C'_\mu$ and $C_\nu$, then we obtain a spectrally unstable equilibrium from $(\mu^\pm, \textbf{E}^0, \textbf{B}^0)$ by suitable scaling. In particular, in the case that $C_\mu$ satisfies some smallness condition, such unstable equilibrium is constructed explicitly.

\end{theorem}

\section{Boundary Conditions}

In this section, we derive from \eqref{perfectconductorboundary} the explicit boundary conditions in cylindrical coordinates for the linearized electric and magnetic fields.

From (\ref{potentialdefinition}), we have 
\begin{equation} \label{Erepresentation}
\begin{split}
\textbf{E}
&  = e_r (-\frac{\partial \phi}{\partial r} -  \frac{\partial A_r}{\partial t } ) + e_\varphi (-  \frac{\partial A_\varphi}{\partial t} ) + e_z ( - \frac{\partial \phi}{ \partial z} - \frac{\partial A_z}{\partial t}    )  \\
\end{split}
\end{equation}
\begin{equation}   \label{Brepresentation}
\begin{split}
\textbf{B}
&  = e_r ( - \frac{\partial A_\varphi}{\partial z} ) + e_\varphi (  \frac{\partial A_r}{\partial z} -  \frac{\partial A_z}{\partial r} ) + e_z \frac{1}{r}  \frac{\partial ( r A_\varphi )}{\partial r}  \ .   \\
\end{split}
\end{equation}
See Appendix A for detailed explanation to the coordinates. On $\partial \Omega$, (\ref{perfectconductorboundary}) becomes
\begin{equation}  \label{perfectconductorcoordinates1}
-  \tilde{r}' E_\varphi e_r +  ( \tilde{r}' E_r   + \tilde{z}' E_z)  e_\varphi - \tilde{z}' E_\varphi e_z   = 0, \   \forall x \in \partial \Omega
\end{equation}
and 
\begin{equation}    \label{perfectconductorcoordinates2}
\tilde{z}' B_r - \tilde{r}' B_z    =0, \  \forall x \in \partial \Omega \ ,
\end{equation}
where $'$ denotes the derivative with respect to $s$. Assuming that the boundary conditions are time-independent and plugging \eqref{Erepresentation} and \eqref{Brepresentation} into \eqref{perfectconductorcoordinates1} and \eqref{perfectconductorcoordinates2}, we can write \eqref{perfectconductorcoordinates1} and \eqref{perfectconductorcoordinates2} in terms of the potentials as
\begin{equation}
\begin{split}
& (\tilde{r}' \frac{\partial}{\partial r} + \tilde{z}' \frac{\partial}{\partial z}) \phi = 0  \ ,  \\
& (\tilde{r}' \frac{\partial}{\partial r} + \tilde{z}' \frac{\partial}{\partial z} +\tilde{r}'  \frac{1}{r}) A_\varphi = 0   \\
\end{split}
\end{equation}
for all $x \in \partial \Omega$. The first line together with $\partial \phi /\partial \varphi =0$ from \eqref{varphiindependenceassumption} implies that $\phi$ is a constant along the surface $\partial \Omega $. As for the second line, for each $x \in \partial \Omega$, let $\Gamma_x$ be the curve $\{ y \in \partial \Omega : \varphi (y) = \varphi (x) \}$. Then the second line becomes
$$
\frac{\partial}{\partial s}  \big( \exp (\int^s_1 \frac{\tilde{r}' (\sigma)}{\tilde{r} (\sigma)} d\sigma ) A_\varphi  \big) =0 \ ,
$$
after being multiplied by $\exp (\int^s_1 \frac{\tilde{r}' (\sigma)}{\tilde{r} (\sigma)} d\sigma )$. This implies
\begin{equation}
A_\varphi = C \exp \big(- \int^s_1 \frac{\tilde{r}' (\sigma)}{\tilde{r} (\sigma)} d\sigma \big) = \frac{C}{\tilde{r} (\beta)} \ .
\end{equation}
Since the choice of $\textbf{A}$ can differ by the gradient of a harmonic function, we can remove the singularity in the boundary condition on $\textbf{A}$ by adding to it the gradient of a harmonic function with the same boundary condition. The linearized problem must then satisfy 
\begin{equation}
A_\varphi =0 
\end{equation}
on $\partial \Omega$.

The Coulomb gauge becomes
\begin{equation}
\frac{1}{r} \frac{\partial (r A_r)}{\partial r} + \frac{1}{r} \frac{\partial A_\varphi}{\partial \varphi} + \frac{\partial A_z}{\partial z} = \frac{1}{r} \frac{\partial (r A_r)}{\partial r}  + \frac{\partial A_z}{\partial z} = 0 \ . 
\end{equation}
On the boundary, we use subscripts $n$ and $tg$ to denote the normal and tangential components, respectively. Then
\begin{equation}
\partial_z = \frac{ - \tilde{z}'  \partial_{tg} - \tilde{r}' \partial_n}{\sqrt{\tilde{z}'^2 + \tilde{r}'^2 }}, \qquad  \partial_r = \frac{- \tilde{r}'  \partial_{tg} +  \tilde{z}' \partial_n}{\sqrt{ \tilde{z}'^2 + \tilde{r}'^2}} 
\end{equation}
and 
\begin{equation}
A_z = \frac{- \tilde{z}' A_{tg} - \tilde{r}' A_n}{\sqrt{\tilde{z}'^2 + \tilde{r}'^2}}, \qquad  A_r = \frac{- \tilde{r}' A_{tg} + \tilde{z}' A_n}{\sqrt{\tilde{z}'^2 + \tilde{r}'^2}}  \ .
\end{equation}
We rewrite the Coulomb gauge in terms of the normal and tangential components on $\partial \Omega$ as
\begin{equation}
( \partial_n A_n + \frac{\tilde{z}'}{\tilde{r} \sqrt{\tilde{z}'^2 + \tilde{r}'^2} } A_n  ) + ( \partial_{tg} A_{tg} - \frac{\tilde{r}'}{\tilde{r} \sqrt{\tilde{z}'^2 + \tilde{r}'^2} } A_{tg} )  = 0 \ . 
\end{equation}
Thus we assume
\begin{equation} \label{CoulombgaugeAn}
\partial_n A_n + \frac{\tilde{z}'}{\tilde{r} \sqrt{\tilde{z}'^2 + \tilde{r}'^2} } A_n  = 0  
\end{equation}
and   \label{CoulombgaugeAtg}
\begin{equation}
\partial_{tg} A_{tg} - \frac{\tilde{r}'}{\tilde{r} \sqrt{\tilde{z}'^2 + \tilde{r}'^2} } A_{tg}   = 0  \ . 
\end{equation}
Again, for each $x \in \partial \Omega$, let $\Gamma_x$ be the curve $\varphi = \varphi (x) = const.$ and $s$ be its arc length parameter. We obtain 
$$
A_{tg} = C \exp \big(   \int^s_1 \frac{\tilde{r}' (\sigma)}{\tilde{r} (\sigma) \sqrt{   \tilde{z}'^2 (\sigma) +  \tilde{r}'^2 (\sigma)  }  }  d \sigma   \big)  \ . 
$$
Thus the linearized problem must satisfy
\begin{equation}
A_{tg} =0 
\end{equation}
on $\partial \Omega$.
The condition (\ref{CoulombgaugeAn}) is equivalent to 
\begin{equation}
\tilde{z}' \partial_r A_n - \tilde{r}' \partial_z A_n + \frac{\tilde{z}'}{\tilde{r}} A_n =0 \ . 
\end{equation}

To summarize, we consider the linearized problem with the following boundary condition on $\phi$ and $\textbf{A}$:
\begin{equation}  \label{fullboundarycondition}
\phi=0, \qquad  A_\varphi =0, \qquad   A_{tg} =0, \qquad  \tilde{z}' \partial_r A_n - \tilde{r}' \partial_z A_n + \frac{\tilde{z}'}{\tilde{r}} A_n =0, \qquad x \in \partial \Omega \ . 
\end{equation}

\section{Linear Stability}

In this section, we prove the first part of Theorem \ref{mainresult}. Firstly we introduce some property of $D^\pm$. 
\begin{lemma}
Let $g(x,v)  = g(r, z, v_r, v_\varphi, v_z  ) $ be a $ C^1$ axisymmetric function on $\bar{\Omega} \times \mathbb{R}^3$, then $g$ satisfies the specular boundary condition iff 
\begin{equation}
\int_\Omega \int_{\mathbb{R}^3} g D^\pm h dv dx =  - \int_\Omega \int_{\mathbb{R}^3} h D^\pm g dv dx
\end{equation}
holds for all $ C^1$ axisymmetric function $h$ with $v$-compact support that satisfy the specular boundary condition (denote this set of functions by $\mathcal{C}$).
\end{lemma}

\begin{proof}
Integrating by parts in $x$ and $v$, we obtain
$$
\int_\Omega \int_{\mathbb{R}^3} (g D^\pm h + h D^\pm g )  dv dx = 2 \pi \int^1_0 \int_{\mathbb{R}^3} r gh \hat{v} \cdot e_n dv dz |_{r = \tilde{r} (\beta), z=\tilde{z} (\beta)}  \ .
$$
If $g$ satisfies the specular boundary condition, then both $g$ and $h$ are even functions of $v_n= v \cdot n/|n|$ on $\partial \Omega$, and therefore the right side of the equality vanishes. 
Conversely, if the right side vanishes for all $h$, then 
$$
\int^1_0 \int_{\mathbb{R}^3} g k (r, z, v_r,  v_\varphi, v_z ) dv dz |_{r = \tilde{r} (\beta), z=\tilde{z} (\beta)} =0
$$
for all test functions $k$ which are odd in $v_n = v \cdot e_n $. So $g( r, z, v_r, v_z, v_\varphi ) |_{r = \tilde{r} (\beta), z=\tilde{z} (\beta)}$ is even in $v_n$, which gives the specular boundary condition. 
\end{proof}

Define
\begin{equation}
dom(D^\pm) =\{ g\in \mathcal{H}^\pm | D^\pm g \in \mathcal{H}^\pm, \la D^\pm g, h \ra_{\mathcal{H}^\pm} = - \la g , D^\pm h \ra_{\mathcal{H}^\pm} , \forall h \in \mathcal{C}  \} \ .
\end{equation}
Then $dom(D^\pm)$ are dense in $\mathcal{H}^\pm$, $D^\pm$ are skew-adjoint on $\mathcal{H}^\pm$.

Denote $F^\pm =  f \mp r  \mu^\pm_p A_\varphi $. Then $F^\pm$ satisfies

\begin{equation}
(\partial_t + D^\pm) F^\pm = \mp \mu^\pm_e \hat{v}\cdot \textbf{E} \ .
\end{equation}

Next we prove that the functions and fields are in the right spaces that we expect them to live in.

\begin{lemma} \label{regularity}
Let $( e^{\lambda t} f^\pm , e^{\lambda t} \textbf{E}, e^{\lambda t} \textbf{B} ) $ be a growing mode (then by definition $f^\pm \in \mathcal{H}^\pm$, $\textbf{E}, \textbf{B} \in L^2 (\Omega;\mathbb{R}^3)$) and define $F^\pm =  f \mp r  \mu^\pm_p A_\varphi$. Then $\textbf{E}, \textbf{B} \in H^1 (\Omega;\mathbb{R}^3) $, 
\begin{equation}
\int_\Omega \int_{\mathbb{R}^3} \frac{1}{|\mu^\pm_e|} ( |F^\pm|^2 +|D^\pm F^\pm|^2 ) dv dx < +\infty  \ .
\end{equation}
Moreover, we have $A_\varphi  \in  H^{2 \dagger}  (\Omega) \subset  L^2_{1/r^2 } (\Omega)$. 
\end{lemma}

\begin{proof}
Plug in $( e^{\lambda t} f^\pm , e^{\lambda t} \textbf{E}, e^{\lambda t} \textbf{B} ) $, we obtain
\begin{equation}
(\lambda + D^\pm) F^\pm = \mp \mu^\pm_e \hat{v}\cdot \textbf{E} \ .
\end{equation}
Since $f^\pm \in \mathcal{H}^\pm $, $A_\varphi \in L^2 (\Omega)$, $\sup_x \int_{\mathbb{R}^3} |\mu^\pm_p| dv < + \infty$, and $\hat{v} \cdot \textbf{E} \in \mathcal{H}^\pm $ (This is because we have the decaying assumption $|\mu^\pm_e (x, v)| \lesssim \frac{1}{1+|e^\pm(x, v)|^\gamma} $, which gives that $\int_\Omega \int_{\mathbb{R}^3} |\hat{v} \cdot \textbf{E}|^2 |\mu^\pm_e| dx dv \lesssim \int_\Omega \int_{\mathbb{R}^3} |\hat{v} \cdot \textbf{E}|^2 \frac{1}{1+|e^\pm|^\gamma} dx dv \lesssim \|\textbf{E} \|^2_{H^1 (\Omega ; \mathbb{R}^3)}  $), we obtain that $D^\pm F^\pm \in \mathcal{H}^\pm$. Also, $F^\pm$ satisfies the specular boundary condition in the weak sense. 

Let $k^\pm = F^\pm/|\mu^\pm_e|$, then the equation for $F^\pm$ is equivalent to $(\lambda+D^\pm) k^\pm = \mp \hat{v} \cdot \textbf{E}$. Let $\epsilon$ be any positive number, $w_\epsilon = \frac{|\mu^\pm_e|}{\epsilon+|\mu^\pm_e|}$, $k^\pm_\epsilon = w_\epsilon k^\pm = F^\pm/(\epsilon+|\mu^\pm_e|)$, then multiplying the equation $(\lambda+D^\pm) k^\pm = \mp \hat{v} \cdot \textbf{E}$ by $w_\epsilon$ and pair with $k_\epsilon$ gives
$$
\la \lambda k_\epsilon + D^\pm k_\epsilon, k_\epsilon \ra_{\mathcal{H}^\pm}  = \mp \la w_\epsilon \hat{v} \cdot \textbf{E}, k_\epsilon \ . \ra_{\mathcal{H}^\pm}
$$
We have $k_\epsilon \in dom(D^\pm)$, so by the skew-symmetry of $D^\pm$ we deduce that $\la D^\pm k_\epsilon, k_\epsilon \ra_{\mathcal{H}^\pm} =0$. The equation becomes
$$
|\lambda| \|k_\epsilon \|^2_{\mathcal{H}^\pm} = |\la w_\epsilon \hat{v} \cdot \textbf{E}, k_\epsilon \ra_{\mathcal{H}^\pm} | \leq \|\textbf{E} \|_{\mathcal{H}^\pm} \|k_\epsilon\|_{\mathcal{H}^\pm}  \ .
$$
Letting $\epsilon \rightarrow 0$, we obtain that $k^\pm \in \mathcal{H}^\pm$, $ \int_\Omega \int_{\mathbb{R}^3} \frac{|F^\pm|^2}{|\mu^\pm_e|} dv dx < +\infty$.

We substitute $( e^{\lambda t} f^\pm , e^{\lambda t} \textbf{E}, e^{\lambda t} \textbf{B} ) $ as well as $\textbf{A}$ and $\phi$ defined as in \eqref{potentialdefinition} into the Maxwell system together with the boundary conditions. Note that by the definition of $F^\pm$ the charge density and the current density of the system are:
\begin{equation}
\begin{split}
& \rho = \int_{\mathbb{R}^3} (F^+ -F^-) dv + r  A_\varphi \int_{\mathbb{R}^3} ( \mu^+_p + \mu^-_p) dv  \\
& \textbf{j} = \int_{\mathbb{R}^3} \hat{v} (F^+ -F^-) dv + r   A_\varphi \int_{\mathbb{R}^3} \hat{v} (\mu^+_p + \mu^-_p) dv   \ .
\end{split}
\end{equation}
Since we already have that $\int_\Omega \int_{\mathbb{R}^3} \frac{|F^\pm|^2}{|\mu^\pm_e|}  dv dx < +\infty$, $\sup_x \int_{\mathbb{R}^3} (|\mu^\pm_e| +| \mu^\pm_p|) dv  < +\infty$, $A_\varphi \in L^2 (\Omega)$, we now know that $\rho$ and $\textbf{j} $ are finite almost everywhere and they are in $L^2 (\Omega)$. 

$A_\varphi$ satisfies the equation $(\lambda^2 -\Delta + \frac{1}{r^2}) A_\varphi = j_\varphi$. Since $\Omega$ may meet the $z$-axis, we mimic the process in \cite{LS1} to deal with the singularity at $r=0$. Recalling the identity \eqref{Deltagexp}, we can apply the standard elliptic theory to the equation $\lambda^2 A_\varphi e^{i \varphi} - \Delta (A_\varphi e^{i\varphi}) = j_\varphi e^{i \varphi} $ to conclude that $A_\varphi e^{i\varphi} \in H^2 (\Omega) $, $A_\varphi \in H^{2 \dagger}  (\Omega) $.
 
For $A_n$ we have
$$
(\lambda^2 -\Delta )A_n + \lambda ( \nabla \phi )_n  = \sum_\pm \int_{\mathbb{R}^3}  \frac{v_n}{\la v \ra}  ( \mu^\pm_e (1-\mathcal{Q}^\pm_\lambda) \phi  + \mu^\pm_p r  A_\varphi + \mu^\pm_e \mathcal{Q}^\pm_\lambda (\hat{v} \cdot \textbf{A})   ) dv  \ .
$$
Therefore
$$
 -\Delta A_n =  - \lambda ( \nabla \phi )_n - \lambda^2 A_n + \sum_\pm \int_{\mathbb{R}^3}  \frac{v_n}{\la v \ra}  ( \mu^\pm_e (1-\mathcal{Q}^\pm_\lambda) \phi  + \mu^\pm_p r  A_\varphi + \mu^\pm_e \mathcal{Q}^\pm_\lambda (\hat{v} \cdot \textbf{A})   ) dv  \ . 
$$
Taking the square of both sides and integrating on $\Omega$, we compute the left side:
\begin{equation}
\begin{split}
\int_\Omega (\Delta A_n)^2 dx  
& = - \int_\Omega  \nabla (\Delta A_n) \cdot \nabla A_n dx + \int_{\partial \Omega} (\Delta A_n) \frac{\partial A_n}{\partial n} \cdot dS_x \\
& =  \int_\Omega |D^2 A_n|^2 dx - \int_{\partial \Omega} | \nabla (\frac{\partial A_n}{\partial n}) \cdot \nabla A_n dS_x   | + \int_{\partial \Omega} (\Delta A_n) \frac{\partial A_n}{\partial n} \cdot dS_x \\
& =  \int_\Omega |D^2 A_n|^2 dx + \int_{\partial \Omega} | \frac{\partial A_n}{\partial n} (\Delta A_n) dS_x   | + \int_{\partial \Omega} (\Delta A_n) \frac{\partial A_n}{\partial n} \cdot dS_x \\
& =  \int_\Omega |D^2 A_n|^2 dx - \int_{\partial \Omega} |  A_n (\Delta A_n) \cdot \frac{-\tilde{z}'}{\tilde{r} \sqrt{\tilde{z}'^2 + \tilde{r}'^2 }} r d \varphi dx_{tg}   |  \\
& + \int_{\partial \Omega} (\Delta A_n) \frac{-\tilde{z}'}{\tilde{r} \sqrt{ \tilde{z}'^2 + \tilde{r}'^2 }} \cdot  r d \varphi dx_{tg}   \ .  \\
\end{split}
\end{equation}
The boundary terms are well-defined integrals and they can be controlled by a constant times $\|A_n \|^2_{H^1}$. Hence
\begin{equation}
\| A_n \|^2_{H^2} < +\infty  \ .
\end{equation}
Therefore $A_n$ is in $H^2 (\Omega )$.

Also, $\phi$ is in $H^2 (\Omega)$ and $A_{tg} \in H^2 (\Omega) $, according to elliptic theory (with Dirichlet boundary condition), so $\tilde{\textbf{A}} \in H^2 (\Omega;\tilde{ \mathbb{R}}^2) $. Hence $\textbf{E}, \textbf{B} \in H^1 (\Omega; \mathbb{R}^3)$. We also obtain from the equation $(\lambda + D^\pm) F^\pm = \mp \mu^\pm_e \hat{v}\cdot \textbf{E}$ that $ \int_\Omega \int_{\mathbb{R}^3} \frac{|D^\pm F^\pm|^2 }{|\mu^\pm_e|} dv dx < +\infty  $.

\end{proof}

Next we prove the adjointness properties for the operators $\mathcal{A}^0_1$, $\mathcal{A}^0_2$ and $\mathcal{B}^0$. Recall that $\mathcal{X} $ is the space consisting of all the scalar functions in $ H^{2, \tau} (\Omega) \cap H^{2 \dagger }  (\Omega)$ that satisfy the Dirichlet boundary condition.  

\begin{lemma}
(i) $\mathcal{A}^0_1$ and $\mathcal{A}^0_2 $ are self-adjoint on $L^{2, \tau} (\Omega)$ with domain $\mathcal{X}$.
(ii) $\mathcal{B}^0$ is well-defined on $\mathcal{X}$. The adjoint operator of $\mathcal{B}^0$ is 
\begin{equation}
(\mathcal{B}^0)^* h =  \sum_\pm \int_{\mathbb{R}^3}  ( \mu^\pm_e r  h + \mu^\pm_e \mathcal{P}^\pm ( \hat{v}_\varphi h ) )  dv  \ .
\end{equation}
\end{lemma}

\begin{proof}
The conclusion follows from the fact that $\mathcal{P}^\pm$ is self-adjoint on $\mathcal{H}^\pm$ and the Dirichlet boundary conditions.
\end{proof}

Note that $\mathcal{P}^\pm$ is bounded with norm no larger than $ \sup_x \int_{\mathbb{R}^3} |\mu^\pm_e| dv$, and that $\mathcal{A}^0_1 - \Delta $, $\mathcal{A}^0_2 + \Delta $ and $ \mathcal{B}^0 $ are bounded $\mathcal{X} \rightarrow L^2 (\Omega)$. Therefore $\mathcal{A}^0_1$ and $\mathcal{A}^0_2$ are well-defined on $\mathcal{X} $. Also, noticng that $\mathcal{P}^\pm$ are orthogonal projections and the Dirichlet boundary condition, we have
\begin{equation}
-\la \mathcal{A}^0_1 h, h\ra_{L^2} = - \la \Delta h, h \ra_{L^2} + \sum_\pm \| (1- \mathcal{P}^\pm)  h \|^2_{\mathcal{H}^\pm} \geq \|\nabla h \|^2_{L^2} \geq c_P^{-1} \| h \|_{L^2},  \  \forall h \in \mathcal{X}
\end{equation}	
with $c_P = c_P (\Omega)$ being the square of the Poincar\'{e} constant of $\Omega$. Therefore we have

\begin{lemma}  \label{A01invertibility}
$(\mathcal{A}^0_1)^{-1}$ exists and it is bounded $L^{2, \tau} (\Omega)  \rightarrow \mathcal{X}$. Moreover, its bound is no larger than the square of the Poincar\'{e} constant of $\Omega$, which we denote by $c_P $. Hence we can define the operator
\begin{equation}
\mathcal{L}^0 =\mathcal{A}^0_2 - \mathcal{B}^0 (\mathcal{A}^0_1)^{-1} (\mathcal{B}^0)^*
\end{equation}
on $\mathcal{X}$, and it is self-adjoint on $L^{2, \tau} (\Omega)$ with domain $\mathcal{X}$.
\end{lemma}

The following couple of lemmas are the same as the ones stated in Section 3.1 in \cite{NS1}, simply differing by a change of coordinates, so we omit the proofs.

\begin{lemma}  \label{invarianceI}
Let $(f^\pm , \textbf{E}, \textbf{B})$ be a solution of the linearized system, and suppose that $F^\pm \in C^1( \mathbb{R},L^2_{1/|\mu^\pm_e|} (\Omega \times \mathbb{R}^3) )$, $ \textbf{E}, \textbf{B} \in C^2( \mathbb{R}, H^2 (\Omega)) $. Then the linearized energy functional
\begin{equation}
\mathcal{I} (f^\pm, \textbf{E}, \textbf{B}) = \sum_\pm \int_\Omega \int_{\mathbb{R}^3} ( \frac{1}{|\mu^\pm_e|} |F^\pm|^2 - r   \mu^\pm_p \hat{v}_\varphi |A_\varphi|^2 ) dv dx + \int_\Omega (|\textbf{E}|^2 + |\textbf{B}|^2 ) dx
\end{equation}
is time-invariant.
\end{lemma}

\begin{lemma} \label{invarianceKg}
By the same assumption as in Lemma \ref{invarianceI}, the functionals
\begin{equation}
\mathcal{K}^\pm_g (f^\pm, \textbf{A}) =  \int_\Omega \int_{\mathbb{R}^3} ( F^\pm \mp \mu^\pm_e \hat{v} \cdot \textbf{A} ) g dv dx 
\end{equation}
are time-invariant for all $g \in ker D^\pm$. In particular, we can deduce that $\int_\Omega \int_{\mathbb{R}^3} f^\pm (t,x, v)  dv dx$ are time-invariant by taking $g =1$.
\end{lemma}

\begin{proposition}  \label{purelygrowing}
If $( e^{\lambda t} f^\pm, e^{\lambda t} \textbf{E} , e^{\lambda t} \textbf{B} ) $ is a complex growing/decaying mode with $Re \lambda \neq 0 $, then $ \lambda $ must be real, i.e. the mode is purely growing/decaying.
\end{proposition}

\begin{proof}
The proof is basically the same as in Section 3.2 in \cite{NS1}, except that when expanding $\int_\Omega |\textbf{B}|^2 dx$ we obtain a term $\int_\Omega \frac{1}{r^2} |A_\varphi|^2 dx$, which turns out to be finite since $A_\varphi$ is in $L^2_{1/r^2}$. We start with splitting $F^\pm $ into its even and odd parts (denoted by $F^\pm_{ev}$ and $F^\pm_{od}$, respectively) with respect to the variable $(v_r, v_z)$. Note that $D^\pm$ map even functions to odd functions and vice versa, so that we obtain the equation for $F^\pm_{od}$:
$$
(\lambda^2 - D^{\pm 2}) F^\pm_{od} = \mp \lambda \mu^\pm_e \hat{v} \cdot \tilde{\textbf{E}} \pm \mu^\pm_e D^\pm ( \hat{v}_\varphi E_\varphi)  \ .
$$
We multiply this equation by $\frac{1}{|\mu^\pm_e|} | \overline{F}^\pm_{od}|$, integrate over $\Omega \times \mathbb{R}^3$, then add up the $+$ and $-$ identities, and examine the imaginary part of the resulting identity. In the end we obtain
$$
Re \lambda Im \lambda \sum_\pm \int_\Omega \int_{\mathbb{R}^3} \big( \frac{1}{|\mu^\pm_e|} | F^\pm_{od}|^2  \big)  dv dx = - Re \lambda Im \lambda \int_\Omega  (  |B_\varphi|^2 + |E_\varphi|^2   ) dx \ .
$$
The opposite signs of the integrals together with the assumption that $Re \lambda \neq 0$ imply that $\lambda$ must be real.
\end{proof}

Next, fixing $\textbf{A} \in L^{2, \tau} (\Omega)$, we define 
\begin{equation}
\mathcal{J}_{\textbf{A}} (F^+, F^-) = \sum_\pm \int_\Omega \int_{\mathbb{R}^3} \frac{1}{|\mu^\pm_e |}  |F^\pm|^2 dv dx + \int_\Omega |\nabla \phi |^2 dx \ ,
\end{equation}
where the electric scalar potential $\phi$ satisfies the following Poisson equation
\begin{equation} \label{PoissonEqn}
-\Delta \phi = \rho = \int_{\mathbb{R}^3} (F^+-F^-) dv  + r  A_\varphi \int_{\mathbb{R}^3} (\mu^+_p +\mu^-_p  ) dv, \qquad \phi |_{\partial\Omega} =0  \ .
\end{equation}
The right side in the Poisson equation is in $L^2$ so $\phi$ is well-defined (exists and is uniquely determined). Define
\begin{equation}
\mathcal{F}_\textbf{A} =\{ (F^+, F^-) \in (L^2_{1/|\mu^\pm_e|} ( \Omega \times \mathbb{R}^3))^2  \cap (L^{2, \tau} )^2 :  \int_\Omega \int_{\mathbb{R}^3} (F^\pm \mp \mu^\pm_e \hat{v} \cdot \textbf{ A}  ) g dv dx =0, \  \forall g \in ker D^\pm      \}
\end{equation}
for fixed $\textbf{A}  \in L^{2, \tau} (\Omega)$. In particular, 
\begin{equation}
\mathcal{F}_0 =\{ (h^+, h^-) \in (L^2_{1/|\mu^\pm_e|} ( \Omega \times \mathbb{R}^3))^2 \cap (L^{2, \tau} )^2 :  \int_\Omega \int_{\mathbb{R}^3} h^\pm   g dv dx =0, \  \forall g \in ker D^\pm      \}  \ .
\end{equation}
Then $\mathcal{J}_\textbf{A}$ is well-defined and nonnegative on $\mathcal{F}_\textbf{A}$, and its infimum over $\mathcal{F}_\textbf{A}$ is finite. We want to minimize $\mathcal{J}_\textbf{A}$ on the manifold $\mathcal{F}_\textbf{A}$.

\begin{lemma} \label{minimization1}
For fixed $\textbf{A} \in L^{2, \tau} (\Omega)$, there exists a minimizer $(F^+_*, F^-_* )$ of the functional $\mathcal{J}_\textbf{A}$ on $\mathcal{F}_\textbf{A}$. Furthermore, if $\phi_* \in H^2 (\Omega)$ is the solution of the Poisson equation (\ref{PoissonEqn}) with respect to $(F^+_*, F^-_* )$, then 
\begin{equation}  \label{minimizerproperty1}
F^\pm_* = \pm \mu^\pm_e (1- \mathcal{P}^\pm) \phi_* \pm \mu^\pm_e \mathcal{P}^\pm (\hat{v} \cdot \textbf{A})  \ .
\end{equation}
\end{lemma}

\begin{proof}
Take a minimizing sequence $(F^+_n, F^-_n ) $ in $\mathcal{F}_\textbf{A}$ of $\mathcal{J}_\textbf{A}$. Note that $F^\pm_n$ are bounded in $L^2_{1/|\mu^\pm_e|}$, we can take weak limits (extracting subsequences if neccesary) $F^\pm_*$, which is also in $\mathcal{F}_\textbf{A}$. Therefore $(F^+_*, F^-_* )$ is a minimizer.

Let $\phi_*$ be as described in the lemma. We want to derive the (\ref{minimizerproperty1}). Denote $h^\pm = F^\pm \mp \mu^\pm_e \hat{v} \cdot \textbf{A} $, in particular,  $h^\pm_* = F^\pm_* \mp \mu^\pm_e \hat{v} \cdot \textbf{A} $. Then $(F^+, F^-) \in \mathcal{F}_\textbf{A}$ iff $(h^+, h^-) \in \mathcal{F}_0 $. 

We observe
\begin{equation}
\begin{split}
\int_{\mathbb{R}^3} (h^+ - h^-) dv 
& = \int_{\mathbb{R}^3} (F^+ - F^-) dv - \int_{\mathbb{R}^3} ( \mu^+_e -\mu^-_e  ) \hat{v} \cdot \textbf{A}   dv \\
& = \int_{\mathbb{R}^3} (F^+ - F^-) dv - \int_{\mathbb{R}^3} ( \mu^+_e -\mu^-_e  ) \hat{v}_\varphi  A_\varphi   dv \\
& = \int_{\mathbb{R}^3} (F^+ - F^-) dv + r  A_\varphi  \int_{\mathbb{R}^3} ( \mu^+_p -\mu^-_p  )   dv  \ .
\end{split}
\end{equation}
Here we used that $ \mu^\pm_e (\hat{v}_r A_r + \hat{v}_z A_z  )$ is odd in $(v_r, v_z)$ (the second line of computation), and $ \partial_{v_\varphi} (\mu^\pm) = \mu^\pm_e \hat{v}_\varphi + r \mu^\pm_p $ (the third line of computation). Therefore $\phi$ is independent of the change of variables $F^\pm \rightarrow h^\pm$. $(F^+_*, F^-_*)$ is a minimizer of $\mathcal{J}_\textbf{A} (F^+, F^-) $ on $\mathcal{F}_\textbf{A}$ iff $(h^+_*, h^-_*)$ is a minimizer of the following functional $\mathcal{J}_0 $ on $\mathcal{F}_0$:
\begin{equation}
\mathcal{J}_0 (h^+, h^-) = \sum_\pm \int_\Omega \int_{\mathbb{R}^3} \frac{1}{| \mu^\pm_e |} | h^\pm \pm \mu^\pm_e \hat{v} \cdot \textbf{A} |^2 dv dx + \int_\Omega |\nabla \phi |^2 dx  \ .
\end{equation}

Now we turn to minimize $\mathcal{J}_0$ on $\mathcal{F}_0$. We write down its first variation
\begin{equation}
\sum_\pm \int_\Omega \int_{\mathbb{R}^3} \frac{1}{|\mu^\pm_e|} (h^\pm_* \pm \mu^\pm_e \hat{v}\cdot \textbf{A} ) h^\pm dv dx  + \int_\Omega \nabla \phi_* \cdot \nabla \phi dx 
\end{equation}
and let it be $0$ for all $(h^+, h^-) \in \mathcal{F}_0$. Using the Dirichlet boundary condition on $\phi$, we compute 
\begin{equation}
\begin{split}
0 
& = \sum_\pm \int_\Omega \int_{\mathbb{R}^3} \frac{1}{|\mu^\pm_e|} (h^\pm_* \pm \mu^\pm_e \hat{v}\cdot \textbf{A} ) h^\pm dv dx  + \int_\Omega \nabla \phi_* \cdot \nabla \phi dx \\
& = \sum_\pm \int_\Omega \int_{\mathbb{R}^3} \frac{1}{|\mu^\pm_e|} (h^\pm_* \pm \mu^\pm_e \hat{v}\cdot \textbf{A} ) h^\pm dv dx - \int_\Omega \phi_* \Delta \phi dx \\
& = \sum_\pm \int_\Omega \int_{\mathbb{R}^3} \frac{1}{|\mu^\pm_e|} (h^\pm_* \pm \mu^\pm_e \hat{v}\cdot \textbf{A} ) h^\pm dv dx + \int_\Omega \int_{\mathbb{R}^3} \phi_* (h^+ - h^-) dv dx  \\
& = \sum_\pm \int_\Omega \int_{\mathbb{R}^3} \frac{1}{|\mu^\pm_e|} (h^\pm_* \pm \mu^\pm_e \hat{v}\cdot \textbf{A}  \mp \mu^\pm_e \phi_*  ) h^\pm dv dx  \ .
\end{split}
\end{equation}

Taking $h^\mp =0$, we obtain the identity for $h^\pm_*$. The computation then implies that $\frac{1}{|\mu^\pm_e|} (h^\pm_* \pm \mu^\pm_e \hat{v}\cdot \textbf{A})  \in {(ker D^\pm) ^\perp }^\perp = ker D^\pm$. Therefore, 
\begin{equation}
D^\pm ( F^\pm_* \mp \mu^\pm_e  \phi_*) = -\mu^\pm_e D^\pm ( \frac{1}{|\mu^\pm_e|} (h^\pm_* \pm \mu^\pm_e \hat{v}\cdot \textbf{A}) =0  \ .
\end{equation}
Hence
\begin{equation}
(1-\mathcal{P}^\pm) (F^\pm_* \mp \mu^\pm_e  \phi_*) =0 \ .
\end{equation}
Note that the constraint $  \int_\Omega \int_{\mathbb{R}^3} (F^\pm \mp \mu^\pm_e \hat{v} \cdot \textbf{ A}  ) g dv dx =0$, $  \forall g \in ker D^\pm  $ in the definition of $\mathcal{F}_\textbf{A}$ is equivalent to $\mathcal{P}^\pm (F^\pm_* \mp \mu^\pm_e \hat{v} \cdot \textbf{A} ) =0 $. Combining this together with $D^\pm ( F^\pm_* \mp \mu^\pm_e  \phi_*) = -\mu^\pm_e D^\pm ( \frac{1}{|\mu^\pm_e|} (h^\pm_* \pm \mu^\pm_e \hat{v}\cdot \textbf{A}) =0$, we obtain the desired identity (\ref{minimizerproperty1}).

\end{proof}

Next, we derive a connection between the minimizer of $\mathcal{J}_\textbf{A}$ and the operator $\mathcal{L}_0$.

\begin{lemma} \label{minimization2}
For fixed $\textbf{A} \in L^{2, \tau} (\Omega)$, and $A_\varphi \in \mathcal{X}$, let $F^\pm_*$ be a minimizer of $\mathcal{J}_\textbf{A}$ on $\mathcal{F}_\textbf{A}$, then 
\begin{equation}
\mathcal{J}_\textbf{A} (F^+_*, F^-_*)= - ( \mathcal{B}^0 (\mathcal{A}^0_1)^{-1} (\mathcal{B}^0)^* A_\varphi, A_\varphi ) + \sum_\pm \|\mathcal{P}^\pm (\hat{v} \cdot \textbf{A}) \|^2_{\mathcal{H}^\pm } \ .
\end{equation}
\end{lemma}

\begin{proof}
Plugging in $F^\pm_* = \pm \mu^\pm_e (1- \mathcal{P}^\pm) \phi_* \pm \mu^\pm_e \mathcal{P}^\pm (\hat{v} \cdot \textbf{A})$ into the expression of $\mathcal{J}_\textbf{A}$, we obtain

\begin{equation}
\begin{split}
\mathcal{J}_\textbf{A} (F^+_*, F^-_*)
& = \sum_\pm \|(1-\mathcal{P}^\pm) \phi_* \|^2_{\mathcal{H}^\pm} + \int_\Omega |\nabla \phi_*|^2 dx  + \sum_\pm \|\mathcal{P}^\pm (\hat{v} \cdot \textbf{A}) \|^2_{\mathcal{H}^\pm} \\
& = -\la \mathcal{A}^1_0 \phi_*, \phi_* \ra_{L^2} + \sum_\pm \|\mathcal{P}^\pm (\hat{v} \cdot \textbf{A}) \|^2_{\mathcal{H}^\pm}  \ .
\end{split}
\end{equation}
Now it suffices to show that $\phi_* = - (\mathcal{A}^0_1)^{-1} (\mathcal{B}^0)^* A_\varphi$. This can be done by plugging in $F^\pm_* = \pm \mu^\pm_e (1- \mathcal{P}^\pm) \phi_* \pm \mu^\pm_e \mathcal{P}^\pm (\hat{v} \cdot \textbf{A})$ into the Poisson equation:
\begin{equation}
\begin{split}
-\Delta \phi_* 
& = \sum_\pm ( \int_{\mathbb{R}^3}  \mu^\pm_e (1-\mathcal{P}^\pm)\phi_* dv +   \int_{\mathbb{R}^3}  r  \mu^\pm_p dv A_\varphi  +  \int_{\mathbb{R}^3} \mu^\pm_e \mathcal{P}^\pm (\hat{v}\cdot \textbf{A}) dv  ) \\
& = \sum_\pm ( \int_{\mathbb{R}^3}  \mu^\pm_e (1-\mathcal{P}^\pm)\phi_* dv +   \int_{\mathbb{R}^3}  r   \mu^\pm_p dv A_\varphi  +  \int_{\mathbb{R}^3} \mu^\pm_e \mathcal{P}^\pm (\hat{v}_\varphi A_\varphi) dv  )
\end{split} 
\end{equation}
which is equivalent to $- \mathcal{A}^0_1 \phi_* = (\mathcal{B}^0)^* A_\varphi $. Therefore  $\phi_* = - (\mathcal{A}^0_1)^{-1} (\mathcal{B}^0)^* A_\varphi$ since $\mathcal{A}^0_1$ is invertible. (Here in the last line of the computation we have used the fact that $\mathcal{P}^\pm (\hat{v}_r A_r + \hat{v}_z A_z)$ is odd in $(v_r, v_z)$.)

\end{proof}

Now we are in a position to prove the linear stability result:

\begin{proposition}
$\mathcal{L}^0 \geq 0$ implies that there is no growing mode $(e^{\lambda t} f^\pm, e^{\lambda t} \textbf{E}, e^{\lambda t} \textbf{B} )$ with $Re \lambda >0$.
\end{proposition}

\begin{proof}
We argue by contradiction. Assuming the contrary, we know that $\lambda \in \mathbb{R}$, so we can assume that $(f^\pm, \textbf{E}, \textbf{B} )$ is real-valued. Since $\mathcal{I} (f^\pm, \textbf{E}, \textbf{B} ) $ and $\mathcal{K}_g^\pm (f^\pm, \textbf{A})$ is invariant in time, they then must be $0$ all the time since we have the factor $e^{\lambda t}$. We can write
\begin{equation}
\begin{split}
0 = \mathcal{I} (f^\pm, \textbf{E}, \textbf{B} ) 
& =  \sum_\pm \int_\Omega \int_{\mathbb{R}^3} ( \frac{1}{|\mu^\pm_e|} |F^\pm|^2 - r  \mu^\pm_p \hat{v}_\varphi |A_\varphi|^2 ) dv dx + \int_\Omega (|\textbf{E}|^2 + |\textbf{B}|^2 ) dx  \\
& = \mathcal{J}_\textbf{A} (F^+, F^-) - \sum_\pm \int_\Omega \int_{\mathbb{R}^3} r   \mu^\pm_p  \hat{ v}_\varphi |A_\varphi|^2 dv dx + \int_\Omega ( \lambda^2 |\textbf{A}|^2 + |\textbf{B}|^2  ) dx   \ .
\end{split}
\end{equation}
Here we used the connection between $\phi$ and $\textbf{A}$, and the Coulomb gauge $\nabla \cdot \textbf{A} =0$ as well as the Dirichlet boundary condition on $\phi$.

Since $\mathcal{K}_g^\pm (f^\pm, \textbf{A}) =0 $ actually gives the constraint in the definition of $\mathcal{F}_\textbf{A}$, we can now minimize $\mathcal{J}_\textbf{A} (F^+, F^- ) $ on $\mathcal{F}_\textbf{A}$. Let $(F^+_*, F^-_*) $ be the minimizer, then by the previous lemmas, we have
\begin{equation}
\begin{split}
0 = \mathcal{I} (f^\pm, \textbf{E}, \textbf{B} ) 
& \geq  - ( \mathcal{B}^0 (\mathcal{A}^0_1)^{-1} (\mathcal{B}^0)^* A_\varphi, A_\varphi ) + \sum_\pm \|\mathcal{P}^\pm (\hat{v} \cdot \textbf{A}) \|^2_{\mathcal{H}^\pm} \\
&  - \sum_\pm \int_\Omega \int_{\mathbb{R}^3} r   \mu^\pm_p  \hat{ v}_\varphi |A_\varphi|^2 dv dx + \int_\Omega ( \lambda^2 |\textbf{A}|^2 + |\textbf{B}|^2  ) dx   \ .
\end{split}
\end{equation}
We compute, as in the proof of Proposition \ref{purelygrowing}:
\begin{equation}
\int_\Omega |\textbf{B}|^2 dx = \int_\Omega ( |\nabla A_\varphi|^2 + \frac{|A_\varphi|^2}{r^2 } + |B_\varphi|^2  ) dx  \ .
\end{equation}
Also, we have
\begin{equation}
\begin{split}
(\mathcal{A}^0_2 A_\varphi, A_\varphi)_{L^2} 
& = \int_\Omega ( |\nabla A_\varphi|^2 + \frac{|A_\varphi|^2}{r^2 }  ) dx + \sum_\pm \| \mathcal{P}^\pm (\hat{v}_\varphi A_\varphi)\|^2_{\mathcal{H}^\pm}  \\
& - \sum_\pm \int_\Omega \int_{\mathbb{R}^3} r  \mu^\pm_p  \hat{ v}_\varphi |A_\varphi|^2 dv dx   
\end{split}
\end{equation}
and
\begin{equation}
\begin{split}
\|\mathcal{P}^\pm (\hat{v} \cdot \textbf{A}) \|^2_{\mathcal{H}^\pm}
& = \| \mathcal{P}^\pm (\hat{v}_\varphi A_\varphi)\|^2_{\mathcal{H}^\pm} + \| \mathcal{P}^\pm (\hat{v}_r A_r + \hat{v}_z A_z)\|^2_{\mathcal{H}^\pm} + \la \mathcal{P}^\pm (\hat{v}_\varphi A_\varphi), \mathcal{P}^\pm (\hat{v}_r A_r + \hat{v}_z A_z) \ra_{\mathcal{H}^\pm} \\
& = \| \mathcal{P}^\pm (\hat{v}_\varphi A_\varphi)\|^2_{\mathcal{H}^\pm} + \| \mathcal{P}^\pm (\hat{v}_r A_r + \hat{v}_z A_z)\|^2_{\mathcal{H}^\pm } \ .
\end{split}
\end{equation}
(Here we used the Dirichlet boundary condition on $A_\varphi$.) Plugging all the three into the $ \mathcal{I} (f^\pm, \textbf{E}, \textbf{B} ) $ inequality, we obtain
\begin{equation}
\begin{split}
0 = \mathcal{I} (f^\pm, \textbf{E}, \textbf{B} ) 
& \geq  \la  \mathcal{L}^0 A_\varphi, A_\varphi \ra_{L^2} + \int_\Omega ( \lambda^2 |\textbf{A}|^2 + |B_\varphi|^2) dx \\
& + \sum_\pm  \| \mathcal{P}^\pm (\hat{v}_r A_r + \hat{v}_z A_z)\|^2_{\mathcal{H}^\pm}  \ .
\end{split}
\end{equation}
If $\mathcal{L}^0 \geq 0$, this implies $\textbf{A} =0$. Using that $\mathcal{I} (f^\pm, \textbf{E}, \textbf{B}) =0$, we deduce $f^\pm =0$, $\textbf{E}=0$. A contradiction.

\end{proof}

\section{Linear Instability}

We now turn to the instability part of Theorem \ref{mainresult}.

We build up the particle trajectories by solving the ODE \eqref{particletrajectoryODE}. For each $(x,v) \in \Omega \times \mathbb{R}^3$, we define its trajectory on $[0, T_1]$ by \eqref{particletrajectoryODE}, where $T_1$ is the first time it hits $\partial \Omega$. Then the trajectory is continued by the rule \eqref{particlereflection}, and after that the construction process is continued via solving \eqref{particletrajectoryODE} again. If a particle hits the spatial boundary infinitely many times in a finite time interval $[0, T]$, then we might not be able to continue the trajectory via \eqref{particletrajectoryODE} and \eqref{particlereflection} for $t > T$. Fortunately, we can actually prove that for almost every particle in $\Omega \times \mathbb{R}^3$, the trajectory is well-defined and hits the boundary at most finitely many times in each finite time interval. For detailed discussion of this observation, see Appendix C. Hence we can define trajectories globally in time for almost every particle, and these trajectories are piecewise differentiable.

We have the following lemmas as in \cite{NS1}:

\begin{lemma}
Let $\partial \Omega$ be $C^1$. For almost $(x, v) \in \Omega \times \mathbb{R}^3$, the particle trajectories $(X^\pm, V^\pm)$ are piecewise $C^1$ in $s \in \mathbb{R}$. For each $s$, the map $ (x,v) \rightarrow (X^\pm (s;x,v), V^\pm (s;x,v))$ is 1 - 1 and differentiable with Jacobian $1$ at all points $(x,v)$ such that $X^\pm (s;x,v)$ are not on the boundary. Moreover, we have the change-of-variable formula
\begin{equation} \label{changeofvariable}
\int_\Omega \int_{\mathbb{R}^3} g(x,v) dx dv = \int_\Omega \int_{\mathbb{R}^3} g( X^\pm (-s;y,w), V^\pm (-s;y,w))  dw dy
\end{equation}
\end{lemma}

\begin{proof}
The proof is basically identical to the one to Lemma 4.1 in \cite{NS1} since it has nothing to do with the geometry. Let $(x, v)$ be an arbitrary point in $\Omega \times \mathbb{R}^3$ so that the trajectory $(X^\pm (s;x,v), V^\pm (s;x,v))$ hits the boundary at most a finite number of times in each finite time interval. Except when it hits the boundary, the trajectory is $C^1$ in time. So the first assertion is clear. Given $s$, let $\mathcal{S}$ be the set of the points $(x, v)$ such that $X^\pm (s;x,v) \notin \partial \Omega$. Clearly, $\mathcal{S}$ is open and its complement in $\Omega \times \mathbb{R}^3$ has Lebesgue measure zero. For each $s$, the trajectory map is one-to-one on $\mathcal{S}$ since the ODE \eqref{particletrajectoryODE} and \eqref{particlereflection} is time-reversible and well-defined. In addition, a direct calculation shows that the Jacobian determinant is time-independent and is therefore equal to one. The change-of-variable formula
\eqref{changeofvariable} holds on the open set $\mathcal{S}$ and so on $\mathbb{R}^3$, as claimed.
\end{proof}

Recall that for any given $(X^\pm, V^\pm) \in \partial \Omega \times \mathbb{R}^3$, we can decompose $V^\pm$ into its $n$-component, $tg$-component and $\varphi$-component, and define $V_*^\pm = - V^\pm_n e_n + V^\pm_{tg} e_{tg} + V^\pm_\varphi e_\varphi $ as in \eqref{reflectedvelocity}. We have 

\begin{lemma}
Let $g(x,v)$ be a $C^1$ radial function on $\overline{\Omega} \times \mathbb{R}^3$. If $g$ is specular on $\partial \Omega$ then for all $s$, $g( X^\pm (s;x,v), V^\pm (s;x,v)  )$ is continuous and also specular on $\partial \Omega$, i.e.
\begin{equation}
g( X^\pm (s;x,v), V^\pm (s;x,v)  ) =g( X^\pm (s;x,v_* ), V^\pm (s;x,v_*)  )
\end{equation}
for almost every $(x,v) \in \partial \Omega \times \mathbb{R}^3$. $v_*$ is defined as in \eqref{reflectedvelocity}.
\end{lemma} 

\begin{proof}
From the equation of the particle trajectory we know that for almost every $x \in \partial \Omega $, the trajectory is unaffected by whether it starts with $v$ or $v_*$. Therefore for all $s$ we have

$$ X^\pm (s;x,v) =  X^\pm (s;x,v_*) $$
$$ V_n^\pm (s;x,v) =  V_n^\pm (s;x,v_*) \ if \ X^\pm (s;x,v) \notin \partial \Omega $$
$$ V_n^\pm (s;x,v) =  -V_n^\pm (s;x,v_*) \ if \ X^\pm (s;x,v) \in \partial \Omega $$
$$ V_j^\pm (s;x,v) =  V_j^\pm (s;x,v_*) \ for \ j=tg, \varphi $$

Since $g$ is specular, on the boundary it takes the same value at $v_n$ and $-v_n$. Hence $g(X^\pm(s), V^\pm (s))$ is a continuous function of $s$ at the points of reflection. Also, it is specular because of the rule \eqref{particlereflection}.

\end{proof}

We already know from Proposition \ref{purelygrowing} that a growing mode must be purely growing. Consider any growing mode $(e^{\lambda t} f^\pm, e^{\lambda t} \textbf{E}, e^{\lambda t} \textbf{B}  )$. Our goal is to find such a solution for the linearized equation for some $\lambda >0$.

Plugging $(e^{\lambda t} f^\pm, e^{\lambda t} \textbf{E}, e^{\lambda t} \textbf{B}  )$ into the linearized Vlasov-Maxwell system, we obtain 
 
\begin{equation}
\lambda f^\pm + D^\pm f^\pm =  \pm (  \mu^\pm_e D^\pm \phi +  \mu^\pm_e \hat{v} \cdot \lambda \textbf{A} + r \mu^\pm_p \lambda A_\varphi + \mu^\pm_p D^\pm (r A_\varphi)  )
\end{equation}
 
Multiplying this equation by $e^{\lambda s}$ and integrating along the particle trajectories $(X^\pm (s; x,v), V^\pm (s; x, v))$ from $s = -\infty$ to $0$, we obtain
\begin{equation}  \label{linearizedfformula}
f^\pm (x, v) = \pm \mu^\pm_p r A_\varphi \pm \mu^\pm_e \phi \pm \mu^\pm_e \mathcal{Q}^\pm_\lambda (\hat{v} \cdot \textbf{A} - \phi)
\end{equation}
where $ \mathcal{Q}^\pm_\lambda (g) (x,v) = \int^0_{-\infty} \lambda e^{\lambda s} g(X^\pm (s; x,v), V^\pm (s;x,v) ) ds $. Plug (\ref{linearizedfformula}) into the Maxwell system, we obtain
\begin{equation} \label{system1}
-\Delta \phi = \sum_\pm \int_{\mathbb{R}^3} ( \mu^\pm_e (1-\mathcal{Q}^\pm_\lambda) \phi  + \mu^\pm_p r  A_\varphi + \mu^\pm_e \mathcal{Q}^\pm_\lambda (\hat{v} \cdot \textbf{A})  ) dv \ ,
\end{equation}
\begin{equation} \label{system2}
(\lambda^2 -\Delta + \frac{1}{r^2}) A_\varphi  = \sum_\pm \int_{\mathbb{R}^3} \hat{v}_\varphi (\mu^\pm_e (1-\mathcal{Q}^\pm_\lambda) \phi  + \mu^\pm_p r A_\varphi + \mu^\pm_e \mathcal{Q}^\pm_\lambda (\hat{v} \cdot \textbf{A})  ) dv \ ,
\end{equation}
\begin{equation} \label{system3}
(\lambda^2 -\Delta ) \tilde{\textbf{A}} + \lambda \nabla \phi  = \sum_\pm \int_{\mathbb{R}^3}  \frac{\tilde{v}}{\la v \ra} ( \mu^\pm_e (1-\mathcal{Q}^\pm_\lambda) \phi  + \mu^\pm_p r  A_\varphi + \mu^\pm_e \mathcal{Q}^\pm_\lambda (\hat{v} \cdot \textbf{A}) ) dv \ .
\end{equation}

We formally define 
\begin{equation}
\mathcal{A}^\lambda_1 h = \Delta h + \sum_\pm \int_{\mathbb{R}^3} \mu^\pm_e (1- \mathcal{Q}^\pm_\lambda) h dv \ ,
\end{equation}
\begin{equation}
\mathcal{A}^\lambda_2 h = ( \lambda^2 - \Delta + \frac{1}{r^2 } ) h - \sum_\pm \int_{\mathbb{R}^3}  \hat{v}_\varphi \big( \mu^\pm_p r  h + \mu^\pm_e \mathcal{Q}^\pm_\lambda (\hat{v}_\varphi h) \big) dv \ ,
\end{equation}
\begin{equation}
\mathcal{B}^\lambda h = - \sum_\pm \int_{\mathbb{R}^3}  \hat{v}_\varphi \mu^\pm_e (1- \mathcal{Q}^\pm_\lambda) h dv \ ,
\end{equation}
\begin{equation}
\begin{split}
& (\mathcal{B}^\lambda )^* h =  \sum_\pm \int_{\mathbb{R}^3}  ( \mu^\pm_p r  h + \mu^\pm_e \mathcal{Q}^\pm_\lambda ( \hat{v}_\varphi h )   )   dv  \ ,  \\
\end{split}
\end{equation}
\begin{equation}
\mathcal{T}^\lambda_1 h  = -\lambda \nabla h - \sum_\pm \int_{\mathbb{R}^3 }  \frac{\tilde{v}}{\la v \ra}  \mu^\pm_e \mathcal{Q}^\pm_\lambda h dv \ ,
\end{equation}
\begin{equation}
\mathcal{T}^\lambda_2 h = \sum_\pm \int_{\mathbb{R}^3} \frac{\tilde{v}}{\la v \ra}  ( \mu^\pm_e \mathcal{Q}^\pm_\lambda ( \hat{v}_\varphi A_\varphi ) + 2 \mu^\pm_p r   \mathcal{Q}^\pm_\lambda A_\varphi  ) dv \ ,
\end{equation}
\begin{equation}
(\mathcal{T}^\lambda_1 )^* \tilde{\textbf{h}} = \sum_\pm \int_{\mathbb{R}^3} \mu^\pm_e \mathcal{Q}^\pm_\lambda (\hat{v} \cdot \tilde{\textbf{h}})  dv +  \lambda \nabla \cdot \tilde{\textbf{h}} \ ,
\end{equation}
\begin{equation}
(\mathcal{T}^\lambda_2)^* \tilde{\textbf{h}} = - \sum_\pm \int_{\mathbb{R}^3}  \hat{v}_\varphi \mu^\pm_e \mathcal{Q}^\pm_\lambda (\hat{v} \cdot \tilde{\textbf{h}} ) dv  \ ,
\end{equation}
\begin{equation}
\mathcal{S}^\lambda \tilde{\textbf{h}} = (- \lambda^2 + \Delta ) \tilde{\textbf{h}} + \sum_\pm \int_{\mathbb{R}^3} \frac{\hat{v}}{\la v \ra} \mu^\pm_e  \mathcal{Q}^\pm_\lambda ( \hat{v} \cdot \tilde{\textbf{h}}) dv \ .
\end{equation}
The adjointness will be proved in Lemma \ref{OperatorAdjointness}. Then the equations (\ref{system1}), (\ref{system2}), (\ref{system3}) can be written as 

\begin{equation}  \label{MaxwellMatrixEqn}
\left( \begin{array}{ccc} 
\mathcal{A}^\lambda_1 & ( \mathcal{B}^\lambda  )^*  &  ( \tilde{\mathcal{T}}^\lambda_1 )^*   \\
 \mathcal{B}^\lambda & \mathcal{A}^\lambda_2 & ( \tilde{\mathcal{T}}^\lambda_2 )^*  \\
 \tilde{\mathcal{T}}^\lambda_1  &  \tilde{\mathcal{T}}^\lambda_2  &  \tilde{\mathcal{S}}^\lambda 
 \end{array} \right)
\left( \begin{array}{ccc}
\phi \\
A_\varphi \\
\tilde{\textbf{A}} 
\end{array} \right)
\begin{array}{ccc}
=0  \ . \end{array} 
\end{equation}

Motivated by \eqref{fullboundarycondition}, we define
$$
\tilde{\mathcal{Y}} = \{  \tilde{\textbf{h}} \in H^{2, \tau} (\Omega; \tilde{\mathbb{R}^2}) :  \nabla \cdot \tilde{\textbf{h}} =0 \  in \  \Omega \  ;  \tilde{r} \tilde{z}' \partial_r h_n - \tilde{r} \tilde{r}' \partial_z h_n + \tilde{z}' h_n =0 , h_{tg} =0 \ on \ \partial \Omega   \}  \ .
$$
Then $\tilde{\textbf{A}}$ is in $\tilde{\mathcal{Y}} $. We now study the properties of the operators $\mathcal{A}^\lambda_1$, $\mathcal{A}^\lambda_2$, etc..

\begin{lemma}
Let $0 < \lambda < +\infty$, then
\vskip 0.1 cm
(i) $\mathcal{Q}^\pm_\lambda$ is bounded from $\mathcal{H}^\pm$ to itself with operator norm one to one.
\vskip 0.1 cm
(ii) $\mathcal{A}^\lambda_1 - \Delta$, $\mathcal{A}^\lambda_2 -\lambda^2 + \Delta$ and $\mathcal{B}^\lambda$ are bounded from $\mathcal{X}$ to $L^{2, \tau} (\Omega)$.   
\vskip 0.1 cm
(iii) $\tilde{\mathcal{S}}^\lambda + \lambda^2  -\Delta $ is bounded from $L^{2, \tau} (\Omega; \tilde{\mathbb{R}^2})$ to $L^{2, \tau} (\Omega; \tilde{\mathbb{R}^2})$.
\vskip 0.1 cm
(iv) $\tilde{\mathcal{T}}^\lambda_1 + \lambda \nabla $ and $ \tilde{\mathcal{T}}^\lambda_2 $ are bounded from $L^{2, \tau} (\Omega)$ to $L^{2, \tau} (\Omega; \tilde{\mathbb{R}^2})$. 
All the operator norms here are independent of $\lambda$.
\end{lemma}

\begin{proof}
$\forall g, h \in \mathcal{H}^\pm$, 
\begin{equation}
\begin{split}
& | \int_\Omega \int_{\mathbb{R}^3}  \mu^\pm_e g \mathcal{Q}^\pm_\lambda h dv dx  |  \\
&  = | \int^0_{-\infty}  \int_\Omega \int_{\mathbb{R}^3} \lambda e^{\lambda s} \mu^\pm_e g(x,v) h(X^\pm (s), V^\pm (s)  ) dv dx  ds    |   \\
& \leq (\int^0_{-\infty} \lambda e^{\lambda s} \int_\Omega \int_{\mathbb{R}^3} |\mu^\pm_e |  |g(x,v)|^2  dv dx ds   )^{1/2}   (\int^0_{-\infty} \lambda e^{\lambda s} \int_\Omega \int_{\mathbb{R}^3} |\mu^\pm_e |  |h( X^\pm (s), V^\pm (s) )|^2  dv dx ds   )^{1/2}  \\
& \leq \|g \|_{\mathcal{H}^\pm} \|h \|_{\mathcal{H}^\pm}  \ .
\end{split}
\end{equation}
Here we changed the variables $(x,v) = (X^\pm (s; x,v), V^\pm (s; x,v) )$, with Jacobian determinant one. This gives that $\|\mathcal{Q}^\pm_\lambda \|_{\mathcal{H}^\pm \rightarrow \mathcal{H}^\pm} \leq 1$. Note that $\mathcal{Q}^\pm_\lambda (1) =1$, so $\|\mathcal{Q}^\pm_\lambda \|_{\mathcal{H}^\pm \rightarrow \mathcal{H}^\pm} = 1$. (i) is proved.

For $\mathcal{A}^\lambda_1 - \Delta $, we have (for $h \in L^{2, \tau} (\Omega)$)
\begin{equation}
\begin{split}
| \la  (\mathcal{A}^\lambda_1 - \Delta) h, h   \ra_{L^2}  | 
& = | \sum_\pm \int_\Omega \int_{\mathbb{R}^3}  \mu^\pm_e (1 - \mathcal{Q}^\pm_\lambda) h^2 dv dx  |  \\
& \leq 2 \sum_\pm \|h \|^2_{\mathcal{H}^\pm}  \\
& \leq 2 \sup_x ( \sum_\pm \int_{\mathbb{R}^3} |\mu^\pm_e| dv ) \|h\|^2_{L^2} \ .
\end{split}
\end{equation}
Hence $\mathcal{A}^\lambda_1 - \Delta $ is $L^{2, \tau} (\Omega) $ to $L^{2, \tau} (\Omega)$. Similarly, $\mathcal{B}^\lambda$ is bounded from $L^{2, \tau} (\Omega)$ to $L^{2, \tau} (\Omega)$, and $\mathcal{A}^\lambda_2 -\lambda^2 + \Delta$ is bounded from $L^{2, \tau} (\Omega) \cap Y$ to $L^{2, \tau} (\Omega)$ since there is the part including $1/r^2 $. (ii) is proved.

(iii) and (iv) follow directly from the definitions and (i). 

\end{proof}

Therefore we deduce

\begin{lemma}
Let $0< \lambda < +\infty$, then
\vskip 0.1 cm
(i) $\mathcal{A}^\lambda_1$ and $\mathcal{A}^\lambda_2$ are well-defined operators from $\mathcal{X}  \subset L^{2, \tau} (\Omega)$ to $L^{2, \tau} (\Omega)$. 
\vskip 0.1 cm
(ii) $\tilde{\mathcal{S}}^\lambda$ is well-defined from $\tilde{\mathcal{Y}} \subset L^{2, \tau} (\Omega; \tilde{\mathbb{R}^2})$ into $L^{2, \tau} (\Omega; \tilde{\mathbb{R}^2})$. 
\vskip 0.1 cm
(iii) $\tilde{\mathcal{T}}^\lambda_1$ is well-defined from $\mathcal{X}_1 := \{ h \in H^{1, \tau}  (\Omega)  | h=0, \forall h \in \partial \Omega   \} $ into $L^{2, \tau} (\Omega; \tilde{\mathbb{R}^2}) $.
\end{lemma}

For the proof of the adjointness of the operators, we begin with an adjoint formula for $\mathcal{Q}^\pm_\lambda$:

\begin{lemma}
For all $h(x,v)$, $g(x, v)$ that are independent of the $\varphi$-component of $x$, we have
\begin{equation}
\int_\Omega \int_{\mathbb{R}^3} \mu^\pm_e h (x,v) \mathcal{Q}^\pm_\lambda (g (x,v)) dv dx = \int_\Omega \int_{\mathbb{R}^3} \mu^\pm_e g (x,\mathcal{R} v) \mathcal{Q}^\pm_\lambda (h (x, \mathcal{R} v)) dv dx
\end{equation}
where $\mathcal{R} v :=  -v_r e_r - v_z e_z + v_\varphi  e_\varphi  $ for $v = v_r e_r + v_z e_z + v_\varphi  e_\varphi   $, $g$, $h \in \mathcal{H}^\pm$. 
\end{lemma}

\begin{proof}

We do the change of variables, which has Jacobian $1$:
$$
(y, w) := (X^\pm (s;x,v) , V^\pm (s;x,v)  ), (x, v) := (X^\pm (-s; y, w), V^\pm (-s; y, w))  \ .
$$
Therefore the left side of the desired formula can be written as
\begin{equation}
\begin{split}
& \int^0_{-\infty} \int_\Omega \int_{\mathbb{R}^3} \lambda e^{\lambda s} \mu_e^\pm h(x,v) g(X^\pm (s;x,v), V^\pm (s;x,v)  ) dv dx ds  \\
& = \int^0_{-\infty} \int_\Omega \int_{\mathbb{R}^3} \lambda e^{\lambda s} \mu_e^\pm h(X^\pm (-s;y,w), V^\pm (-s;y,w)  ) g(y,w) dw dy ds   \ .  \\
\end{split}
\end{equation}
Note that the ODE for the characteristics and the specular condition are invariant under the transformation $s \rightarrow -s$, $r \rightarrow r$, $\varphi \rightarrow \varphi$, $z \rightarrow z$, $v_r \rightarrow -v_r$, $v_z \rightarrow -v_z$, $v_\varphi \rightarrow v_\varphi$. Let $X_{r, z}$ be the projection on $(r, z)$-plane of $X$. Note that $\textbf{E}^0$ and $\textbf{B}^0$ has no $e_\varphi$-component. For $x \notin \partial \Omega$, a direct computation gives 
$$
X^\pm_{r,z} (-s;x,v) = X^\pm_{r,z} (s;x, \mathcal{R}v ) , V^\pm (-s;x,v) = \mathcal{R} V^\pm (s;x, \mathcal{R}v )   \ .
$$
Hence
\begin{equation}
\begin{split}
& \int^0_{-\infty} \int_\Omega \int_{\mathbb{R}^3} \lambda e^{\lambda s} \mu_e^\pm h(x,v) g(X^\pm (s;x,v), V^\pm (s;x,v)  ) dv dx ds  \\
& = \int^0_{-\infty} \int_\Omega \int_{\mathbb{R}^3} \lambda e^{\lambda s} \mu_e^\pm h(X^\pm (s;y, \mathcal{R} w), \mathcal{R} V^\pm (-s;y, \mathcal{R} w)  ) g(y,w) dw dy ds \\
& = \int^0_{-\infty} \int_\Omega \int_{\mathbb{R}^3} \lambda e^{\lambda s} \mu_e^\pm h (X^\pm (s;x,v), \mathcal{R} V^\pm (s;x,v)  ) g (x, \mathcal{R} v)  dv dx ds    \\
\end{split}
\end{equation}
since $h(x,v)$, $g(x, v)$ that are independent of the $\varphi$-component of $x$($X$). This gives the formula we want. 

\end{proof}

Now we have 

\begin{lemma}  \label{OperatorAdjointness}
Let $0 < \lambda < +\infty$, then
\vskip 0.1 cm
(i) $\mathcal{A}^\lambda_1$ and $\mathcal{A}^\lambda_2$ are self-adjoint operators on  $\mathcal{X}$.
\vskip 0.1 cm 
(ii) $\tilde{\mathcal{S}}^\lambda$ is self-adjoint on $L^{2, \tau} (\Omega; \tilde{\mathbb{R}}^2)$ with domain $\tilde{\mathcal{Y}}_1 = \{  \tilde{\textbf{h}} \in H^{2, \tau} (\Omega; \tilde{\mathbb{R}^2}) :  \tilde{r} \tilde{z}' \partial_r h_n -  \tilde{r} \tilde{r}' \partial_z h_n + \tilde{z}' h_n =0 , h_{tg} =0, \forall x \in \partial \Omega   \}$.
\vskip 0.1 cm
(iii) The adjoint of $\tilde{\mathcal{T}}^\lambda_1$, $\tilde{\mathcal{T}}^\lambda_2$, $\mathcal{B}^\lambda$ are as stated before. $( \tilde{\mathcal{T}}^\lambda_1 )^* $, $(\tilde{\mathcal{T}}^\lambda_2  )^* $, $( \mathcal{B}^\lambda )^*$ are well-defined on $\{ \tilde{\textbf{h}} \in L^{2, \tau} (\Omega; \tilde{\mathbb{R}}^2) : \nabla \cdot \tilde{\textbf{h}} =0  \} $, $ L^{2, \tau} (\Omega; \tilde{\mathbb{R}}^2 ) $ and $L^{2, \tau} (\Omega)$, respectively. $(\tilde{\mathcal{T}}^\lambda_2  )^* $, $( \mathcal{B}^\lambda )^*$, and $( \tilde{\mathcal{T}}^\lambda_1 + \lambda \nabla )^*  $ are bounded operators. 
\end{lemma}

\begin{proof}
For $h$, $g \in L^{2, \tau} (\Omega)$ (note that they are just functions of $x$), we have
\begin{equation}
\la ( \mathcal{A}^\lambda_1 - \Delta) h, g \ra_{L^2} = \sum_\pm \la (1- \mathcal{Q}^\pm_\lambda) h, g \ra_{\mathcal{H}^\pm} = \sum_\pm \la h, (1- \mathcal{Q}^\pm_\lambda)  g \ra_{\mathcal{H}^\pm} = \la h, ( \mathcal{A}^\lambda_1 - \Delta)  g \ra_{L^2}   \ ,
\end{equation}
\begin{equation}
\begin{split}
\la ( \mathcal{A}^\lambda_2 + \Delta) h, g \ra_{L^2} 
& = \la \frac{1}{r^2  } h, g \ra_{L^2}   -\sum_\pm \int_\Omega \int_{\mathbb{R}^3} \hat{v}_\varphi ( \mu^\pm_e r  h g + \mu^\pm_e \mathcal{Q}^\pm_\lambda (\hat{v}_\varphi h ) g ) dv dx \\
& =  \la \frac{1}{r^2 } g, h \ra_{L^2} -\sum_\pm \int_\Omega \int_{\mathbb{R}^3} \hat{v}_\varphi ( \mu^\pm_e r  h g + \mu^\pm_e \mathcal{Q}^\pm_\lambda (\hat{v}_\varphi g ) h ) dv dx \\
& = \la ( \mathcal{A}^\lambda_2 + \Delta) g, h \ra_{L^2}   \ .
\end{split}
\end{equation}
From this we obtain that $\mathcal{A}^\lambda_1 -\Delta $ and $\mathcal{A}^\lambda_2 +\Delta$ are self-adjoint. Since we have the Dirichlet boundary condition, it then follows that $\mathcal{A}^\lambda_1 $ and $\mathcal{A}^\lambda_2$ are self-adjoint on $\mathcal{X}$.

The adjointness of $\mathcal{B}^\lambda$ and $(\mathcal{B}^\lambda )^* $ also follows from the previous lemma because
\begin{equation}
\begin{split}
\la \mathcal{B}^\lambda  h, g  \ra_{L^2} 
& =  - \sum_\pm\int_\Omega \int_{\mathbb{R}^3}  \hat{v}_\varphi \mu^\pm_e (1- \mathcal{Q}^\pm_\lambda) h \cdot g dv dx \\
& =  \sum_\pm \int_{\mathbb{R}^3}  h ( \mu^\pm_e r  g + \mu^\pm_e \mathcal{Q}^\pm_\lambda ( \hat{v}_\varphi g ) )  dv dx   \ . \\
\end{split}
\end{equation}

For $\tilde{\mathcal{T}}^\lambda_1$ and $\tilde{\mathcal{T}}^\lambda_2$, we have ($h \in \mathcal{X}_1$, $\tilde{\textbf{g}} \in L^2 (\Omega; \tilde{\mathbb{R}}^2) $)
$$
\la ( \tilde{\mathcal{T}}^\lambda_1 + \lambda \nabla  )h, \tilde{\textbf{g}}  \ra_{L^2} = \sum_\pm \la  \mathcal{Q}^\pm_\lambda h, \hat{v} \cdot \tilde{\textbf{g}} \ra_{\mathcal{H}^\pm} = - \sum_\pm \la h, \mathcal{Q}^\pm_\lambda (\hat{v} \cdot \tilde{\textbf{g}}) \ra_{\mathcal{H}^\pm } \ , 
$$
$$
\la \tilde{\mathcal{T}}^\lambda_2 h, \tilde{\textbf{g}}  \ra_{L^2} =  - \sum_\pm \la  \mathcal{Q}^\pm_\lambda (\hat{v}_\varphi h), \hat{v} \cdot \tilde{\textbf{g}} \ra_{\mathcal{H}^\pm} =  \sum_\pm \la h,  \hat{v}_\varphi \mathcal{Q}^\pm_\lambda (\hat{v} \cdot \tilde{\textbf{g}}) \ra_{\mathcal{H}^\pm } \ .
$$
From the Dirichlet boundary condition on $\phi$, there holds
$$
\la \nabla h, \tilde{\textbf{g}} \ra_{L^2} = \la h, \nabla \cdot \tilde{\textbf{g}} \ra_{L^2}  + \int_{\partial \Omega} h \tilde{\textbf{g}} \cdot e_r d S_x = - \la h, \nabla \cdot \tilde{\textbf{g}} \ra_{L^2} \ .
$$
Hence,
$$
\la  \tilde{\mathcal{T}}^\lambda_1 h, \tilde{\textbf{g}}  \ra_{L^2} = - \sum_\pm \la h, \mathcal{Q}^\pm_\lambda (\hat{v} \cdot \tilde{\textbf{g}}) \ra_{\mathcal{H}^\pm } + \lambda \la h, \nabla \cdot \tilde{\textbf{g}} \ra_{L^2} = \la   h, (\tilde{\mathcal{T}}^\lambda_1)^* \tilde{\textbf{g}}  \ra_{L^2}   \ ,
$$
$$
\la \tilde{\mathcal{T}}^\lambda_2 h, \tilde{\textbf{g}}  \ra_{L^2}  =  \sum_\pm \la h,  \hat{v}_\varphi \mathcal{Q}^\pm_\lambda (\hat{v} \cdot \tilde{\textbf{g}}) \ra_{\mathcal{H}^\pm } = \la h, ( \tilde{\mathcal{T}}^\lambda_2)^* \tilde{\textbf{g}}  \ra_{L^2}  \ .
$$
The adjointness of $\tilde{\mathcal{T}}^\lambda_1$ and $(\tilde{\mathcal{T}}^\lambda_1 )^*$, $\tilde{\mathcal{T}}^\lambda_2$ and $ (\tilde{\mathcal{T}}^\lambda_2 )^* $ are therefore verified. 

Also, since $\mathcal{Q}^\pm_\lambda$ is bounded with norm one, it follows that all these operators are bounded. 

For $\tilde{\mathcal{S}}^\lambda$, we have (let $\tilde{\textbf{h}} \in \tilde{\mathcal{Y}}_1$)
$$
\la (\tilde{\mathcal{S}}^\lambda - \Delta ) \tilde{\textbf{h}}, \tilde{\textbf{g}}  \ra_{L^2} = -\lambda^2 \la \tilde{\textbf{h}}, \tilde{\textbf{g}}\ra_{L^2} - \sum_\pm \la \mathcal{Q}^\pm_\lambda (\hat{v} \cdot \tilde{\textbf{h}} ), \hat{v} \cdot \tilde{\textbf{g}} \ra_{\mathcal{H}^\pm} \ . 
$$
Therefore $\tilde{\mathcal{S}}^\lambda - \Delta$ is self-adjoint on $\tilde{\mathcal{Y}}_1$. Recall that $\partial_n$ means taking derivative in the normal direction, and $g_r$, $g_z$, $h_r$, $h_z$ are the $r$, $z$-components of $\tilde{g}$, $\tilde{h} $. For the operator $\Delta$, if $\tilde{\textbf{h}} \in \tilde{\mathcal{Y}}_1$, we have
\begin{equation}
\begin{split} 
& \la  \Delta \tilde{\textbf{h}}, \tilde{\textbf{g}}  \ra_{L^2} -  \la \tilde{\textbf{h}}, \Delta \tilde{\textbf{g}}  \ra_{L^2}  \\
& = \int_{\partial \Omega} (\partial_n \tilde{\textbf{h}} \cdot \tilde{\textbf{g}} - \partial_n \tilde{\textbf{g}} \cdot \tilde{\textbf{h}} ) d S_x   \\ 
& =  \int_{\partial \Omega} (\partial_n h_n \cdot g_n - \partial_n g_n \cdot h_n + \partial_n h_{tg} \cdot g_{tg}  - \partial_n g_{tg} \cdot h_{tg} ) d S_x  \\ 
& =  \int_{\partial \Omega} ( - \frac{\tilde{z}'}{\tilde{r} \sqrt{\tilde{z}'^2 + \tilde{r}'^2}} h_n   g_n - \partial_n  g_n  h_n + \partial_n h_{tg} \cdot g_{tg}  - \partial_n g_{tg} \cdot h_{tg} ) d S_x   \ . \\ 
\end{split} 
\end{equation} 
This requires exactly that $\tilde{\textbf{g}}$ should also lie in $\tilde{\mathcal{Y}}_1$ if we want it to be zero for arbitrary $\tilde{\textbf{h}}$:
\begin{equation}
\begin{split} 
& \la  \Delta \tilde{\textbf{h}}, \tilde{\textbf{g}}  \ra_{L^2} -  \la \tilde{\textbf{h}}, \Delta \tilde{\textbf{g}}  \ra_{L^2}  \\
& =  \int_{\partial \Omega} ( - \frac{\tilde{z}'}{\tilde{r} \sqrt{ \tilde{z}'^2 + \tilde{r}'^2 }} h_n   g_n + \frac{\tilde{z}'}{\tilde{r} \sqrt{ \tilde{z}'^2 + \tilde{r}'^2 }}  g_n  h_n + \partial_n h_{tg} \cdot g_{tg}  - \partial_n g_{tg} \cdot h_{tg} ) d S_x  \\ 
& = 0   \ .\\ 
\end{split} 
\end{equation} 
Here the integral $\int_{\partial \Omega} \frac{\tilde{z}'}{\tilde{r} \sqrt{ \tilde{z}'^2 + \tilde{r}'^2 }} h_n g_n dS_x$ is finite, since $dS_x = r d\varphi dx_{tg} = \tilde{r} \sqrt{1 + \tilde{r}'^2/\tilde{z}'^2} d \varphi dz    $, which eliminates the $1/r$ factor. Also, 
$$
\int_{\partial \Omega} \partial_n g_{tg} \cdot h_{tg} dS_x = \lim_{\epsilon \rightarrow 0} \int_{x \in \partial \Omega, \epsilon \leq z \leq 1- \epsilon} \partial_n g_{tg} \cdot h_{tg} \tilde{r} \sqrt{1 + \tilde{r}'^2 / \tilde{z}'^2} d \varphi dz  =0 
$$
and the analogue happens to $\int_{\partial \Omega} \partial_n h_{tg} \cdot g_{tg} dS_x $.

Therefore we can conclude that  $\tilde{\mathcal{S}}^\lambda$ is self-adjoint.

\end{proof}

Next we study the signs of certain operators.

\begin{lemma}
(i) Let $0< \lambda < +\infty$. The operator $\mathcal{A}^\lambda_1$ is negative definite on $L^{2, \tau} (\Omega)$ with domain $\mathcal{X}_0$ and it is $1-1$ from $\mathcal{X}_0$ to $ L^{2, \tau} (\Omega)$. The inverse $ (\mathcal{A}^\lambda_1)^{-1}$ maps $L^{2, \tau} (\Omega)$ into $\mathcal{X}$ with operator bound independent of $\lambda$. Here $\mathcal{X}_0$ is the space consisting of scalar functions in $H^{2, \tau} (\Omega)$ which satisfy the Dirichlet boundary condition.
\vskip 0.1 cm
(ii) For $\lambda$ sufficiently large, $\tilde{\mathcal{S}}^\lambda$ as well as $\tilde{\mathcal{S}}^0$ are negative definite on $L^{2, \tau} (\Omega; \tilde{\mathbb{R}}^2) $ with domain $\tilde{\mathcal{Y}}$.   
\end{lemma}

\begin{proof}
(i) Since $\mathcal{Q}^\pm_\lambda$ is bounded with norm one, we have 
\begin{equation}
\begin{split}
& \sum_\pm  \int_\Omega \int_{\mathbb{R}^3} \mu^\pm_e h (1- \mathcal{Q}^\pm_\lambda) h dv \\
& =  \sum_\pm  \int_\Omega \int_{\mathbb{R}^3} \mu^\pm_e h^2 dv - \sum_\pm  \int_\Omega \int_{\mathbb{R}^3} \mu^\pm_e h  \mathcal{Q}^\pm_\lambda (h) dv   \\
&  \leq 0  \ .
\end{split}
\end{equation}
Therefore $\mathcal{A}^\lambda_1$ is negative definite. Note that $\| \mathcal{Q}^\pm_\lambda (h) \|_{\mathcal{H}^\pm}= \| h \|_{\mathcal{H}^\pm} $ iff $h = const.$,. Hence $\mathcal{A}^\lambda_1 h =0 $ iff $h = const.$. But we have the Dirichlet boundary condition, so $\mathcal{A}^\lambda_1 h =0 $ iff $h = 0$. $\mathcal{A}^\lambda_1$ has discrete spectrum since it is relatively compact with respect to $\Delta$, and it has trivial kernel. Therefore it is invertible.

(ii) For $\tilde{\textbf{h}} = h_r e_r+ h_z e_z  \in \tilde{\mathcal{Y}}$, we have
\begin{equation}
\begin{split}
- \la \tilde{\mathcal{S}}^\lambda \tilde{\textbf{h}} , \tilde{\textbf{h}} \ra_{L^2} 
& = \la (\lambda^2 -\Delta) \tilde{\textbf{h}} , \tilde{\textbf{h}} \ra_{L^2 } \\
&  +\sum_\pm (  \la \mathcal{Q}^\pm_\lambda (\hat{v}_r h_r), \hat{v}_r h_r \ra_{\mathcal{H}^\pm}  + \la \mathcal{Q}^\pm_\lambda (\hat{v}_z h_z), \hat{v}_z h_z \ra_{\mathcal{H}^\pm} -2  \la \mathcal{Q}^\pm_\lambda (\hat{v}_r h_r), \hat{v}_z h_z \ra_{\mathcal{H}^\pm}  )  \ .
\end{split}
\end{equation}
Using the constraints on the space $\tilde{\mathcal{Y}}$, we compute
\begin{equation}
\begin{split}
& \la (\lambda^2 -\Delta) \tilde{\textbf{h}} , \tilde{\textbf{h}} \ra_{L^2 } \\
& =\lambda^2 \|  \tilde{\textbf{h}}  \|^2_{L^2} + \| \nabla  \tilde{\textbf{h}} \|^2_{L^2} - \int_{\partial \Omega} \partial_ n ( h_n e_n + h_{tg} e_{tg}) \cdot  ( h_n e_n + h_{tg} e_{tg}) dS_x \\
& =\lambda^2 \|  \tilde{\textbf{h}}  \|^2_{L^2} + \| \nabla  \tilde{\textbf{h}} \|^2_{L^2} + \int_{\partial \Omega}  \frac{\tilde{z}'}{\tilde{r} \sqrt{\tilde{z}'^2 + \tilde{r}'^2}} h_n^2 dS_x   \ . \\
\end{split}
\end{equation}
Hence we have
\begin{equation}
\begin{split}
- \la \tilde{\mathcal{S}}^\lambda \tilde{\textbf{h}} , \tilde{\textbf{h}} \ra_{L^2} 
& = \lambda^2 \|  \tilde{\textbf{h}}  \|^2_{L^2} + \| \nabla  \tilde{\textbf{h}} \|^2_{L^2} + \int_{\partial \Omega}  \frac{\tilde{z}'}{\tilde{r} \sqrt{\tilde{z}'^2 + \tilde{r}'^2}} h_n^2 dS_x  \\
&  +\sum_\pm (  \la \mathcal{Q}^\pm_\lambda (\hat{v}_r h_r), \hat{v}_r h_r \ra_{\mathcal{H}^\pm}  + \la \mathcal{Q}^\pm_\lambda (\hat{v}_z h_z), \hat{v}_z h_z \ra_{\mathcal{H}^\pm} -2  \la \mathcal{Q}^\pm_\lambda (\hat{v}_r h_r), \hat{v}_z h_z \ra_{\mathcal{H}^\pm}   ) \\
& \geq \lambda^2 \|  \tilde{\textbf{h}}  \|^2_{L^2} + \| \nabla  \tilde{\textbf{h}} \|^2_{L^2} + \int_{\partial \Omega}  \frac{\tilde{z}'}{\tilde{r} \sqrt{\tilde{z}'^2 + \tilde{r}'^2}} h_n^2 dS_x - C  \|  \tilde{\textbf{h}}  \|^2_{L^2}    
\end{split}
\end{equation}
where $C$ only depends on $ \mu^\pm_e$. Therefore if $\lambda >0$ is large enough, $\la \tilde{\mathcal{S}}^\lambda \tilde{\textbf{h}} , \tilde{\textbf{h}} \ra_{L^2} \leq 0$, so that $\tilde{\mathcal{S}}^\lambda$ is negative definite.  

As for $\tilde{\mathcal{S}}^0$, we compute  
\begin{equation}
\begin{split}
- \la \mathcal{S}^0 \tilde{\textbf{h}} , \tilde{\textbf{h}} \ra_{L^2} 
& = \la ( -\Delta) \tilde{\textbf{h}} , \tilde{\textbf{h}} \ra_{L^2 } \\
&  +\sum_\pm (  \la \mathcal{P}^\pm (\hat{v}_r h_r), \hat{v}_r h_r \ra_{\mathcal{H}^\pm}  + \la \mathcal{P}^\pm (\hat{v}_z h_z), \hat{v}_z h_z \ra_{\mathcal{H}^\pm} -2  \la \mathcal{P}^\pm (\hat{v}_r h_r), \hat{v}_z h_z \ra_{\mathcal{H}^\pm}   ) \\
& =  \| \nabla  \tilde{\textbf{h}} \|^2_{L^2} + \int_{\partial \Omega}  \frac{\tilde{z}'}{\tilde{r} \sqrt{\tilde{z}'^2 +\tilde{r}'^2}} h_n^2 dS_x  \\
&  +\sum_\pm (  \| \mathcal{P}^\pm ( \hat{v}_r h_r)\|^2_{\mathcal{H}^\pm}  +  \| \mathcal{P}^\pm ( \hat{v}_z h_z)\|^2_{\mathcal{H}^\pm}  - 2 \la \mathcal{P}^\pm (\hat{v}_r h_r), \mathcal{P}^\pm ( \hat{v}_z h_z )  \ra_{\mathcal{H}^\pm}  ) \\
& = \| \nabla  \tilde{\textbf{h}} \|^2_{L^2} + \int_{\partial \Omega}  \frac{\tilde{z}'}{\tilde{r} \sqrt{\tilde{z}'^2 +\tilde{r}'^2}} h_n^2 dS_x  \\
&  +\sum_\pm  \| \mathcal{P}^\pm ( \hat{v}_r h_r) - \mathcal{P}^\pm ( \hat{v}_z h_z)\|^2_{\mathcal{H}^\pm}  \ . \\
\end{split}
\end{equation}
Here the integral $\int_{\partial \Omega} \frac{\tilde{z}'}{\tilde{r} \sqrt{\tilde{z}'^2 +\tilde{r}'^2}} h_n^2  dS_x$ is finite, since $dS_x = r d\varphi dx_{tg} = \tilde{r} \sqrt{ 1+ \tilde{r}'^2/\tilde{z}'^2} d \varphi dz   $, which eliminates the $1/r$ factor. Also, 
$$
\int_{\partial \Omega} \partial_n h_{tg} \cdot h_{tg} dS_x = \lim_{\epsilon \rightarrow 0} \int_{x \in \partial \Omega, \epsilon \leq z \leq 1- \epsilon} \partial_n h_{tg} \cdot h_{tg} \tilde{r} \sqrt{ 1+ \tilde{r}'^2/\tilde{z}'^2} d \varphi dz  =0
$$
from the Dirichlet boundary condition. Therefore $\tilde{\mathcal{S}}^0$ is also negative definite.  

\end{proof}

We introduce a lemma on the behavior of $\mathcal{Q}^\pm_\lambda$ when $\lambda =0 $ and $\lambda \rightarrow +\infty$. 

\begin{lemma}
Let $\lambda$, $ \mu >0 $, $g \in \mathcal{H}^\pm$, then
\vskip 0.1 cm
(i) $\lim_{\lambda \rightarrow 0^+} \|  \mathcal{Q}^\pm_\lambda g - \mathcal{P}^\pm g  \|_{\mathcal{H}^\pm} =0 $.
\vskip 0.1 cm
(ii) $\lim_{\lambda \rightarrow +\infty} \|  \mathcal{Q}^\pm_\lambda g -  g  \|_{\mathcal{H}^\pm} =0 $.
\vskip 0.1 cm
(iii) $ \|  \mathcal{Q}^\pm_\lambda  - \mathcal{Q}^\pm_\mu   \|_{\mathcal{H}^\pm \rightarrow \mathcal{H}^\pm}  \leq 2 | \log \lambda - \log \mu | $.
\end{lemma} 
 
The proof is identical to the one for Lemma 4.7 in \cite{NS1}.

Recall from \eqref{MaxwellMatrixEqn} that we want to solve the Maxwell equation
\begin{equation}
\left( \begin{array}{ccc}
\mathcal{A}^\lambda_1 & ( \mathcal{B}^\lambda  )^*  &  ( \tilde{\mathcal{T}}^\lambda_1 )^*   \\
 \mathcal{B}^\lambda & \mathcal{A}^\lambda_2 & ( \tilde{\mathcal{T}}^\lambda_2 )^*  \\
 \tilde{\mathcal{T}}^\lambda_1  &  \tilde{\mathcal{T}}^\lambda_2  &  \tilde{\mathcal{S}}^\lambda 
 \end{array} \right)
\left( \begin{array}{ccc}
\phi \\
A_\varphi \\
\tilde{\textbf{A}} 
\end{array} \right)
\begin{array}{ccc}
=0  \ .   \end{array}   
\end{equation}
We can eliminate $\phi$ from the first line of the equation since $\mathcal{A}^\lambda_1$ is invertible:
\begin{equation}
\phi := - (\mathcal{A}^\lambda_1 )^{-1} (  (\mathcal{B}^\lambda)^* A_\varphi + ( \tilde{\mathcal{T}}^\lambda_1  )^* \tilde{\textbf{A}}  )  \ .
\end{equation}
Substituting this into the second and the third line, we obtain the following equation
\begin{equation}     \label{reducedMaxwellMatrixEqn}
\begin{array}{ccc}
\mathcal{M}^\lambda
 \end{array} 
\left( \begin{array}{ccc}
A_\varphi \\
\tilde{\textbf{A}} 
\end{array} \right)
 \begin{array}{ccc}
=
 \end{array} 
\left( \begin{array}{ccc}
\mathcal{L}^\lambda  &  ( \tilde{\mathcal{V}}^\lambda )^*  \\
\tilde{\mathcal{V}}^\lambda  &  \tilde{\mathcal{U}}^\lambda 
\end{array} \right) 
\left( \begin{array}{ccc}
A_\varphi \\
\tilde{\textbf{A}} 
\end{array} \right)
 \begin{array}{ccc}
=0   \end{array} 
\end{equation}
where
\begin{equation}
\begin{split}
& \mathcal{L}^\lambda := \mathcal{A}^\lambda_2 - \mathcal{B}^\lambda (\mathcal{A}^\lambda_1)^{-1} (\mathcal{B}^\lambda)^*   \ , \\
& \tilde{\mathcal{V}}^\lambda := \tilde{\mathcal{T} }^\lambda_2 - \tilde{\mathcal{T} }^\lambda_1 (\mathcal{A}^\lambda_1)^{-1} (\mathcal{B}^\lambda)^* \ ,  \\
&  \tilde{\mathcal{U}}^\lambda  := \tilde{\mathcal{S}}^\lambda - \tilde{\mathcal{T} }^\lambda_1 (\mathcal{A}^\lambda_1)^{-1}  (\tilde{\mathcal{T} }^\lambda_1 )^*   \ .
\end{split}
\end{equation}
$\mathcal{L}^\lambda$ is well-defined from $\mathcal{X} $ to $L^{2, \tau} (\Omega)$ and it is self-adjoint on $L^{2, \tau} (\Omega)$. $\tilde{\mathcal{U}}^\lambda$ is well-defined from $\tilde{\mathcal{Y}}$ to $L^{2, \tau} (\Omega; \tilde{\mathbb{R}}^2)$ and it is self-adjoint on $L^{2, \tau} (\Omega; \tilde{\mathbb{R}}^2)$. $\tilde{\mathcal{V}}^\lambda$ is well-defined from $L^{2, \tau} (\Omega)$ to $L^{2, \tau} (\Omega; \tilde{\mathbb{R}}^2)$. Therefore, $\mathcal{M}^\lambda$ is self-adjoint on $L^{2, \tau} (\Omega) \times L^{2, \tau} (\Omega; \tilde{\mathbb{R}}^2)$ with domain $\mathcal{X} \times \tilde{\mathcal{Y}}$.

We shall study the behavior of these operators when $\lambda$ is close to $0$ and when $\lambda$ is large. With all the ingredients provided in the previous several lemmas, we are able to repeat the process in \cite{NS1} and obtain

\begin{lemma}
Let $0 < \lambda < +\infty$, then
\vskip 0.1 cm
(i) $\mathcal{L}^\lambda \geq 0$ for $\lambda$ sufficiently large. $\forall h \in \mathcal{X} $, $\lim_{\lambda \rightarrow 0^+}  \| (\mathcal{L}^\lambda - \mathcal{L}^0) h \|_{L^2 (\Omega)} =0$. 
\vskip 0.1 cm
(ii) Let $\kappa^\lambda := \inf_{h \in \mathcal{X} , \|h \|_{L^2} =1}  \la \mathcal{L}^\lambda h, h \ra_{L^2 (\Omega)}   $ be the smallest eigenvalue of $\mathcal{L}^\lambda$, then $\kappa^\lambda$ is continuous as a function of $\lambda$ when $\lambda >0$.
\end{lemma}

From the strong convergence of $\mathcal{L}^\lambda$ to $\mathcal{L}^0$, we obtain

\begin{corollary}
If $\mathcal{L}^0$ is not positive definite, then $\exists \lambda_3 >0$, s.t. $\forall 0 \leq \lambda \leq \lambda_3$, $\mathcal{L}^\lambda$ is not positive definite.
\end{corollary}

Next we state some continuity results for the operators. The proof is identical to the one for Lemma 4.9 in \cite{NS1}. 

\begin{lemma}
For fixed $\mu >0$,
\vskip 0.1 cm
(i) $\lim_{\lambda \rightarrow \mu} \| \mathcal{L}^\lambda - \mathcal{L}^\mu \|_{L^2 \rightarrow L^2}  =0 $. Similar results hold for $\tilde{\mathcal{S}}^\lambda $, $\tilde{\mathcal{T}}^\lambda_1 + \lambda \nabla $ and $\tilde{\mathcal{T}}^\lambda_2$.
\vskip 0.1 cm
(ii) $\forall \tilde{\textbf{h}}, \tilde{\textbf{g}} \in \tilde{\mathcal{Y}}$, $\lim_{\lambda \rightarrow \mu} \la (\tilde{\mathcal{U}}^\lambda - \tilde{\mathcal{U}}^\mu) \tilde{\textbf{h}}, \tilde{\textbf{g}}  \ra_{L^2} =0$, $\lim_{\lambda \rightarrow \mu} \| (\tilde{\mathcal{V}}^\lambda - \tilde{\mathcal{V}}^\mu)^*  \tilde{\textbf{h}}  \|_{L^2} =0$. The same thing holds for $\tilde{\mathcal{V}}^0 =0$ with $\mu =0$.
\vskip 0.1 cm
(iii) $\forall h \in \mathcal{X}$, $\tilde{\textbf{g}} \in \tilde{\mathcal{Y}}$, $ \lim_{\lambda \rightarrow +\infty} \la  \tilde{\mathcal{V}}^\lambda h, \tilde{\textbf{g}}  \ra_{L^2} =0$.
\end{lemma}

The next lemma gives a bound on $\tilde{\mathcal{U}}^\lambda$.

\begin{lemma}
(i) $\exists \ 0 < \lambda_1 \leq \lambda_2 <+\infty$, s.t. $\forall \  0< \lambda \leq \lambda_1$ and $\lambda \geq \lambda_2$, the operator $\tilde{\mathcal{U}}^\lambda$ is $1-1$ and onto from $\tilde{\mathcal{Y}}$ to $L^2 (\Omega; \tilde{\mathbb{R}}^2)$. 
\vskip 0.1 cm
(ii) Furthermore, there holds 
\begin{equation}
- \la \tilde{\mathcal{U}}^\lambda \tilde{\textbf{h}}, \tilde{\textbf{h}} \ra_{L^2} \geq C_0 (  \| \tilde{\textbf{h}} \|^2_{L^2} +  \| \nabla \tilde{\textbf{h}} \|^2_{L^2}   ), \  \forall \tilde{\textbf{h}} \in \tilde{\mathcal{Y}} \ .
\end{equation}
Here $C_0$ is some positive constant independent of $\lambda$ within the intervals $0< \lambda \leq \lambda_1$ and $\lambda \geq \lambda_2$.
\end{lemma}

\begin{proof}
We first prove (ii). For all $\tilde{\textbf{h}} \in \tilde{\mathcal{Y}}$, by definition, integrat by parts and use the boundary condition on $\tilde{\textbf{h}}$, we have
\begin{equation}
\begin{split}
& - \la \tilde{\mathcal{U}}^\lambda  \tilde{\textbf{h}}, \tilde{\textbf{h}}   \ra_{L^2}  \\
& = \lambda^2 \| \tilde{\textbf{h}}\|^2_{L^2}  +  \| \nabla  \tilde{\textbf{h}}\|^2_{L^2} + \int_{\partial \Omega} \frac{\tilde{z}'}{\tilde{r} \sqrt{\tilde{z}'^2 + \tilde{r}'^2}} h^2_n d S_x  + \la  (\mathcal{A}^\lambda_1)^{-1} (\tilde{\mathcal{T}}^\lambda_1)^*  \tilde{\textbf{h}},   (\tilde{\mathcal{T}}^\lambda_1)^*  \tilde{\textbf{h}}  \ra_{L^2}  \\
& + \sum_\pm (  \la \mathcal{Q}^\pm_\lambda (\hat{v}_r h_r), \hat{v}_r h_r  \ra_{\mathcal{H}^\pm}  +  \la \mathcal{Q}^\pm_\lambda (\hat{v}_z h_z), \hat{v}_z h_z  \ra_{\mathcal{H}^\pm}  -2 \la \mathcal{Q}^\pm_\lambda (\hat{v}_r h_r), \hat{v}_z h_z  \ra_{\mathcal{H}^\pm} ) \ .
\end{split}
\end{equation}
By the boundedness of $\mathcal{Q}^\pm_\lambda$ and $(\mathcal{A}^\lambda_1)^{-1}$, for large $\lambda$ we have
\begin{equation}
\begin{split}
 - \la \tilde{\mathcal{U}}^\lambda  \tilde{\textbf{h}}, \tilde{\textbf{h}}   \ra_{L^2}  
 &  \geq ( \lambda^2 -C) \| \tilde{\textbf{h}}\|^2_{L^2}  +  \| \nabla  \tilde{\textbf{h}}\|^2_{L^2} + \int_{\partial \Omega}  \frac{\tilde{z}'}{\tilde{r} \sqrt{\tilde{z}'^2 + \tilde{r}'^2}} h^2_n d S_x \\
 &  \geq C_0 (  \| \tilde{\textbf{h}} \|^2_{L^2} +  \| \nabla \tilde{\textbf{h}} \|^2_{L^2}  )  \ .  \\
\end{split}
\end{equation}
For small $\lambda$, we argue by contradiction. Suppose that there are sequences $\lambda_j \rightarrow 0^+$, $ \tilde{\textbf{h}}^j  \in \tilde{\mathcal{Y}} $, s.t. $  \| \tilde{\textbf{h}}^j\|^2_{L^2}  +  \| \nabla  \tilde{\textbf{h}}^j\|^2_{L^2} =1$, but $ \la \tilde{\mathcal{U}}^{\lambda_j} \tilde{\textbf{h}}^j, \tilde{\textbf{h}}^j \ra_{L^2} \rightarrow 0$. 

By extracting subsequence if necessary, $ \tilde{\textbf{h}}^j $ converges weakly to some $\tilde{\textbf{h}}^0$ in $H^1$ and strongly in $L^2$ as $n \rightarrow +\infty$, and $\tilde{\textbf{h}}^0$ satisfies $\tilde{\textbf{h}}^0 \in \tilde{\mathcal{Y}}$ and $ \| \tilde{\textbf{h}}^0\|^2_{L^2}  +  \| \nabla  \tilde{\textbf{h}}^0\|^2_{L^2} =1$. Also, we have $\| (\tilde{\mathcal{T}}^{\lambda_j}_1)^* \tilde{\textbf{h}}^j \|_{L^2} \rightarrow 0$. Therefore $ \la  (\mathcal{A}^{\lambda_j}_1)^{-1} (\tilde{\mathcal{T}}^{\lambda_j}_1)^*  \tilde{\textbf{h}}^j,   (\tilde{\mathcal{T}}^{\lambda_j}_1)^*  \tilde{\textbf{h}}^j  \ra_{L^2} \rightarrow 0$.

Moreover, 
\begin{equation}
\begin{split}
& \sum_\pm (  \la \mathcal{Q}^\pm_{\lambda_j} (\hat{v}_r h^j_{r}), \hat{v}_r h^j_{r}  \ra_{\mathcal{H}^\pm}  +  \la \mathcal{Q}^\pm_{\lambda_j} (\hat{v}_z h^j_{z } ), \hat{v}_z h^j_{z }  \ra_{\mathcal{H}^\pm}  -2 \la \mathcal{Q}^\pm_{\lambda_j} (\hat{v}_r h^j_{r}), \hat{v}_z h^j_{z }  \ra_{\mathcal{H}^\pm}    )  \\
& \rightarrow\sum_\pm ( \| \mathcal{P}^\pm (\hat{v}_r h^0_{r} )\|^2_{\mathcal{H}^\pm} +  \| \mathcal{P}^\pm (\hat{v}_z h^0_{z } )\|^2_{\mathcal{H}^\pm} - 2 \la  \mathcal{P}^\pm (\hat{v}_r h^0_{r} ), \mathcal{P}^\pm (\hat{v}_z h^0_{z } ) \ra_{\mathcal{H}^\pm}  ) \\
& \geq 0  
\end{split}
\end{equation}
as $n \rightarrow +\infty$. Therefore, $ \| \nabla \tilde{\textbf{h}}^0 \|^2_{L^2} +  \int_{\partial \Omega}  \frac{\tilde{z}'}{\tilde{r} \sqrt{\tilde{z}'^2 + \tilde{r}'^2}} (h^0_n)^2 d S_x =0$, which implies $\tilde{\textbf{h}}^0 =0$ by the boundary conditions. This is a contradiction to $ \| \tilde{\textbf{h}}^0\|^2_{L^2}  +  \| \nabla  \tilde{\textbf{h}}^0\|^2_{L^2} =1$.

Now we shall prove (i). From (ii) we know that $\tilde{\mathcal{U}}^\lambda$ is 1-1 from $\tilde{\mathcal{Y}}$ to $L^2 (\Omega ; \tilde{\mathbb{R}}^2)$. We use the Lax-Milgram Theorem to prove this. Let $\tilde{\mathcal{Y}}_2 :=  \{  \tilde{\textbf{h}} \in H^{1, \tau} (\Omega; \tilde{\mathbb{R}^2}) | \nabla \cdot \tilde{\textbf{h}} =0, \forall x \in \Omega ; \tilde{r}  \tilde{z}' \partial_r h_n - \tilde{r} \tilde{r}' \partial_z h_n +  \tilde{r} \tilde{z}'  h_n =0 , h_{tg} =0, \forall x \in \partial \Omega   \}$. 

From the definition of $\tilde{\mathcal{U}}^\lambda$, we introduce the corresponding bilinear form
\begin{equation}
\begin{split}
B^\lambda ( \tilde{\textbf{h}} , \tilde{\textbf{g}})
& := \int_\Omega (\lambda^2 \tilde{\textbf{h}} \cdot \tilde{\textbf{g}} + \nabla \tilde{\textbf{h}} \cdot \nabla \tilde{\textbf{g}}  ) dx + \int_{\partial \Omega}  \frac{\tilde{z}'}{\tilde{r} \sqrt{\tilde{z}'^2 + \tilde{r}'^2}} h_n g_n d S_x \\
& + \sum_\pm \la \mathcal{Q}^\pm_\lambda (\hat{v} \cdot \tilde{\textbf{h}}), \hat{v} \cdot \tilde{\textbf{g}} \ra_{\mathcal{H}^\pm} + \la ( \mathcal{A}^\lambda_1)^{-1} ( \tilde{\mathcal{T}}^\lambda_1)^* \tilde{\textbf{h}},  ( \tilde{\mathcal{T}}^\lambda_1)^* \tilde{\textbf{g}} \ra_{L^2}  \ .
\end{split}
\end{equation}
From (ii), $B^\lambda$ is coercive on $\tilde{\mathcal{Y}}_2 \times \tilde{\mathcal{Y}}_2$ when $\lambda$ is small enough or large enough. Hence by Lax-Milgram there exists an $ \tilde{\textbf{h}} \in \tilde{\mathcal{Y}}_2$ for each $f \in L^2 (\Omega ; \tilde{\mathbb{R}}^2)$ such that $\tilde{\mathcal{U}}^\lambda \tilde{\textbf{h}} =f$. From the definition of $\tilde{\mathcal{U}}^\lambda$, $\Delta \tilde{\textbf{h}} \in L^2 (\Omega; \tilde{\mathbb{R}}^2)$. Thus, $\tilde{\textbf{h}} \in H^2 (\Omega; \tilde{\mathbb{R}}^2) \cap \tilde{\mathcal{Y}}_2 = \tilde{\mathcal{Y}}$. Hence $\tilde{\mathcal{U}}^\lambda$ is 1-1 and onto from $\tilde{\mathcal{Y}}$ to $L^2 (\Omega ; \tilde{\mathbb{R}}^2)$.

\end{proof}

From \eqref{reducedMaxwellMatrixEqn}, our goal is to prove that 
\begin{equation}
\begin{array}{ccc}
\mathcal{M}^\lambda
 \end{array} 
\left( \begin{array}{ccc}
k \\
\tilde{\textbf{h}} 
\end{array} \right)
 \begin{array}{ccc}
=
 \end{array} 
\left( \begin{array}{ccc}
\mathcal{L}^\lambda  &  ( \tilde{\mathcal{V}}^\lambda )^*  \\
\tilde{\mathcal{V}}^\lambda  &  \tilde{\mathcal{U}}^\lambda 
\end{array} \right) 
\left( \begin{array}{ccc}
k \\
\tilde{\textbf{h}} 
\end{array} \right)
 \begin{array}{ccc}
=0   \end{array} 
\end{equation}
admits a nonzero solution in $\mathcal{X} \times \tilde{\mathcal{Y}}$ for some $\lambda > 0$. This is done by repeating the process in \cite{NS1}: first truncate it into a finite-dimensional problem and compare the number of negative eigenvalues for the cases when $\lambda$ is small and large, obtain a solution for the finite-dimensional system, and then pass to the limit, attaining a solution to the infinite-dimensional system, and then naturally recover the growing mode. We omit this part of the proof. For details, see Lemma 4.11 and 4.12 in \cite{NS1}. In the end we have

\begin{lemma} \label{fullequationsolving}
If $\mathcal{L}^0$ is NOT positive definite, then there exists $ \lambda_0 >0$ and a non-zero vector function $ ( k_0 , \tilde{\textbf{h}}_0  ) \in \mathcal{X} \times \tilde{\mathcal{Y}}$ such that 
\begin{equation}
\begin{array}{ccc}
\mathcal{M}^{\lambda_0}
 \end{array} 
\left( \begin{array}{ccc}
k_0 \\
\tilde{\textbf{h}}_0
\end{array} \right)
 \begin{array}{ccc}
=
 \end{array} 
 \begin{array}{ccc}
0  \ . \end{array} 
\end{equation}
\end{lemma}

Now we can recover a growing mode of the linearized Vlasov-Maxwell system with its boundary conditions and prove the instability result.

\begin{theorem}
Assume that $\mathcal{L}^0$ is NOT positive definite, then there exists a growing mode $( e^{\lambda_0 t} f^\pm, e^{\lambda_0 t} \textbf{E}, e^{\lambda_0 t} \textbf{B}   )$ of the linearized Vlasov-Maxwell system. Here $\lambda_0 > 0$, $ f^\pm \in \mathcal{H}^\pm$, $\textbf{E}$, $\textbf{B} \in H^1 (\Omega ; \mathbb{R}^3)$.
\end{theorem}

\begin{proof}
Let $\lambda_0$, $k_0$, $\tilde{\textbf{h}}_0$ be as constructed in the previous lemma. Define
\begin{equation}
\begin{split}
& \phi := - (\mathcal{A}^{\lambda_0}_1)^{-1} ( (\mathcal{B}^{\lambda_0})^* k_0 + (\tilde{\mathcal{T}}^{\lambda_0}_1 )^* \tilde{\textbf{h}}_0  )  \ , \\
& \textbf{A} := \tilde{\textbf{h}}_0 + k_0 e_\varphi  \ ,  \\
& f^\pm (x,v) : =  \pm \mu^\pm_e (1- \mathcal{Q}^\pm_{\lambda_0} ) \phi \pm r   \mu^\pm_p A_\varphi \pm \mu^\pm_e  \mathcal{Q}^\pm_{\lambda_0} (\hat{v} \cdot \textbf{A})  \ , \\
& \textbf{E} := -\nabla \phi - \lambda_0 \textbf{A}  \ , \\
& \textbf{B} := \nabla \times \textbf{A}  \ . \\
\end{split}
\end{equation}
Then by definition $e^{\lambda_0 t} \phi$, $e^{\lambda_0 t}  \textbf{A}$ solves the equations $- \Delta \phi = \rho$ and $ \partial_t^2 \textbf{A} - \Delta \textbf{A} + \partial_t \nabla \phi = \textbf{j}$, $e^{\lambda_0 t} \textbf{E}$, $e^{\lambda_0 t} \textbf{B}$ solves the linearized Maxwell system. Since $\phi$, $A_\varphi \in \mathcal{X}$ and $\tilde{\textbf{A}} \in \tilde{\mathcal{Y}}$, we have $\textbf{E}$, $\textbf{B} \in H^1 (\Omega)$ and they satisfy the specular boundary condition. Also, since $\mathcal{Q}^\pm_{\lambda_0} (g)$ satisfies the specular boundary condition as long as $g$ does, we obtain that $f^\pm $ also satisfies the specular boundary condition.

Now it suffices to check the Vlasov equations for $e^{\lambda_0 t} f^\pm$. We do it for $f^+$ since $f^-$ goes similarly.

Let $g^+ := f^+ -\mu^+_e \phi - r   \mu^+_p A_\varphi $, $h^+ := \mu^+_e (\hat{v} \cdot \textbf{A} - \phi ) $, then the formula for $f^+$ above is equivalent to
\begin{equation}
g^+ = \mathcal{Q}^+_{\lambda_0} h^+  \ .
\end{equation}
We can rewrite the Vlasov equation for $f^+$ as
\begin{equation} \label{VlasovRewritten}
(\lambda_0 + D^+) g^+ = \lambda_0 h^+  
\end{equation}
in the distributional sense.

Recall that for $v= v_r e_r + v_z e_z +v_\varphi e_\varphi$ and $\mathcal{R} v = -v_r e_r - v_z e_z +v_\varphi e_\varphi $, we have $ ( \mathcal{R} g ) (x,v) = g (x, \mathcal{R} v) $. For each test function $k(x,v)$ in $C^1_c (\Omega \times \mathbb{R}^3) $ that satisfies the axisymmetry and the specular condition, using the adjoint formula for $\mathcal{Q}^+_{\lambda_0}$, and the facts that $\mathcal{R}^2 = Id $, $\mathcal{R} D^+ \mathcal{R} = -D^+$, and $\mathcal{Q}^+_{\lambda_0} (\lambda_0 +D^+) k = \lambda_0 k$, we compute
\begin{equation}
\begin{split}
& \la (\lambda_0 + D^+) g^+ , k \ra_{\mathcal{H}^\pm} =\la g^+, (\lambda_0 - D^+) k \ra_{\mathcal{H}^\pm} = \la \mathcal{Q}^+_{\lambda_0} h^+, (\lambda_0 - D^+) k \ra_{\mathcal{H}^\pm} \\
& = \la \mathcal{R} h^+, \mathcal{Q}^+_{\lambda_0} \mathcal{R} (\lambda_0 -D^+) k \ra_{\mathcal{H}^\pm} = \la \mathcal{R} h^+,  \mathcal{Q}^+_{\lambda_0}  (\lambda_0 +D^+) \mathcal{R} k   \ra_{\mathcal{H}^\pm}  \\
& \la \mathcal{R} h^+, \lambda_0 \mathcal{R} k \ra_{\mathcal{H}^\pm} = \la \lambda_0 h^+ , k \ra_{\mathcal{H}^\pm} \ .
\end{split}
\end{equation}
The Vlasov equation \eqref{VlasovRewritten} is verified. This completes the proof of Theorem \ref{mainresult} (iii). 

\end{proof}

\section{Example: Stable Equilibria}

In this section, we construct a family of stable equilibria and prove the first part of Theorem \ref{mainresultexample}. We assume that $\mu^\pm$ satisfies a slightly stronger decay assumption than \eqref{decayassumption}:
\begin{equation}  \label{decayassumptionexamplepart}
\mu^\pm_e (e, p) <  0, \  |\mu^\pm_e (e, p)| + |\mu^\pm_p (e, p)| \leq \frac{C_\mu }{1+ |e|^\gamma} \ , \ \gamma > 4 \ .
\end{equation}

Recall that by definition, for any $h \in \mathcal{X}$:
\begin{equation}
\la \mathcal{L}^0 h, h\ra_{L^2} = \la \mathcal{A}^0_2 h, h\ra_{L^2} - \la  (\mathcal{A}^0_1)^{-1} (\mathcal{B}^0 )^* h, (\mathcal{B}^0 )^* h \ra_{L^2}   \ . 
\end{equation}
Since $\mathcal{A}^0_1$ is negative definite, the second term above is non-negative. We now focus on the first term.

Using the Dirichlet boundary condition, we have
\begin{equation} \label{stableexamplecomputation}
\begin{split}
\la \mathcal{A}^0_2 h, h\ra_{L^2} =
& \int_\Omega ( |\nabla h|^2  +  \frac{1}{r^2  } |h|^2 ) dx \\
& - \sum_\pm \int_\Omega \int_{\mathbb{R}^3} r  \mu^\pm_p  \hat{v}_\varphi |h|^2 dv dx  + \sum_\pm \|\mathcal{P}^\pm (\hat{v}_\varphi h) \|^2_{\mathcal{H}^\pm}  \ . 
\end{split}
\end{equation}
Only the second term on the right side does not have a definite sign, and the others are all non-negative.

The main result of this section is the following result, which gives Theorem \ref{mainresultexample} (i).

\begin{theorem} \label{stableexample}
Let $(\mu^{\pm}, \phi^0, A^0_\varphi )$ be an inhomogeneous equilibrium. Suppose $\mu^\pm$ satisfy
\begin{equation}
p \mu^\pm_p (e, p) \leq 0 \ .
\end{equation}
Then the equilibrium is spectrally stable provided either (i) or (ii) holds.

(i) 
\begin{equation}  \label{stableexamplecondition01}
\sup_x  \big( |A^0_\varphi (x) | r(x) (2+ 2^\gamma | \phi^{0}  (x) |^\gamma ) \big) \leq C \cdot C_\mu^{-1}  c_P^{-1} , \quad \forall x \in \Omega
\end{equation}
where $C = C(\gamma) = ( \frac{8}{3} \pi + \frac{4}{\gamma} \pi^2 )^{-1}  \leq  (2  \int_{\mathbb{R}^3} \la v \ra^{-1} \frac{1}{1+ \la v \ra^\gamma} dv )^{-1} $ is a constant only depends on $\gamma$. The constant $c_P$ is the square of the Poincar\'{e} constant, therefore it depends on the boundary $\partial \Omega$. 

(ii)   
\begin{equation}   \label{stableexamplecondition02}
\sup_x \big(  |A^0_\varphi (x)|  r^3 (x) (2+ 2^\gamma | \phi^{0}  (x) |^\gamma ) \big) \leq  C_\mu^{-1} /2   , \quad \forall x \in \Omega
\end{equation}
\end{theorem}

In order to prove Theorem \ref{stableexample}, we first introduce the following lemma, which will be useful throughout Section 7 and 8.

\begin{lemma} \label{integralbound}
For each $x \in \Omega$, $-\infty < \zeta \leq 1$, there holds
\begin{equation}
 \int_{\mathbb{R}^3} \la v \ra^{\zeta} 
 \frac{1}{ 1+ |\la v \ra \pm \phi^{ 0} (x) |^\gamma} dv  \leq  (2+ 2^\gamma | \phi^{ 0}  (x) | ^\gamma ) \int_{\mathbb{R}^3} \la v \ra^{\zeta}   \frac{1}{ 1+ \la v\ra^\gamma} dv
\end{equation}
\end{lemma}

\begin{proof}
For any $x \in \Omega$, $\zeta \leq 1$, the integral $ \int_{\mathbb{R}^3} \la v \ra^{\zeta}   \frac{1}{ 1+ \la v\ra^\gamma} dv$ is convergent because of \eqref{decayassumptionexamplepart}.
For any $x$, it suffices to prove 
\begin{equation}
1+ \la v \ra^\gamma \leq (2+ 2^\gamma | \phi^{ 0}  (x) | ^\gamma )  (1+ |\la v\ra \pm \phi^{ 0} (x) |^\gamma )  \ . 
\end{equation}
Let $D_1 : = \{ v \in \mathbb{R}^3 :  |\la v\ra \pm \phi^{0} (x) | > \frac{1}{2} \la v \ra  \}$ , $D_2 : =   \{ v \in \mathbb{R}^3 :  |\la v\ra \pm \phi^{0} (x) | \leq  \frac{1}{2} \la v \ra  \} $. We will show that the claim holds true on both sets. Indeed, on $D_1$,
\begin{equation}
1+ \la v \ra^\gamma \leq 2 (1+ |\la v\ra \pm \phi^{0} (x) |^\gamma )   \ .
\end{equation}
On $D_2$, note that $|\la v\ra \pm \phi^{0} (x) | \leq  \frac{1}{2} \la v \ra$ implies $1 \leq \la v \ra \leq  2 | \phi^{0}  (x) |$. Therefore 
\begin{equation}
1 + \la v \ra^\gamma \leq 1+ 2^\gamma | \phi^{0}  (x) |^\gamma  \ .
\end{equation}
The lemma is verified.
\end{proof}

Now we prove Theorem \ref{stableexample}. 

\begin{proof}
To prove (i), it suffices to make sure that the term $- \sum_\pm \int_\Omega \int_{\mathbb{R}^3} r  \mu^\pm_p  \hat{v}_\varphi |h|^2 dv dx $ is controlled by the other terms.
Recalling \eqref{epdefinition}, we estimate 
\begin{equation}  \label{stableexamplecomputation2}
\begin{split}
& - \sum_\pm \int_\Omega \int_{\mathbb{R}^3} r  \mu^\pm_p \hat{v}_\varphi |h|^2 dv dx \\
& = - \sum_\pm \int_\Omega \int_{\mathbb{R}^3} p^\pm \mu^\pm_p (e^\pm, p^\pm) |h|^2 / \la v \ra dv dx + \sum_\pm \int_\Omega \int_{\mathbb{R}^3} r  A^0_\varphi \mu^\pm_p (e^\pm, p^\pm)  |h|^2 / \la v \ra dv dx \\
& \geq  \sum_\pm \int_\Omega \int_{\mathbb{R}^3} r  A^0_\varphi \mu^\pm_p (e^\pm, p^\pm)  |h|^2 / \la v \ra dv dx \\
& \geq -  \sum_\pm  \int_\Omega    |A^0_\varphi| (  \int_{\mathbb{R}^3} \la v \ra^{-1} |\mu^\pm_p (e^\pm, p^\pm) |   dv )   r   |h|^2   dx \\
& \geq - C_\mu  \sum_\pm \int_\Omega   |A^0_\varphi| (  \int_{\mathbb{R}^3} \la v \ra^{-1} \frac{1}{1+| \la v \ra \pm \phi^{ 0} (x)    |^\gamma} dv )   r   |h|^2   dx  \ .  \\
\end{split}
\end{equation}
Using Lemma \ref{integralbound} with $\zeta = -1$, the last line is bounded from below by
$$
- 2 C_\mu \sup_x \big( |A^0_\varphi (x) |  r(x) (2+ 2^\gamma | \phi^{ 0}  (x) |^\gamma )   \big) \big(  \int_{\mathbb{R}^3} \la v \ra^{-1} \frac{1}{1+\la v \ra^\gamma } dv \big) \big( \int_\Omega |h|^2 dx \big)  \ .
$$
Hence by \eqref{stableexamplecondition01},  
$$
- \sum_\pm \int_\Omega \int_{\mathbb{R}^3} r  \mu^\pm_p \hat{v}_\varphi |h|^2 dv dx   \geq  - \|\nabla h \|^2_{L^2}   \ . 
$$
Therefore, by \eqref{stableexamplecomputation2}, $\la \mathcal{A}^0_2 h, h\ra_{L^2} \geq 0$, from which we deduce $\la \mathcal{L}^0 h, h\ra_{L^2} \geq 0$, i.e. the equilibrium is spectrally stable.

The proof of (ii) is identical to the one of (i), except that in the last line of \eqref{stableexamplecomputation2}, we bound the term 
$$ - C_\mu   \sum_\pm \int_{\Omega}    |A^0_\varphi| (  \int_{\mathbb{R}^3} \la v \ra^{-1}  \frac{1}{1+|\la v \ra \pm \phi^{ 0} (x) |^\gamma} dv )   r   |h|^2   dx $$
by $- \int_{\Omega} \frac{1}{r^2} |h|^2 dx$ from below.
\end{proof}

By estimating the Poincar\'{e} constant, we can obtain a detailed description of the condition \eqref{stableexamplecondition01}, and deduce the following lemma. This sheds some light on our understanding of the role of geometry of the boundary on the stability of the equilibrium.

\begin{corollary} \label{stableexampleshape}
Let $L_1 = \sup_{x \in \Omega} z(x) - \inf_{x \in \Omega} z(x) $, $L_2 =\sup_{x \in \Omega} r(x) - \inf_{x \in \Omega} r(x) $. Then the Poincar\'{e} constant on $\Omega$ for $\varphi$-independent functions (with homogeneous Dirichlet boundary condition) is bounded by $ \min \{ \big( \pi^{-2} \sup_{x \in \Omega} r(x) \frac{L_1^2 L_2^2}{L_1^2 + L_2^2} \big)^{1/2} , L_1^{1/2} \}$. Therefore the condition for stability as derived in Theorem \ref{stableexample} can be expressed as
\begin{equation}
p \mu^\pm_p (e, p) \leq 0
\end{equation}
and
\begin{equation}
\sup_x  |A^0_\varphi| \sup_x (r(x) (2+ 2^\gamma | \phi^{0}  (x) |^\gamma ) ) \leq (\frac{8}{3} \pi + \frac{4}{\gamma} \pi^2 )^{-1}   \max \{  \pi^{2}  \inf_{x \in \Omega} r(x)^{-1}   \frac{L^2_1 + L^2_2}{L^2_1  L^2_2}, L^{-1}_1 \}  , \quad \forall x \in \Omega \ .
\end{equation}
\end{corollary}

\begin{proof}  
It suffices to work out the bound for the Poincar\'{e} constant. Take $\Omega_0 :  = \{ x \in \Omega : \varphi (x) =0  \}$. Let $\Omega'_0$ be a rectangle with $\Omega$ inscribed in it, and the length of the side parallel to the $z$-axis being $L_1$, and the length of the side perpendicular to it being $L_2$. For any function defined on $\Omega_0$ with homogeneous Dirichlet boundary condition, we do zero extension for $g$, making it into a function defined on $\Omega'_0$.
We compute, using that the first Dirichlet eigenvalue on the rectangle $\Omega'_0$ is $\pi^2 (\frac{1}{L_1^2} + \frac{1}{L_2^2})$:
\begin{equation}
\begin{split}
\| u \|^2_{L^2 (\Omega)} 
& = \int_{\Omega} |u|^2 dx \\
& = 2 \pi \int_{\Omega_0} |u (r, z)|^2 r dr dz \\
& \leq 2 \pi  \sup_{x \in \Omega} r(x)  \int_{\Omega'_0} |u (r, z)|^2  dr dz \\
& \leq 2 \pi  \sup_{x \in \Omega} r(x)  \pi^{-2} \frac{L_1^2 L_2^2}{ L_1^2 + L_2^2} \| \nabla u (r, z) \|^2_{L^2_{r, z} (\Omega'_0)} \\
& \leq 2 \pi  \sup_{x \in \Omega} r(x)  \pi^{-2} \frac{L_1^2 L_2^2}{ L_1^2 + L_2^2} (2 \pi)^{-1} \| \nabla u (r, z) \|^2_{L^2 (\Omega)} \\
& =  \pi^{-2}  \sup_{x \in \Omega} r(x)  \frac{L_1^2 L_2^2}{L_1^2 + L_2^2}  \| \nabla u (r, z) \|^2_{L^2 (\Omega)}  \ . \\
\end{split}
\end{equation}
On the other hand, we can use $(x_1, x_2, x_3)$ to denote the Cartesian coordinates. Choose the coordinate frame such that $\inf_x x_3 \geq - L_1/2 $, $\sup_x x_3 \leq L_1/2 $. We compute     
$$
u (x_1, x_2, x_3) = \int^{x_3}_{- L_1/2} \partial_{x_3} u (x_1, x_2, x_3) d x'_3 \ .
$$
Hence
$$
| u (x_1, x_2, x_3) | \leq  L^2_1 \|  \nabla u (x_1, x_2, \cdot) \|_{L^2_{x_3}}  \ .
$$
Integrating in $(x_1, x_2, x_3) $, by Cauchy-Schwarz we obtain
\begin{equation}
\| u  (x_1, x_2, x_3) \|_{L^2_\Omega} \leq L_1 \| \nabla u  (x_1, x_2, x_3) \|_{L^2_\Omega}  \ .
\end{equation}
The bound for the Poincar\'{e} constant hence follows. Noticing that $ ( 2 \int_{\mathbb{R}^3} \la v \ra^{-1} \frac{1}{1+ \la v \ra^\gamma} dv )^{-1}  \leq (\frac{8}{3} \pi + \frac{4}{\gamma} \pi^2 )^{-1}$, the result then follows from Theorem \ref{stableexample}.

\end{proof}

\textit{Remark} Based on this theorem, we can begin to discuss the role of the shape of the domain. Suppose we have a thin torus with large radius $a_0$ with respect to the axis of rotation, and the size of the cross-section is small compared to $a_0$, then as $a_0$ gets large, the constraint on $|A^0_\varphi|$ required for stability gets stronger, and hence the instability increases.

\section{Example: Unstable Equilibria}

Now we construct a family of unstable equilibria and prove the second part of Theorem \ref{mainresultexample}. Again we assume that $\mu^\pm$ satisfies the additional decay assumption \eqref{decayassumptionexamplepart}.

Our main task in this section is to prove Theorem \ref{mainresultexample} (ii). Firstly, we find some conditions on the equilibrium that imply instability. The main condition is the strong dependence on the angular momentum. Later in Theorem \ref{unstableexampleexistence}, we prove that such an equilibrium exists under certain circumstances. Here we do not attempt to make the conditions on $\mu^\pm$ or all the constants sharp.

\begin{proposition} \label{unstableexample2}
Denote $b : = \sup_{x \in \Omega} r(x) >1$. Let $\delta$, $\epsilon$ be constants in $(0, 1)$ such that $\delta > \epsilon$, $ 0 < \epsilon + \delta <1$. Let $\mu^\pm$ be such that 
\begin{equation}
p \mu^\pm_p (e, p) \geq C'_\mu |p| \la p  \ra^{-\epsilon} \nu (e) 
\end{equation}
for some positve function $\nu (e)$, with 
\begin{equation}  \label{nucondition}
\nu (e) \geq C_\nu \exp (- e) 
\end{equation}
for some constant $C_\nu >0$, and that $\mu^\pm $ satisfies \eqref{decayassumptionexamplepart}. For each $K \geq 1$, let 
$$
\mu^{K, \pm } (e, p) = \frac{1}{K^\delta} \mu^\pm (e, Kp) \ .
$$
Suppose for some $K  \geq 1$ that $(\phi^{K, 0}, A_\varphi^{K, 0})$ is a pair of solutions to the following coupled system
\begin{equation}   \label{phiK0equation}
- \Delta \phi^{K, 0} = \int_{\mathbb{R}^3} ( \mu^{K, +} (e^{K, +}, p^{K, +}) -  \mu^{K, -} (e^{K, -}, p^{K, -}) ) dv
\end{equation}
\begin{equation}   \label{AK0equation}
(-\Delta +\frac{1}{r^2 })  A_\varphi^{K, 0} = \int_{\mathbb{R}^3} \hat{v}_\varphi ( \mu^{K, + } (e^{K, +}, p^{K, +}) - \mu^{K, - } (e^{K, -}, p^{K, -}) ) dv
\end{equation}  
where 
$$e^{K, \pm} =  \la v \ra \pm \phi^{K, 0} (x)  , \ p^{K, \pm} = r   (\hat{v}_\varphi \pm A_\varphi^{K, 0} (x))  \ , $$
and 
$$ A_\varphi^{K, 0} =0 , \ \phi^{K, 0} =0 \ on \ \partial \Omega \ , $$
such that $  \| \phi^{K, 0} \|_{L^\infty} \leq 1/2$, and $  \| A_\varphi^{K, 0}\|_{L^\infty} \leq  1/2 $. If $K$ is large enough, then the equilibrium $( \mu^{K, \pm } (e, p),  \phi^{K, 0}, A_\varphi^{K, 0})$ is spectrally unstable. 

More precisely, let $h \in \mathcal{X}$ being normalized in such a way that $ \int_\Omega ( |\nabla h|^2 + \frac{1}{r^2 } |h|^2 ) dx  =1 $, and $K$ is so large that 
\begin{equation}  \label{unstableinequality}
1 - H_1 C_1 C'_\mu  K^{1-\delta -\epsilon}  + 120 \cdot  2^{\gamma} \pi^2  b^2  (H_1+H_2) C_\mu^2  K^{1- 2 \delta }  + 2^\gamma H_2 C_2  C_\mu K^{-\delta}  +   256 \pi^2 c_P C_\mu^2  H_2 K^{-2 \delta}   < 0 \ ,
\end{equation} 
holds, where 
$$H_1 =   \int_{r \geq 1}  r  |h|^2 dx \ ,$$
$$H_2 = \|h\|^2_{L^2} \ , $$ 
$$C_1  =    2^{-1-\epsilon/2} b^{-\epsilon}   \ , $$
$$C_2 =  \frac{8}{3} \pi + \frac{4}{\gamma -1} \pi^2 \ ,  $$
and $c_P$ is the square of the Poincar\'{e} constant of $\Omega$, $C_\nu$ is as defined in \eqref{nucondition}. 
Then the equilibrium $( \mu^{K, \pm } (e, p),  \phi^{K, 0}, A_\varphi^{K, 0})$ is spectrally unstable.  
\end{proposition}

\textit{Remark} The condition $p \mu^\pm_p \geq C'_\mu |p|  \la p  \ra^{-\epsilon} \nu (e)  $ together with that $|\mu^\pm_p (e, p)| \leq \frac{C_{\mu}}{1+ |e|^\gamma}$ (assumed throughout all discussion) implies that $|\nu (e) | \leq \frac{  \la p  \ra^{\epsilon}  }{1+ |e|^\gamma} $ with $\gamma >  4 $, which ensures the integrals in the proof are convergent.

Now we discuss the proof of Theorem \ref{unstableexample2}. We replace $\mu^\pm$ in the operators by $\mu^{K, \pm }$ and define
\begin{equation}
D^{K, \pm} = v \cdot \nabla_x \pm (\textbf{E}^{ K, 0} + \hat{v} \times \textbf{B}^{K, 0} ) \cdot \nabla_v  \ . 
\end{equation}
Here $\textbf{E}^{K, 0} = - \nabla \phi^{K, 0}$, $\textbf{B}^{K, 0} = \nabla \times \textbf{A}^{ K, 0} $, $\textbf{A}^{K, 0} = A^{K, 0}_\varphi e_\varphi $, where $A^{K, 0}_\varphi$ is a bounded solution to the equation
\begin{equation}
(-\Delta +\frac{1}{r^2  })  A_\varphi^{K, 0} = \int_{\mathbb{R}^3} \hat{v}_\varphi ( \mu^{K, + } (e^{K, +}, p^{K, +}) - \mu^{K, - } (e^{K, -}, p^{K, -}) ) dv
\end{equation}  
Let $\mathcal{H}^{K, \pm}$ be the space of functions of $x \in \Omega$ and $v \in \mathbb{R}^3$ with the norm
\begin{equation}
\| g (x, v) \|_{\mathcal{H}^{K, \pm}} : = \big( \int_\Omega \int_{\mathbb{R}^3}  | ( \mu^{K, \pm})_e (e^\pm, p^\pm) | g^2 (x,v)   dx dv  \big)^{1/2}
\end{equation}
and $\mathcal{P}^{K, \pm} $ be the orthogonal projection from $\mathcal{H}^{K, \pm} $ onto $ker D^{K, \pm} $.
In analogy with \eqref{A01definition}, \eqref{A02definition} and \eqref{B0definition}, we define    
\begin{equation}
\mathcal{A}^{0, K}_1 h := \Delta h + \sum_\pm \int_{\mathbb{R}^3} ( \mu^{K, \pm} )_e (1- \mathcal{P}^{K, \pm}) h dv
\end{equation}
\begin{equation}
\mathcal{A}^{0, K}_2 h := (-\Delta + \frac{1}{r^2 }) h - \sum_\pm \int_{\mathbb{R}^3} \hat{v}_\varphi \big( (\mu^{ K, \pm} )_p r   h  + ( \mu^{K, \pm})_e  \mathcal{P}^{K, \pm} (\hat{v}_\varphi h )  \big) dv
\end{equation}
\begin{equation}   \label{B0Kdefinition}
\mathcal{B}^{0, K} h := - \sum_\pm \int_{\mathbb{R}^3}  \hat{v}_\varphi (\mu^{ K, \pm})_e (1-\mathcal{P}^{K, \pm}) h dv
\end{equation}
\begin{equation}
\mathcal{L}^{0, K} := \mathcal{A}^{0, K}_2 - \mathcal{B}^{0, K} (\mathcal{A}^{0, K}_1)^{-1} (\mathcal{B}^{0, K} )^*  \ . 
\end{equation}

We have the following lemma analogous to Lemma \ref{A01invertibility}:

\begin{lemma}  \label{A0K1invertibility}
We have
\begin{equation}
\| \mathcal{A}^{0, K}_1 \|_{L^2 \rightarrow L^2} \geq c^{-1}_P \ .
\end{equation}
Hence for all $\tilde{g}  \in \mathcal{X}$, there holds
\begin{equation}
| \la ( \mathcal{A}^{0, K}_1 ) ^{-1} \tilde{g}, \tilde{g}  \ra_{L^2}| \leq c_P \|    \tilde{g} \|^2_{L^2}  \ .
\end{equation}
Here $c_P$ is the square of the Poincar\'{e} constant of $\Omega$.
\end{lemma}
 
The proof is analogous to the one in Lemma \ref{A01invertibility}, so we omit it.

We also need the following

\begin{lemma}  \label{B0Kbound}
For all $K \geq 1$, $\|(\mathcal{B}^{0, K})^*\|_{L^2 \rightarrow L^2} \leq  8 \sqrt{2} \pi C_\mu \frac{1}{K^\delta}  $.
\end{lemma}  

\begin{proof}
Firstly, we have, for each $x \in \Omega$,
\begin{equation}   \label{muebound}
\begin{split}
& \int_{\mathbb{R}^3} \|  (\mu^{K, \pm})_e  \|_{L^\infty_x} dv  \\
&  \leq  \frac{1}{K^\delta} \int_{\mathbb{R}^3} | \mu^\pm_e (e^{K, \pm}, Kp^{K, \pm}) | dv   \\
&   \leq C_\mu \frac{1}{K^\delta} \int_{\mathbb{R}^3} \frac{1}{1+ |e^{K, \pm}|^\gamma} dv  \\ 
&   \leq C_\mu \frac{1}{K^\delta} (2 + 2^\gamma |\phi^{K, 0}|^\gamma) \int_{\mathbb{R}^3} \frac{1}{1+ \la v \ra^\gamma} dv  \\ 
& \leq  3 C_\mu \frac{1}{K^\delta}  \int_{\mathbb{R}^3} \frac{1}{1+ \la v \ra^\gamma} dv  \\
&  \leq  \frac{8 \pi C_\mu}{K^\delta} \ . \\
\end{split}
\end{equation}
by Lemma \ref{integralbound} with $\zeta =0$. Therefore from $0 = \int_{\mathbb{R}^3} \partial_{v_\varphi} \mu^{K, \pm} dv = \int_{\mathbb{R}^3} r  (\mu^{K, \pm} )_p dv + \int_{\mathbb{R}^3} \hat{v}_\varphi (\mu^{K, \pm} )_e dv   $, we obtain
\begin{equation}   \label{mupbound}
| r   \int_{\mathbb{R}^3} (\mu^{K, \pm} )_p (e^{K, \pm}, p^{K, \pm}) dv |  = |\int_{\mathbb{R}^3} \hat{v}_\varphi (\mu^{K, \pm} )_e dv| \leq | \int_{\mathbb{R}^3}  (\mu^{K, \pm} )_e dv|  \leq  \int_{\mathbb{R}^3} \|  (\mu^{K, \pm})_e  \|_{L^\infty_x} dv  \leq  \frac{8 \pi C_\mu}{K^\delta} \ .
\end{equation}

Now, from the definition \eqref{B0Kdefinition} and the decay assumption \eqref{decayassumptionexamplepart}, we have, for all $h \in L^2 (\Omega)$,
\begin{equation}
\begin{split}
\frac{1}{2} \|(\mathcal{B}^{0, K})^* h\|^2_{L^2} 
& \leq \sum_\pm \int_\Omega ( \int_{\mathbb{R}^3}  (\mu^{K, \pm})_p r  h dv  )^2 dx+ \sum_\pm \int_\Omega ( \int_{\mathbb{R}^3}     ( \mu^{K, \pm})_e   \mathcal{P}^{K, \pm} (\hat{v}_\varphi h)  dv  )^2 dx  \\
& \leq \sum_\pm \int_\Omega ( \int_{\mathbb{R}^3}  r   (\mu^{K, \pm})_p  dv  )^2   |h|^2  dx  \\
& + \sum_\pm (   \int_{\mathbb{R}^3} ( \int_\Omega | (\mu^{K, \pm})_e  \mathcal{P}^\pm (\hat{v}_\varphi h ) |^2 dx )^{1/2} dv    )^2 \\
& \leq \sum_\pm \int_\Omega ( \int_{\mathbb{R}^3}  r   (\mu^{K, \pm})_p  dv  )^2  ( \sup_v |h|)^2  dx  \\  
& + \sum_\pm (   \int_{\mathbb{R}^3}  \| (\mu^{K, \pm })_e \|^{1/2}_{L^\infty_x} ( \int_\Omega | (\mu^{K, \pm})_e| |  \mathcal{P}^\pm (\hat{v}_\varphi h ) |^2 dx )^{1/2} dv    )^2   \\
& \leq \sum_\pm  \sup_x |\int_{\mathbb{R}^3}  r  (\mu^{K, \pm})_p dv  |^2   \int_\Omega   |h|^2  dx  \\
& + \sum_\pm    \int_{\mathbb{R}^3} ( \int_\Omega | (\mu^{K, \pm})_e | | \mathcal{P}^\pm (\hat{v}_\varphi h ) |^2 dx ) dv   \int_{\mathbb{R}^3} \|  (\mu^{K, \pm})_e  \|_{L^\infty_x} dv     \\
& = \sum_\pm \big(  \sup_x |\int_{\mathbb{R}^3}  r  (\mu^{K, \pm})_p dv  |^2 \| h \|^2_{L^2}  +  \int_{\mathbb{R}^3} \|  (\mu^{K, \pm})_e  \|_{L^\infty_x} dv   \|\mathcal{P}^\pm (\hat{v}_\varphi h)\|_{\mathcal{H}^{K, \pm}}^2   \big)  \\
& \leq \sum_\pm \big(  \sup_x |\int_{\mathbb{R}^3}  r   (\mu^{K, \pm})_p dv  |^2 \| h \|^2_{L^2}  +    \int_{\mathbb{R}^3} \|  (\mu^{K, \pm})_e  \|_{L^\infty_x} dv   \| (\hat{v}_\varphi h)\|_{\mathcal{H}^{K, \pm}}^2  \big)   \\
& \leq  \sum_\pm \big(  \sup_x |\int_{\mathbb{R}^3}  r   (\mu^{K, \pm})_p dv  |^2 \| h \|^2_{L^2}  +  (  \int_{\mathbb{R}^3} \|  (\mu^{K, \pm})_e  \|_{L^\infty_x} dv )^2   \| h\|_{L^2_x}^2 \big)    \\
& \leq  64 \pi^2  C_\mu^2 \frac{1}{K^{2 \delta}}  \|h \|_{L^2}^2   \ .
\end{split}
\end{equation}
In the last line we used \eqref{mupbound} and \eqref{muebound}. The lemma then follows.
\end{proof}

Using Lemma \ref{A0K1invertibility} and Lemma \ref{B0Kbound}, we obtain 
\begin{lemma}   \label{unstableexampletermA0K1}
For any $h \in \mathcal{X}$, there holds
\begin{equation} 
\begin{split} 
\la (\mathcal{A}^{0, K}_1)^{-1}  (\mathcal{B}^{0, K})^* h,  (\mathcal{B}^{0, K})^* h \ra_{L^2}  
&  \leq  c_P  \|(\mathcal{B}^{0, K})^* h \|^2_{L^2} \\
& \leq 2 c_P  \cdot  128  \pi^2 C_\mu^2 \frac{1}{K^{2 \delta}}   \|h\|^2_{L^2} \\
& = 256 \pi^2 c_P C_\mu^2  H_2  K^{-2 \delta}    \ .
\end{split}
\end{equation}
\end{lemma}

We will also need the following 
\begin{lemma}  \label{potentialAbound}
Suppose for some $K \geq 1$, there exists a pair of solution $(\phi^{K, 0}, A_\varphi^{K, 0})$ to the coupled system \eqref{phiK0equation} ~ \eqref{AK0equation}, and we have $\| \phi^{K, 0} \|_{L^\infty} \leq 1/2$, $\|A^{K, 0}_\varphi \|_{L^\infty} \leq 1/2$. Then there holds furthermore that 
\begin{equation} \label{potentialphiboundeqn}
\|\phi^{K, 0} \|_{L^\infty}  \leq \frac{20 \pi b^2 C_\mu}{K^\delta} \ ,
\end{equation} 
\begin{equation}  \label{potentialAboundeqn}
\|A^{K, 0}_\varphi \|_{L^\infty}  \leq \frac{20 \pi b^2 C_\mu}{K^\delta} \ . 
\end{equation}
\end{lemma}  

\begin{proof}
$\phi^{K, 0} $ and $A^{K, 0}_\varphi $ satisfies the elliptic equations \eqref{AK0equation} and \eqref{phiK0equation}. Using the assumption that $\| \phi^{K, 0} \|_{L^\infty} $ is bounded by $\frac{1}{2}$ as well as Lemma \ref{integralbound} (applied in ths cases $\zeta = 1$ and $\zeta = 0$), we obtain
\begin{equation}  \label{potentialAboundeqncomp}
\begin{split}
|(-\Delta +\frac{1}{r^2 })  A_\varphi^{K, 0} | 
& \leq 2 C_\mu  \frac{1}{K^\delta} \int_{\mathbb{R}^3} \frac{|p|}{1+|e|^\gamma} dv  \\
&  \leq 2 C_\mu  \frac{1}{K^\delta} (2 + 2^\gamma |\phi^{K, 0} (x)|^\gamma) \Big(  \int_{\mathbb{R}^3} \frac{|v_\varphi|}{1+\la v \ra^\gamma} dv +  \int_{\mathbb{R}^3} \frac{|A^{K, 0}_\varphi (x)|}{1+\la v \ra^\gamma} dv \Big)  \\
&  \leq 6 C_\mu  \frac{1}{K^\delta}  \Big(  \int_{\mathbb{R}^3} \frac{|v_\varphi|}{1+\la v \ra^\gamma} dv +  \int_{\mathbb{R}^3} \frac{1}{1+\la v \ra^\gamma} dv \Big)  \\
&  \leq 6 C_\mu  \frac{1}{K^\delta}  \Big(  \int_{|v| \geq 1} |v|^{-\gamma +1} dv +  \int_{|v| \geq 1} |v|^{-\gamma} dv + 2 \int_{|v| <1} 1 dv \Big)  \\
&  \leq 6 C_\mu  \frac{1}{K^\delta}  \Big(  \int_{|v| \geq 1} |v|^{-\gamma +1} dv +  \int_{|v| \geq 1} |v|^{-\gamma} dv + 2 \int_{|v| <1} 1 dv \Big)  \\
&  \leq  78 \pi C_\mu  \frac{1}{K^\delta} \ , \\
\end{split}
\end{equation}

Since $(-\Delta +\frac{1}{r^2 })  A_\varphi^{K, 0} = - \exp{( - i \varphi )} \Delta (A_\varphi^{K, 0} \exp{( i \varphi ) } ) $, we set $u = K^\delta  A_\varphi^{K, 0} \exp{( i \varphi )}$, and then for \eqref{potentialAboundeqn} it suffices to show that $ \sup_{x \in \Omega} | u (x) | \leq 20 \pi b^2 C_\mu$. In fact, mutiplying \eqref{potentialAboundeqncomp} by $K^\delta$, we obtain
$$
|- \Delta u| \leq 78 \pi C_\mu
$$
Recall $b = \sup_{x \in \Omega} r(x)$. Let $w_1 =  - r^2 \cdot \frac{78 \pi C_\mu}{4} + 20 \pi b^2 C_\mu$, so that $w_1 \geq 0$ on $\partial \Omega$. Then 
$$
- \Delta (u - w_1) \leq 0  \ .
$$
By the maximum principle, $u - w_1 \leq 0$, $u \leq w_1 \leq \sup_{x \in \Omega} w_1 (x)$.

On the other hand, let $w_2 =   r^2 \cdot \frac{78 \pi C_\mu}{4} - 20 \pi b^2 C_\mu$, such that $w_2 \leq 0$ on $\partial \Omega$. Then 
$$
- \Delta (u - w_2) \geq 0  \ .
$$
By the maximum principle, $u - w_2 \geq 0$, $u \geq w_2 \geq \inf_{x \in \Omega} w_2 (x)$.

Hence $ \sup_{x \in \Omega} | u (x) | \leq \max \{   \sup_{x \in \Omega} w_1 (x),  - \inf_{x \in \Omega} w_2 (x) \} \leq 20 \pi b^2 C_\mu$. Therefore $\|A^{K, 0}_\varphi \|_{L^\infty} \leq \frac{20 \pi b^2 C_\mu}{K^\delta} $, and similarly $\|\phi^{K, 0} \|_{L^\infty} \leq \frac{20 \pi b^2 C_\mu}{K^\delta} $ from \eqref{phiK0equation}. 

\end{proof}

Now we are ready to prove Proposition \ref{unstableexample2}.

\begin{proof}
For $h \in \mathcal{X}$, we first take care of the term $\la \mathcal{A}^{0, K}_2 h, h\ra_{L^2} $. For simplicity, we normalize $h$ such that $ \int_\Omega ( |\nabla h|^2 + \frac{1}{r^2 } |h|^2 ) dx  =1 $. Recall that $e^{K, \pm} =  \la v \ra \pm \phi^{K, 0} (x)  , \ p^{K, \pm} = r   (\hat{v}_\varphi \pm A_\varphi^{K, 0} (x))  $. 

We observe 
\begin{equation}
\begin{split}
\la \mathcal{A}^{0, K}_2 h, h \ra_{L^2} 
& = 1 - \sum_\pm \int_\Omega \int_{\mathbb{R}^3} \la v \ra^{-1} K \frac{1}{K^\delta} p^{K, \pm} \mu^\pm_{p} (e^{K, \pm}, Kp^{K, \pm})  |h|^2 dv dx  \\
&  + \sum_\pm  \int_\Omega \int_{\mathbb{R}^3} K \frac{1}{K^\delta} \frac{ \mu^\pm_{p} (e^{K, \pm}, K p^{K, \pm})}{\la v \ra} r A^{K, 0}_\varphi |h|^2 dv dx + \sum_\pm \| \mathcal{P}^{K, \pm} (\hat{v}_\varphi h)  \|^2_{\mathcal{H}^{K, \pm}} \\
& := 1+ I+II+III  \ .    \\
\end{split}
\end{equation}
We are going to show that when $K \geq 1 $ is large enough, the term $I$ dominates all the others and remains negative. We compute 
\begin{equation}
\begin{split}
I 
& = - \sum_\pm \int_\Omega \int_{\mathbb{R}^3} \la v \ra^{-1} K \frac{1}{K^\delta} p^{K, \pm} \mu^\pm_{p} (e^{K, \pm}, Kp^{K, \pm}) dv \  |h|^2 dx \\
& \leq - C'_\mu \sum_\pm \int_\Omega \int_{\mathbb{R}^3} \la v \ra^{-1} K^{1-\epsilon} \frac{1}{K^\delta} |p^{K, \pm}| \la p^{K, \pm} \ra^{-\epsilon} \nu (e^{K, \pm})  dv \ |h|^2 dx \\
& \leq -  C'_\mu  K^{1- \delta - \epsilon}  \sum_\pm \int_\Omega \int_{ | p^{K, \pm} (x, v) | >1}   \la v \ra^{-1}   | p^{K, \pm} |^{1-\epsilon}  \cdot 2^{-\epsilon/2}  \nu (e^{K, \pm}) dv  \  |h|^2 dx \\
& \leq - C'_\mu   K^{1- \delta -\epsilon }  M    \ , \\
\end{split}
\end{equation}
where
$$ M = 2^{-\epsilon/2}  \sum_\pm \int_{ r(x) \geq 1}   \inf_{x \in \Omega, r(x) \geq 1}  \Big[   \int_{ \{ v \in \mathbb{R}^3,  | p^{K, \pm} (x, v) | >1 \}} \la v \ra^{-1} |v_\varphi  \mp A^{K, 0}_\varphi (r, z) |^{1-\epsilon} \nu (e^{K, \pm})  dv \Big]  r^{1- \epsilon}  |h|^2 dx    \ . $$
Noting that $\|A^{K, 0}_\varphi \|_{L^\infty} \leq 1/2$ and $\|\phi^{K, 0} \|_{L^\infty} \leq 1/2$, we have
\begin{equation}
\begin{split}
M 
& =  2^{-\epsilon/2}   \sum_\pm \int_{ r(x) \geq 1}  \inf_{x \in \Omega, r(x) \geq 1} \Big[  \int_{ \{ v \in \mathbb{R}^3,  | p^{K, \pm} (x, v) | >1 \}} \la v \ra^{-1} |v_\varphi  \mp A^{K, 0}_\varphi (r, z) |^{1-\epsilon}  \nu (e^{K, \pm})  dv \Big]  r^{1- \epsilon}  |h|^2 dx      \\
& \geq  2^{-\epsilon/2}  \sum_\pm \int_{ r(x) \geq 1}   \inf_{x \in \Omega, r(x) \geq 1} \Big[  \int_{ \{ v \in \mathbb{R}^3,  |v_\varphi | > 2 + \frac{1}{r(x)} \}} \la v \ra^{-1} (1 +\frac{1}{r(x)})^{1-\epsilon}  \nu (  \la v \ra \pm \phi^{K, 0} )  dv \Big]   r^{1- \epsilon}  |h|^2 dx     \\
& \geq  2^{-\epsilon/2}  \sum_\pm \int_{  r(x) \geq 1}   \inf_{x \in \Omega, r(x) \geq 1} \Big[  \int_{ \{ v \in \mathbb{R}^3,  |v_\varphi | > 3 \}} \la v \ra^{-1}    \nu (  \la v \ra \pm \phi^{K, 0} )  dv \Big]   r^{1- \epsilon}  |h|^2 dx     \\
& \geq  2^{-\epsilon/2}  C_\nu   \sum_\pm \int_{  r(x) \geq 1}  \inf_{x \in \Omega, r(x) \geq 1}  \Big[  \int_{ \{ v \in \mathbb{R}^3,  |v_\varphi | > 3 \}} \la v \ra^{-1}   \exp (- \la v \ra \mp \phi^{K, 0} )  dv \Big]   r^{1- \epsilon}  |h|^2 dx      \\
& \geq  2^{-\epsilon/2}  C_\nu  \cdot 2 \exp(- \frac{1}{2})  \int_{ r(x) \geq 1}   \int_{ \{ v \in \mathbb{R}^3,  |v_\varphi | > 3 \}} \la v \ra^{-1}  \exp (- \la v \ra  )  dv   r^{1- \epsilon}  |h|^2 dx     \\
& \geq  2^{-\epsilon/2}  C_\nu  \cdot \frac{1}{2}    \int_{  r(x) \geq 1}     r^{1- \epsilon}  |h|^2 dx  \\
& \geq   2^{-\epsilon/2}  \cdot \frac{1}{2} b^{-\epsilon}  C_\nu   \int_{  r(x) \geq 1}     r   |h|^2 dx  \ , \\
\end{split}
\end{equation}
where $b = \sup_{x \in \Omega} r(x) >1$. Hence 
$$I \leq    - H_1 C_1 C_\nu C'_\mu K^{1- \delta -\epsilon}  \ , $$
where $H_1 =    \int_{r \geq 1}     r   |h|^2 dx $, $C_1 =      2^{-1-\epsilon/2}  b^{-\epsilon}  $. From Remark 1, we have $|\nu (e^{K, \pm}) | \leq \frac{  \la p^{K, \pm}  \ra^{\epsilon}  }{C(1+ |e^{K, \pm}|^\gamma)} $ with $\gamma >  4 $, which ensures that the integral $ \int_{\mathbb{R}^3} \la v \ra^{-1} v_\varphi^2  \nu (e^{K, \pm})  dv $ is convergent. In the end we obtain $I = O (K^{1-\delta -\epsilon})$.

$II$ vanishes if $A^{K, 0}_\varphi = 0$. If $A^{K, 0}_\varphi \neq 0$, we estimate, using Lemma \ref{potentialAbound},
\begin{equation}
\begin{split}
II 
& \leq 20 \pi b^2 C_\mu  (H_1 + H_2)   K^{1-\delta }  \frac{1}{K^\delta}    \sum_\pm    \sup_{x } ( \int_{\mathbb{R}^3} \la v \ra^{-1} |\mu^\pm_{p} (e^{K, \pm}, K p^{K, \pm})| dv) \\
& \leq  20 \pi b^2 C_\mu   (H_1 + H_2)    C_\mu K^{1-2 \delta  }    \sum_\pm  \sup_x     \int_{\mathbb{R}^3} \frac{1}{\la v \ra (1+ |e^{K, \pm}|^\gamma)} dv \\
& \leq 20 \pi b^2 C_\mu   (H_1 + H_2)     C_\mu K^{1-2 \delta  }   \sum_\pm     \sup_x \int_{\mathbb{R}^3} \frac{1}{\la v \ra (1+ |\la v\ra \pm \phi^{K, 0} (x)|^\gamma)} dv \\
& \leq 40 \cdot 2^{\gamma} \pi b^2   (H_1 + H_2)   C_\mu^2  K^{1-2 \delta } \int_{\mathbb{R}^3} \frac{1}{\la v \ra (1+ \la v\ra^\gamma)} dv \\
& \leq   40 \cdot  2^{\gamma} \pi  b^2   (H_1 + H_2)   C_\mu^2  K^{1- 2 \delta } \cdot 3 \pi       \\
& \leq   120  \cdot  2^{\gamma} \pi^2  b^2   (H_1 + H_2)  C_\mu^2  K^{1- 2 \delta }    \ .   \\
\end{split}
\end{equation}
Here $H_2 =  \int_\Omega  |h|^2 dx =  \|h \|^2_{L^2} $. Therefore the term $II$ is $O(K^{1- 2 \delta})$.

For $III= \sum_\pm \|\mathcal{P}^{K, \pm}(\hat{v}_\varphi h)\|^2_{\mathcal{H}^{K, \pm}}$, we estimate
\begin{equation}
\begin{split}
III = \|\mathcal{P}^{K}(\hat{v}_\varphi h)\|^2_{\mathcal{H}^{K, \pm}} 
& \leq   \frac{1}{K^\delta} \sum_\pm  \int_\Omega \sup_v |\mathcal{P}^{K, \pm} (\hat{v}_\varphi h)|^2 dx  \sup_x (\int_{\mathbb{R}^3} | \mu^\pm_e (e^{K, \pm}, Kp^{K, \pm})| dv)  \\
& \leq      \frac{1}{K^\delta}  H_2 \sum_\pm   \sup_x (\int_{\mathbb{R}^3} | \mu^\pm_e (e^{K, \pm}, Kp^{K, \pm})| dv)  \\
& \leq      \frac{1}{K^\delta}  H_2  \sum_\pm   \sup_x (\int_{\mathbb{R}^3} \frac{C_{\mu}}{1+ | \la v \ra \pm \phi^{K, 0} (x) |^\gamma} dv)  \\
& \leq 2^\gamma   H_2 C_\mu \cdot 2  \frac{1}{K^\delta} \int_{\mathbb{R}^3} \frac{1}{1+ \la v \ra^\gamma} dv \\
& \leq  2^\gamma  H_2 C_2 C_\mu  \frac{1}{K^\delta}  \ .  \\
\end{split}
\end{equation}  
Here $C_2 =  \frac{8}{3} \pi + \frac{4}{\gamma -1} \pi^2  $. Therefore $III = O (K^{-\delta})$.

Combining together all the estimates, as well as Lemma \ref{unstableexampletermA0K1}, we conclude 
\begin{equation} \label{unstableinequality0}
\begin{split}
& \la \mathcal{L}^{0, K} h, h \ra_{L^2} \\
& \leq 1 - H_1 C_1 C_\nu C'_\mu  K^{1-\delta -\epsilon}  + 120 \cdot  2^{\gamma} \pi^2  b^2  (H_1 +H_2) C_\mu^2  K^{1- 2 \delta }  + 2^\gamma H_2 C_2  C_\mu K^{-\delta} \\
& +   256 \pi^2 c_P C_\mu^2  H_2 K^{-2 \delta}    \ . \\
\end{split}
\end{equation}
Hence $( \mu^{K, \pm } ,  \phi^{K, 0}, A_\varphi^{K, 0})$ is spectrally unstable if $C_\mu$, $K$ and $h$ satisfies \eqref{unstableinequality}.

\end{proof}

The theorem below ensures that the unstable equilibria constructed in Proposition \ref{unstableexample2} do exist under certain circumstances when $\mu^\pm$ satisfy some smallness assumption, which gives Theorem \ref{mainresultexample} (ii). This result is obtained by solving a coupled elliptic system. There is a possibility to extend this existence result to more general cases without such assumption.

\begin{theorem}  \label{unstableexampleexistence}
Suppose that $\Omega$ satisfies $ b = \sup_{x \in \Omega} r(x) >1$ and the constant $C_\mu$ satisfies a certain smallness condition. Then there exists an unstable equilibrium as constructed in Theorem \ref{unstableexample2}.
\end{theorem}

\begin{proof}
Firstly we show that if $C_\mu K \leq \frac{1}{2C}$ for some constant $C$ which only depends on $\gamma$ and $\Omega$, then there exists a solution to the elliptic system \eqref{AK0equation}, \eqref{phiK0equation}, with the homogeneous Dirichlet boundary conditions as well as $  \| \phi^{K, 0} \|_{L^\infty} \leq 1/2$ and $  \| A_\varphi^{K, 0}\|_{L^\infty} \leq  1/2 $. For this, we adapt the idea from Appendix C in \cite{NS1}. 

For any fixed $0 < \alpha <1$, let $S := \{ (\phi, A_\varphi) \in C^\alpha (\overline{\Omega}) \times C^\alpha (\overline{\Omega}) : \sup_{x \in \Omega} |\phi (x)| \leq 1/2, \  \sup_{x \in \Omega} |A_\varphi (x) | \leq 1/2 \}$ furnished with the norm $ \|(\phi, A_\varphi) \|_{S} = \| \phi \|_{C^\alpha} + \|  A_\varphi  \|_{C^\alpha}$. Let $K \geq 1$. Denote the right hand side of the system by $F = (F_1, F_2)$:
\begin{equation}
F_1 (x, \phi, A_\varphi ) : = \int_{\mathbb{R}^3} ( \mu^{K, +} (e^{ +}, p^{+}) - \mu^{K, -} (e^{-} , p^{ -} )  ) dv \ , 
\end{equation}
\begin{equation}
F_2 (x, \phi, A_\varphi ) : = \int_{\mathbb{R}^3} \hat{v}_\varphi ( \mu^{K, +} (e^{ +}, p^{+}) - \mu^{K, -} (e^{-} , p^{ -} ) ) dv \ ,
\end{equation}
where $e^\pm$ and $p^\pm$ are defined as
$$e^\pm=  \la v \ra \pm \phi (x)  , \ p^\pm = r   ( v_\varphi \pm A_\varphi (x))  \ , $$
for any $(\phi, A_\varphi) \in S$.

We have
\begin{equation}
\begin{split}
& | F_1 (x, \phi (x),  A_\varphi (x) )  - F_1 (y, \phi (y), A_\varphi (y) )   |  \\
& \leq \sum_\pm \int_{\mathbb{R}^3}  |  \mu^{K, \pm} (e^{ \pm} (x, v), p^{ \pm} (x, v)) - \mu^{K, \pm} (e^{\pm} (y, v), p^{  \pm} (y, v) )   | dv \\
& \leq C_\mu \sum_\pm \int_{\mathbb{R}^3}  \frac{1}{1+ |\bar{e}^\pm (x, y, v)|^\gamma  }  \{  |\phi (x) - \phi (y) |  + K | r(x) v_\varphi -r(y) v_\varphi |  \\
& + K| r(x) A_\varphi (x) - r(y) A_\varphi (y)  | \} dv \ ,  \\
\end{split}
\end{equation}
where $ |\bar{e}^\pm (x, y, v) |= \min \{ | e^{ \pm} (x,v) | , | e^{\pm} (y, v) | \}$. The same bound is valid for $F_2$. 
For any $(\phi , A_\varphi) \in S$, there holds $|\bar{e}^\pm (x, y, v)| \geq \la v \ra - 1/2 \geq 1/2 $. Thus 
\begin{equation}
\| F(\cdot, \phi, A_\varphi ) \|_S \leq  \frac{8 \pi b}{\gamma} \cdot C_\mu K  \big(  1 + \| ( \phi, A_\varphi ) \|_S  \big)
\end{equation} 
Hence if $C_\mu K  \leq \frac{\gamma}{16 \pi b} $, then $F$ maps $S$ into $S$. 

For each $(\phi, A_\varphi  ) \in S$, standard elliptic theory implies that there exists a unique solution $(\phi^{1}, A^{1}_\varphi ) \in C^{2+ \alpha} (\bar{\Omega})  \times  C^{2+ \alpha} (\bar{\Omega})$ of the linear problem
\begin{equation}
- \Delta \phi^{ 1} = \int_{\mathbb{R}^3} ( \mu^{K, +} (e^{+}, p^{+}) -  \mu^{K, -} (e^{-}, p^{-}) ) dv
\end{equation}
\begin{equation}
(-\Delta +\frac{1}{r^2 })  A_\varphi^{ 1} = \int_{\mathbb{R}^3} \hat{v}_\varphi ( \mu^{K, + } (e^{+}, p^{ +}) - \mu^{K, - } (e^{ -}, p^{-}) ) dv
\end{equation}  
(note that $(-\Delta +\frac{1}{r^2 })  A_\varphi^{1} = - \exp{( - i \varphi )} \Delta (A_\varphi^{1} \exp{( i \varphi ) } ) $)
with homogeneous Dirichlet boundary conditions. Moreover,
\begin{equation}
\| (\phi^{1}, A^{1}_\varphi)  \|_{C^{\alpha +2} } \leq C_0 \cdot C_\mu K \| (F_1 ( \cdot , \phi, A_\varphi ), F_2 (  \cdot , \phi, A_\varphi  ) )  \|_{C^\alpha} \ .
\end{equation}
for a fixed constant $C_0$ which only depends on $\gamma$ and $\Omega$.

Now, we define $\mathcal{T} (\phi, A_\varphi  ) = (\phi^{1}, A^{1}_\varphi   ) $, then $\mathcal{T}$ maps $S$ into itself. Similarly, for any $(\phi^2, A_\varphi^2  ) \in S$, $(\phi^3, A_\varphi^3  ) \in S$,  there holds
\begin{equation}
\| \mathcal{T} (  \phi^2, A^2_\varphi ) - \mathcal{T} (  \phi^3, A^3_\varphi  )  \|_{S} \leq C_0 \cdot C_\mu K \|  (  \phi^2, A^2_\varphi ) -  (  \phi^3, A^3_\varphi  )   \|_{S} 
\end{equation}
If $C_\mu K   \leq \min \{ 1/(2 C_0), \frac{\gamma}{16 \pi b}  \} $, we can see that $\mathcal{T}$ is a contraction, and therefore it has a unique fixed point $(\phi^{K, 0}, A^{K, 0}_\varphi)$ in $S$, which is the solution pair we want.

To construct an unstable equilibrium, it suffices to find $C_\mu$, $K$ and $h \in \mathcal{X}$, such that both $C_\mu K   \leq \min \{ 1/(2 C_0), \frac{\gamma}{16 \pi b}  \} $ and \eqref{unstableinequality} hold. In order to achieve this, we first pick $K \geq  1$, such that 
\begin{equation}  \label{unstableinequality2}
\begin{split}
& K^{\delta - \epsilon} >  2000 \cdot 2^\gamma  \pi^2 b^2  C_1^{-1} C_\nu^{-1}    \ , \\
& K^{1-\epsilon} >  8 \cdot 2^\gamma C_1^{-1} C_2  C_\nu^{-1}   \ , \\
& K^{1 + \delta - \epsilon} >  2048 \pi^2   c_P  C_1^{-1} C_\nu^{-1}     \ . \\
\end{split}
\end{equation}
We pick $C_\mu \in (0, 1)$ and $C'_\mu \geq \frac{1}{2} C_\mu$ such that
$$C_\mu K   \leq \min \{ 1/(2 C_0), \frac{\gamma}{16 \pi b}  \} \ . $$
Since $\sup_{x \in \Omega} r(x) >1$, there exists $h \in \mathcal{X}$, with $ \int_\Omega ( |\nabla h|^2 + \frac{1}{r^2 } |h|^2 ) dx  =1 $, $H_2 = \int_\Omega |h|^2 dx \leq 1$,  and $H_1 = \int_{r \geq 1} r |h|^2 dx > 1$ large enough (hence $H_1 +H_2 \leq 2 H_1$), such that
\begin{equation} \label{unstableinequality3}
\begin{split}
&  K^{1 -\delta - \epsilon} > 4 C_1^{-1}  C_\nu^{-1} (C'_\mu )^{-1} H_1^{-1}  \ . \\
\end{split}
\end{equation}
\eqref{unstableinequality2} and \eqref{unstableinequality3} together with the conditions on $h$ and $ C_\mu \in (0, 1)$, $C'_\mu \geq \frac{1}{2} C_\mu$ imply \eqref{unstableinequality}.
Hence the corresponding equilibrium $( \mu^{K, \pm } (e, p),  \phi^{K, 0}, A_\varphi^{K, 0})$ is spectrally unstable.  
\end{proof}

\section{Appendix A}

For the readers' convenience, we provide the information of the derivatives under the cylindrical coordinates. Using $(x_1, x_2, x_3)$ to denote the Cartesian coordinates and $(r, \varphi, z)$ to be the cylindrical coordinates, we have
$$  
x_1 = r \cos \varphi, \qquad x_2 = r \sin \varphi, \qquad  x_3 = z 
$$
and therefore
$$
e_r = (\cos \varphi, \sin \varphi , 0), \qquad e_\varphi = (- \sin \varphi , \cos \varphi, 0), \qquad e_z =(0, 0, 1) \ .
$$
Hence for any function $f(r, \varphi, z)$ and any vector field $ \textbf{A} (r, \varphi, z)$, we have
$$
\nabla_x f = \frac{\partial f}{\partial r} e_r + \frac{1}{r} \frac{\partial f}{\partial \varphi} e_\varphi + \frac{\partial f }{\partial z } e_z \ ,
$$
$$
\Delta_x f = \frac{1}{r} \frac{\partial}{\partial r} ( r \frac{\partial f}{\partial r}) + \frac{1}{r^2} \frac{\partial^2 f}{\partial \varphi^2} + \frac{\partial^2 f}{\partial z^2}   \ ,
$$
$$
\nabla_x \cdot \textbf{A} = \frac{1}{r} \frac{\partial (r A_r)}{\partial r}  + \frac{1}{r} \frac{\partial A_\varphi}{\partial \varphi}  + \frac{\partial A_z }{\partial z }  \ ,
$$
$$
\nabla_x \times \textbf{A} = (\frac{1}{r} \frac{\partial A_z}{\partial \varphi} - \frac{\partial A_\varphi}{\partial z} ) e_r +( \frac{\partial A_r}{\partial z} - \frac{\partial A_z}{\partial r} ) e_\varphi + \frac{1}{r} ( \frac{\partial (r A_\varphi) }{\partial r } - \frac{\partial A_r}{\partial \varphi} ) e_z \ .
$$

\section{Appendix B}

In this section, we prove the invariance of $e^\pm (x,v) = \la v \ra \pm \phi^0 (r, z)  $ and $ p^\pm (x,v) = r  ( v_\varphi  \pm A^0_\varphi (r, z))  $ along the particle trajectories:
\begin{equation}
\dot{X}^\pm = \hat{V}^\pm,  \  \dot{V}^\pm =  \pm \textbf{E}^0 (X^\pm) \pm \hat{V}^\pm \times \textbf{B}^0 (X^\pm) \ .
\end{equation}

We only compute the $+$ case for simplicity. The $-$ case is similar. In fact, along the particle trajectories, we have 
\begin{equation}
\begin{split}
\dot{e}
&  = \hat{V} \cdot \dot{V} + \dot{X} \cdot \nabla \phi^0 \\
& =   \hat{V} \cdot (\textbf{E}^0 + \hat{V} \times \textbf{B}^0 )  + \hat{V} \cdot \nabla \phi^0  \\
& =     \hat{V} \cdot (-\nabla \phi^0 + \hat{V}  \times \textbf{B}^0 )  + \hat{V}\cdot \nabla \phi^0 \\
& =0  \ ,
\end{split}
\end{equation}
\begin{equation}
\begin{split}
\dot{p} 
& = \dot{X} \cdot \nabla_X p + \dot{V} \cdot \nabla_V p \\
& = \dot{X} \cdot \nabla_X ( r  ( v_\varphi + A^0_\varphi (r, z)) ) + \dot{V} \cdot \nabla_V ( r  ( v_\varphi  + A^0_\varphi (r, z)) )      \\
& =  \dot{X} \cdot \nabla_X (r v_\varphi) +\dot{ X} \cdot \nabla_X (r A^0_\varphi) + \dot{V} \cdot r  \nabla_V  v_\varphi       \\
& = \dot{X} \cdot \nabla_X (r v_\varphi) + \dot{X} \cdot A^0_\varphi e_r + \dot{ X} \cdot r \frac{\partial A^0_\varphi}{\partial r} e_r + \dot{X} \cdot r \frac{\partial A^0_\varphi}{\partial z} e_z + r ( \textbf{E}^0 +  \hat{V} \times \textbf{B}^0 ) \cdot e_\varphi    \\
& = \dot{X} \cdot \nabla_X (r v_\varphi) + \dot{X} \cdot A^0_\varphi e_r + \dot{ X} \cdot r \frac{\partial A^0_\varphi}{\partial r} e_r + \dot{X} \cdot r \frac{\partial A^0_\varphi}{\partial z} e_z \\
& + r ( -\nabla\phi^0 + \dot{X} \times \frac{\partial A^0_\varphi}{\partial z} e_r + \dot{X} \times \frac{1}{r} A^0_\varphi e_z +\dot{X} \times \frac{\partial A^0_\varphi}{\partial r} e_z    ) \cdot e_\varphi      \\
& = \dot{X} \cdot \nabla_X (r v_\varphi) + \dot{X} \cdot A^0_\varphi e_r + \dot{ X} \cdot r \frac{\partial A^0_\varphi}{\partial r} e_r + \dot{X} \cdot r \frac{\partial A^0_\varphi}{\partial z} e_z - r  \frac{\partial A^0_\varphi}{\partial z} \dot{X} \cdot  e_z - A^0_\varphi   \dot{X} \cdot e_r - r \frac{\partial A^0_\varphi}{\partial r}  \dot{X} \cdot e_r      \\
& =\dot{X} \cdot \nabla_X (r v_\varphi)  \ . \\
\end{split}
\end{equation}
(Here $r$, $\varphi$, $z$, $v_r$, $v_\varphi$, $v_z$ are components of $X$ and $V$.) In the computation above, we used that $\nabla \phi^0$ does not have $e_\varphi$-component. We evaluate $\dot{X} \cdot \nabla_X (r v_\varphi) $ using the Cartesian coordinates:
\begin{equation}
\begin{split}
\dot{X} \cdot \nabla_X (r v_\varphi) 
& = \dot{X} \cdot \nabla_X ( r V \cdot ( -\frac{X_2}{r} , \frac{X_1}{r} , 0 )) \\
& = \dot{X} \cdot \nabla_X (- X_2 V_1 + X_1 V_2) \\
& = \frac{1}{\la V \ra} (V_1, V_2, V_3 ) \cdot (V_2, -V_1, 0) \ . \\
& = 0
\end{split}
\end{equation}
Hence $\dot{p} =0$. The invariance along the particle trajectories is verified.

\section{Appendix C}

In this appendix, we show that for almost every particle in $\Omega \times \mathbb{R}^3$, the corresponding trajectory hits the spatial boundary at most a finite number of times in each finite time interval. We will follow the proof in Appendix D in \cite{NS1}. For simplicity, we drop the $\pm$ symbols (since the two cases are completely analogous) and denote the particle trajectories by 
\begin{equation*}
\Phi_s (x,v) : = (X(s;x,v), V(s;x,v)) 
\end{equation*}  
for each $(x, v) \in \bar{\Omega} \times \mathbb{R}^3$ and $s \in \mathbb{R}$. For each $s$, as long as the map $\Phi_s (\cdot )$ is well-defined, the Jacobian determinant is time-independent and therefore equal to $1$. Thus we obtain that $\Phi_s (\cdot )$ is measure-preserving in $\bar{\Omega} \times \mathbb{R}^3$ with respect to the Lebesgue measure $\mu$. Denote by $\sigma$ the induced surface measure on $\partial \Omega \times \mathbb{R}^3$. 

For each $e_0 \geq 0$, let $\Sigma_0 : = \{ (x,v ) \in \partial \Omega \times \mathbb{R}^3 : |e(x,v)| \leq e_0 , \ v_n \leq 0 \}$. Due to the fact that $\textbf{E}(t, x)$ and $\textbf{B}(t,x)$ are bounded functions, $\Sigma_0$ is a bounded set in $\mathbb{R}^6$ (since $|e(x,v)| \leq e_0 $ gives the boundedness of $|v|$). In particular, $\sigma (\Sigma_0) < \infty$. Now for any $(x,v ) \in \Sigma_0$, the particle trajectory starting from $(x, v)$ at $s=0$ can be continued by the ODE \eqref{particletrajectoryODE} for a certain time. We denote by $\alpha (x, v)$ the first positive time when the trajectory meets the boundary, i.e. $X(\alpha(x,v); x,v ) \in \partial \Omega$. Moreover, we introduce
\begin{equation*}
\Psi (x, v) := \bar{\Phi}_{\alpha(x, v)} (x, v)  , \ where \
 \bar{\Phi}_{s} (x, v) : = (X, V_*) (s; x, v) \ . 
\end{equation*}
Here $V_*$ is given by \eqref{reflectedvelocity}. $\Psi (x, v)$ describes the moment when, at the first positive time of the particle hitting the boundary, it gets reflected specularly with the same speed as just before the collision. We can define $\Psi^{-1} (x, v)$ according to the reversibility of the particle trajectories, and define $\Psi^0$ as the identity map. The set of the particles that hits the boundary with $v_n =0$ (meaning they touch $\partial \Omega$ grazingly) has measure zero. To exclude these particles,  we denote 
\begin{equation*}
\Sigma_1 : = \Sigma_0 \setminus \cup_{k\geq 0} \Psi^{-k} (\{ v_n =0 \}) \ . 
\end{equation*}  
Let $s_N = \sum_{k =0}^N \alpha (\Psi^k (x, v))$ be the time of the $N$-th collision. By the time-reversibility of the particle trajectories, the set $\Psi^{-N} (\{ v_n =0\}) = \Psi_{-s_N} (\{ v_n =0\}) $ has $\sigma$-measure zero for every $N \geq 0$, since $\{ v_n =0\}$ has measure zero. Thus, $\Psi$ is a well-defined map from $\Sigma_1$ to $\Sigma_1$. 

Since $\Phi_s (\cdot) $ is $\mu$-measure preserving, so is $\Psi (\cdot)$ with respect to the surface measure $\sigma$. Now, we define
\begin{equation*}
Z : = \{ (x, v) \in \Sigma_1 : \sum_{k \geq 0} \alpha (\Psi^k (x,v)) < \infty \} \ ,
\end{equation*}     
i.e. $Z$ is the set of particles whose trajectories bounce off $\partial \Omega$ infinitely many times within a finite time interval. We now claim $\sigma (Z) =0$. It suffices to prove that   
\begin{equation*}
Z_\epsilon : = \{ (x, v) \in \Sigma_1 : \epsilon < \sum_{k \geq 0} \alpha (\Psi^k (x,v)) < \infty \} 
\end{equation*} 
has $\sigma$-measure zero for arbitrarily small $\epsilon >0$. Indeed, since $\Psi (\cdot)$ is $\sigma$-measure preserving, the Poincar\'{e} recurrence theorem (see \cite{B1}, for example) tells us that either $Z_\epsilon$ has $\sigma$-measure zero, or there exists a point $(x,v ) \in Z_\epsilon$ such that $\Psi^{N_j} (x,v) \in Z_\epsilon$ for some increasing sequence of integers $N_j$ ($j \geq 1$). However, the latter case would yield, by definition of $Z_\epsilon$,
\begin{equation*}
\epsilon < \sum_{l\geq 0} \alpha (\Psi^l \Psi^{N_j} (x, v)) = \sum_{k \geq N_j} (\Psi^k (x, v)) 
\end{equation*}  
holds for each $N_j$. Thus the series $\sum_{k \geq 0} (\Psi^k (x, v)) $ diverges, which implies that $(x,v) \notin Z_\epsilon$ and hence gives a contradiction. Therefore $Z_\epsilon$ must have $\sigma$-measure zero, and the claim is proved. 

Finally, let $Z^*$ be the set of points in $\Omega \times \mathbb{R}^3$ whose trajectories hit $Z \cup ( \cup_{k\geq 0} \Psi^{-k} (\{v_n =0 \}) )$ within a finite interval. Since the electric and the magnetic fields are uniformly bounded, the trajectories in $Z^*$ have bounded lengths. Hence $Z^*$ has $\mu$-measure zero since $Z \cup ( \cup_{k\geq 0} \Psi^{-k} (\{v_n =0 \}) )$ has $\sigma$-measure zero. This proves that for almost $(x,v ) \in \Sigma_0$, the corresponding trajectory touches the boundary at most finitely many times in each finite time interval. Letting $e_0 = 1, \ 2, \ 3, \cdots$ and taking the union of the corresponding $\Sigma_0$'s, we conclude that for almost every particle in $\Omega \times \mathbb{R}^3$, the corresponding trajectory hits $\partial \Omega$ at most a finite number of times in each finite time interval.

\section{Acknowledgement}

The author thanks her advisor, Walter Strauss for all the guidance, encouragement and patience, without which this work would be impossible. Also, she thanks Benoit Pausader, Yan Guo, Justin Holmer and Toan Nguyen for helpful discussions.


\begin{thebibliography}{00}
\bibitem{B1} L. Barreira, \emph{Poincar\'{e} recurrence: Old and New}, 415--422, World Scientific, Singapore, 2005   
\bibitem{BF1} J. Batt and K. Fabian, \emph{Stationary solutions of the relativistic Vlasov-Maxwell system of plasma physics}, Chinese Ann. Math. Ser. B 14 (1993), no. 3, 253–278
\bibitem{CCM1} S. Caprino, G. Cavallaro and C. Marchioro, \emph{Time evolution of a Vlasov-Poisson plasma with magnetic confinement}, Kinetic and Related Models (2012), Vol. 5, Issue 4, 729--742
\bibitem{CCM2} S. Caprino, G. Cavallaro and C. Marchioro, \emph{On a Vlasov-Poisson plasma confined in a torus by a magnetic mirror}, Journal of Math. Anal. and Appl. (2015), Vol. 427, Issue 1, 31--46
\bibitem{Friedberg1} J. P. Friedberg, \emph{Ideal Magnetohydrodynamics}, Plenum Press, New York, 1987
\bibitem{G1} Y. Guo, \emph{Stable magnetic equilibria in collisionless plasmas}, Commun. Pure Appl. Math. (1997) 50 (9): 891--933
\bibitem{GS1} Y. Guo and W. Strauss, \emph{Instability of periodic BGK equilibria} Commun. Pure Appl. Math. (1995) 48 (8): 861--894
\bibitem{GS2} Y. Guo and W. Strauss, \emph{Relativistic unstable periodic BGK waves} Commun. Pure Appl. Math. (1999) 18 (1): 87--122
\bibitem{GS3} Y. Guo and W. Strauss, \emph{Magnetically created instability in a collisionless plasma} J. Math. Pure Appl. (2000) (9) 79 (10): 975--1009
\bibitem{HK1} D. Han-Kwan, \emph{On the confinement of a tokamak plasma}, SIAM J. Math. Anal. (2010) 42 (6): 2337--2367
\bibitem{KL1} C. Kim and D. Lee, \emph{The Boltzmann equation with specular boundary condition in convex domains}, Commun. Pure Appl. Math. (2018) Vol. 71, Issue 3, 411--504
\bibitem{LS1} Z. Lin and W. Strauss, \emph{Linear stability and instability of relativistic Vlasov-Maxwell systems}, Commun. Pure. Appl. Math. (2007) 60 (5): 724--787
\bibitem{LS3} Z. Lin and W. Strauss, \emph{Nonlinear stability and instability of relativistic Vlasov-Maxwell systems}, Commun. Pure. Appl. Math. (2007) 60: 789--837
\bibitem{LS2} Z. Lin and W. Strauss, \emph{A sharp stability criterion for the Vlasov-Maxwell system}, Invent. Math. (2008) 173 (3): 497--546
\bibitem{NS2} T. Nguyen and W. Strauss, \emph{Stability analysis of collisionless plasmas with specularly reflecting boundary}, SIAM J. Math. Anal. (2013) 45 (2): 777--808
\bibitem{NS1} T. Nguyen and W. Strauss, \emph{Linear stability analysis of a hot plasma in a solid torus}, Arch. Rational. Mech. Anal. (2014) 211: 619--672
\bibitem{Nicholson1} D. R. Nicholson, \emph{Introduction to plasma theory}, Wiley, New York, 1983
\bibitem{Penrose1} O. Penrose, \emph{Electrostatic instability of a non-Maxwellian plasma}, Phys. Fluids (1960) 3: 258--265 
\end{thebibliography}
\end{document}